\documentclass{article}
\usepackage{cite}
\usepackage{amsmath}
\usepackage{amssymb}
\usepackage[T1]{fontenc}

\usepackage[utf8]{inputenc}
\usepackage{latexsym,amsfonts,amscd, amsthm}
\usepackage{epsfig}
\usepackage{changebar}
\usepackage{tikz}
\usepackage{cite}
\usetikzlibrary{arrows,backgrounds,snakes}

\theoremstyle{plain}
\newtheorem{thm}{Theorem}[section]

\newtheorem{cor}[thm]{Corollary}
\newtheorem{lem}[thm]{Lemma}

\newtheorem{defn}[thm]{Definition}
\newtheorem{rem}[thm]{Remark}

\newtheorem{prop}[thm]{Proposition}

\newcommand{\beq}{\begin{equation}}
\newcommand{\eeq}{\end{equation}}
\newcommand{\beqa}{\begin{eqnarray}}
\newcommand{\eeqa}{\end{eqnarray}}
\newcommand{\bit}{\begin{itemize}}
\newcommand{\bta}{{\bar \tau}}
\newcommand{\eit}{\end{itemize}}
\newcommand{\bedef}{\begin{defn}}
\newcommand{\edefn}{\end{defn}}
\newcommand{\bpro}{\begin{prop}}
\newcommand{\epro}{\end{prop}}

\newcommand{\bF}{{\bf F}}

\newcommand{\tA}{\tilde{A}}

\newcommand{\bu}{{\bf u}}

\newcommand{\M}{\mathcal{M}}

\newcommand{\h}{\hat}

\newcommand{\bU}{{\bf U}}
\newcommand{\bV}{{\bf V}}

\newcommand{\bG}{{\bf G}}
\newcommand{\bff}{{\bf f}}
\newcommand{\bg}{{\bf g}}
\newcommand{\avg}[1]{\left\langle#1\right\rangle}
\newcommand{\be}{{\bf e}}
\newcommand{\R}{\mathbb{R}}

\newcommand{\bT}{{\bf T}}

\newcommand{\tb}{\tilde{b}}
\newcommand{\ta}{\tilde{a}}

\newcommand\QQ[1]{{#1}}

\usepackage[T1]{fontenc} % imposta la codifica dei font
\usepackage[utf8]{inputenc} 		% lettere accentate da tastiera
\newcommand{\ds}{\displaystyle}
\newcommand{\RR}{\mathbb{R}}
\usepackage{amsmath} 
\usepackage{amssymb}
\usepackage{graphicx}
\usepackage{appendix}
\usepackage{authblk}
\begin{document}

\title{
Asymptotic Analysis of High Order IMEX-RK Methods for ES-BGK Model at Navier-Stokes level}

\author[1]{Sebastiano Boscarino \thanks{Email: \texttt{boscarino@dmi.unict.it}}}
\author[2]{Seung Yeon Cho \thanks{Email: \texttt{chosy89@gnu.ac.kr}}}

\affil[1]{Department of Mathematics and Computer Science, University of Catania, Catania, Italy}
\affil[2]{Department of Mathematics, Gyeongsang National University, Jinju, Korea}

\date{}
\maketitle
\thispagestyle{empty}
\begin{abstract}
Implicit-explicit Runge-Kutta (IMEX-RK) time discretization methods are very popular when solving stiff kinetic equations. In \cite{Hu}, an asymptotic analysis shows that a specific class of high-order IMEX-RK schemes can accurately capture the Navier-Stokes limit without needing to resolve the small scales dictated by the Knudsen number. In this work, we extend the asymptotic analysis to general IMEX-RK schemes, known in literature as Type I and Type II. We further introduce some IMEX-RK methods developed in \cite{Boscarino2017} to attain uniform accuracy in the wide range of Knudsen numbers. Several numerical examples are presented to verify the validity of the obtained theoretical results and the effectiveness of the methods.
\end{abstract}

{\bf keywords.} Stiff kinetic equations, BGK/ES-BGK models, IMEX Runge–Kutta methods, Compressible Euler equations, Compressible Navier–Stokes equations

\section{Introduction}
\noindent One of the most well-known kinetic models for rarefied gas dynamics is the Boltzmann transport equation (BTE). Its dimensionless form is written as
\begin{equation}\label{Beq}
	\partial_t f + v\cdot \nabla_x f = 	\frac{1}{\varepsilon}Q(f,f),
\end{equation}
where $f(t,x,v)$ is the distribution function which depends on time $t > 0$, on the position of particles {$x=(x_1,\cdots,x_{d_x}) \in \RR^{d_x}$} and on their velocity {$v=(v_1,\cdots,v_{d_v}) \in \RR^{d_v}$}. Here $\varepsilon$ is the so-called Knudsen number, defined as the ratio of the mean free path of molecules and the characteristic length scale of the physical problem. The dimension of space and velocity domains are denoted by $d_x$ and $d_v$, respectively.

The Boltzmann collision operator $Q(f,f)$ is a non-linear operator that describes the binary collisions between molecules. It acts only on the velocity dependence of the distribution function $f$, and has the following fundamental properties 
\begin{itemize}
	\item of conserving mass momentum and energy:
	\begin{equation}\label{rel1}
		\left \langle Q(f,f)\phi(v) \right \rangle = \textbf{0} \in \RR^{d_v+2},
	\end{equation}
	for $\phi(v) = \left(1, v, \frac{1}{2}|v|^2\right)^\top$, and $\left \langle g \right \rangle := \int_{\RR^{d_v}}g(v) dv$,
	\item to satisfy $H$-theorem:
	\begin{equation}\label{rel2}
		\int_{\RR^{d_v}}Q(f,f) \ln f dv \le 0,
	\end{equation}
	\item to vanish, i.e., $Q(f,f)=0$ when $f$ is the local Maxwellian:
	\begin{equation}\label{Max}
		\mathcal{M}[f](t,x,v) = \frac{\rho(t,x)}{(2 \pi T(t,x))^{d_v/2}}\exp\left(-\frac{|v-u(t,x)|^2}{2T(t,x)}\right),
	\end{equation}
	where $\rho$, {$u=(u_1,\cdots, u_{d_v})$}, and $E$ are density, mean velocity and energy associated to $f$:
	\begin{equation}
		\left \langle f \phi(v) \right \rangle  = \left( \rho, \rho u, E\right)^\top = U\in \RR^{d_v+2},
	\end{equation}
	and temperature $T$ is given by $\displaystyle \frac{d_v \rho T}{2}=E- \frac{1}{2}\rho |u|^2$.
	We introduce the vector $U$ for later use.
\end{itemize}

When the Knudsen number is small, it is well known that BTE is closely related to fluid models such as compressible Euler or compressible Navier-Stokes (CNS) equations. The form of these fluid models associated to BTE are traditionally derived using the perturbation techniques like the Hilbert or Chapman-Enskog expansions \cite{Golse, Cerc, CH}. 
Thus, BTE can be used for various Knudsen number, i.e., from rarefied to continuum gas dynamics.

In spite of its good predictability and close relationship with fluid models, computation of the Boltzmann collision operator is very expensive due to its high dimensionality. Furthermore, the problem becomes more severe as the Knudsen number gets closer to zero (fluid regime). In this case, solving BTE by a standard explicit numerical scheme requires the use of a time step of the order of $\varepsilon$, which leads to very expensive numerical computations. Even if one adopts an implicit or semi-implicit time discretization for the collision part, it is still numerically challenging to construct an efficient implicit solver due to the complicate structure of the Boltzmann collision operator. 

To circumvent the issue on computational cost of the Boltzmann collision operator $Q(f,f)$ in (\ref{Beq}), simpler kinetic models have been proposed to mimic the main properties of the full integral operator $Q(f,f)$. One such model is the BGK model \cite{BGK}:
\begin{equation}\label{BGKo}
	\frac{\partial f}{\partial t} + v \cdot \nabla_x f = {\frac{\tau}{\varepsilon}}\left(\mathcal{M}[f] -f \right),
\end{equation}
where $\tau$ is the collision frequency that depends on $\rho$ and $T$. The BGK model replaces the Boltzmann collision operator with a simple relaxation toward the local Maxwellian. Note that the BGK model still maintains the conservation of mass, momentum, and energy, as well as the entropy inequality \cite{Cerc}. In addition, the BGK model describes the correct fluid limit as $\varepsilon \to 0$, i.e., at the \textit{leading order term} $\varepsilon=0$, it yields the compressible Euler equations (see \cite{Cerc, Golse}). Unfortunately, at the first-order correction in $\varepsilon$, the transport coefficients obtained at the Navier--Stokes level are not satisfactory. In particular, the Prandtl number defined by
$\displaystyle \textrm{Pr} = \frac{\gamma}{\gamma -1}\frac{\mu}{\kappa},$
which relates the viscosity $\mu$ to the heat conductivity $\kappa$ of gases is fixed by $1$. Here, the polytropic constant $\gamma$ for monatomic molecule with transional motions is given by $\displaystyle\gamma=\frac{d_v+2}{d_v}$. Note that for most realistic gases, we have {\textrm{Pr}}$< 1$. In addition, in the hard-sphere model for monoatomic gases ($\gamma = 5/3$) in Boltzmann equation, its Prandtl number is very close to $2/3$. 

Many variants of the BGK model have been proposed in order to give the correct transport {coefficients} at the Navier--Stokes level. In \cite{Ho}, Holway  proposed a model that possesses several desirable properties. It not only satisfies the correct conservation laws and the entropy condition, but also yields the Navier-Stokes equations with a Prandtl number less than one through the Chapman-Enskog expansion. This model is known as the ellipsoidal statistical model (ES-BGK) and reads:
\begin{align}\label{ES-BGK}
	\begin{split}
		&\frac{\partial f}{\partial t} + v \cdot \nabla_x f = \frac{\tau}{\varepsilon}(\mathcal{G}[f] - f),
	\end{split}
\end{align}
where the collision frequency $\tau =\frac{\rho T}{\mu(1-\nu)}$ depends on the free parameter $-\frac{1}{2} \le \nu < 1$.
The anisotropic Gaussian distribution $\mathcal{G}[f]$ is defined by
\begin{equation}\label{Gauss}
	\mathcal{G}[f](t,x,v) = \frac{\rho(t,x)}{\sqrt{\textrm{det}(2\pi \mathcal{T}(t,x))}}\exp\left(-\frac{(v-u(t,x))^\top\mathcal{T}(t,x)^{-1}(v-u(t,x))}{2}\right).
\end{equation}
The temperature tensor $\mathcal{T}(t,x)$ and stress tensor $\Theta(t,x)$ are defined as
\begin{equation}\label{TTT}
	\mathcal{T}(t,x) = (1-\nu)T(t,x) Id + \nu \Theta(t,x),
\end{equation}
\begin{equation}\label{rhoTheta}
	\Theta(t,x) = \frac{1}{\rho}\int_{R^{d_v}}(v-u) \otimes (v-u) f(t,x,v)dv = \frac{1}{\rho}\left \langle  (v-u) \otimes (v-u) f(v) \right \rangle,
\end{equation}
respectively, {where $Id$ is the $d_v\times d_v$ identity matrix and $(v-u) \otimes (v-u)$ denotes the tensor product\footnote{{The tensor product of two vectors $a=(a_1,...,a_d),b=(b_1,...,b_d) \in \mathbb{R}^{d}$ means the outer product of $u$ and $v$, i.e., $a\otimes b\in\mathbb{R}^{d\times d}$ with $(a\otimes b)_{ij}=a_ib_j$.}} of two vectors $(v-u)$ and $(v-u)$.}
%and $Id$ is the $d_v\times d_v$ identity matrix.

As emphasized before, ES-BGK model also has a close relationship with fluid models. 
The Chapman-Enskog expansion applied to ES-BGK model gives the compressible Euler equations \cite{DiPa} at the {\em leading order term}:
\begin{equation}\label{Eq:Euler}
	\partial_t \left(
	\begin{array}{l}
		\rho\\
		\rho u \\
		E 
	\end{array}
	\right)
	+
	\nabla_x \cdot \left( 
	\begin{array}{l}
		\rho u \\
		\rho u \otimes u + p Id\\ 
		(E + p ) u
	\end{array}\right) = 0,
\end{equation}
where $p=\rho T$. While at the first-order correction in $\varepsilon$, we obtain the CNS equations \cite{Cerc} (for details see Appendix \ref{App1_bis}):
\begin{align}\label{NSeq}
	\partial_t \left(
	\begin{array}{l}
		\rho\\
		\rho u \\
		E 
	\end{array}
	\right)
	+
	\nabla_x\cdot \left( 
	\begin{array}{l}
		\rho u \\
		\rho u \otimes u + p Id\\ 
		(E + p ) u
	\end{array}\right)=\left( 
	\begin{array}{l}
		0 \\
		\varepsilon \nabla_x\cdot  (\mu\sigma(u))\\ 
		\varepsilon  \nabla_x\cdot( \mu\sigma(u)u + q)
	\end{array}\right)
\end{align}
where the stress tensor $\sigma(u)$ and heat flux $q$ are given by 
$$\sigma(u) = \nabla_x u + (\nabla_x u)^\top - \frac{2}{d_v} \nabla_x \cdot uId,\quad
q = \kappa \nabla_x T,
$$
with viscosity $\mu = \frac{p}{(1-\nu)\tau}$
and thermal conductivity
$
\kappa = \frac{d_v + 2}{2}\frac{p}{\tau}.
$
Thus, the Prandtl number is given by 
$$
\frac{2}{3} \leq Pr = \frac{d_v+2}{2}\frac{\mu}{\kappa} = \frac{1}{1 - \nu} <  \infty,
$$
and this implies that desired such number can be recovered by choosing $\nu$ appropriately. Note that the equation for the total energy can be replaced by the equation for the temperature $T$ \cite{BaGoLe91} {as follows:}
\begin{equation}\label{Temperature}
	\frac{d_v}{2}\rho \left( \partial_t T + u \cdot \nabla_x T \right) + \rho T \nabla_x \cdot u = \mathcal{O}(\varepsilon).
\end{equation}

Considering the relationship with fluid models, when developing numerical methods for ES-BGK model, such methods should have a correct asymptotic behavior, i.e., for small parameter $\varepsilon$, the schemes should degenerate into a good approximation of the fluid asymptotic (compressible Euler or CNS equations). 

In the fluid dynamic limit, i.e., in the case of the limiting compressible Euler equations, numerical methods that work effective while keeping the mesh size and time step fixed as the Knudsen number approaches zero, are referred to as asymptotic preserving (AP) \cite{Jin}.
Additionally, a scheme that not only preserves the correct asymptotic behavior but also maintains high accuracy in the fluid dynamic regime is called asymptotically accurate (AA) \cite{IMEX_book}. 
In other words, a scheme that satisfies both the AP and AA properties exhibits robustness and high accuracy in the fluid dynamic regime, without resolving the small Knudsen number.

{On the construction of AP and AA methods for kinetic models, IMEX time discretization \cite{ARS, CK, Pa-Ru, Boscarino10, Boscarino2007, Boscarino09} have been successfully applied and studied. }We {focus} our literature reviews to the works on the NS limit based on IMEX time discretization applied to BGK or ES-BGK model. 
{There is, of course, a considerable amount of literature on the use of IMEX-RK methods for BGK-type equations, (see, e.g.,\cite{PareschiDimarco, DiPa,  IMEX_book}). For instance, in \cite{Mieussens,TQ}, the authors considered a micro-macro decomposition of the BGK equation and then applied IMEX-RK schemes to the resulting coupled system. In \cite{JinFilbet}, authors introduced a first order IMEX method for ES-BGK model %The collision term is treated implicitly, allowing $\Delta t$ to be chosen independently of $\varepsilon$, while the convection part is treated explicitly. Authors prove 
	and showed that it is consistent to first order time discretization of CNS equations. High order IMEX-RK methods of type CK (or type II) \cite{Hu} and IMEX multistep methods \cite{PareschiDimarco} are also similarly applied to ES-BGK model, and the methods are shown to be able to capture the NS limit under suitable conditions. However, reference \cite{Hu} focuses solely on a specific subclass of IMEX schemes (namely, type CK or type II methods, \cite{IMEX_book}).}

{This paper provides a broader and more unified analysis of the various IMEX-RK schemes, specifically type I and type II, extending and improving upon existing results in the literature for the ES-BGK model. In particular, we study their asymptotic behavior and generalize the results presented in \cite{Hu}, showing that both types of schemes are capable of capturing the compressible Navier-Stokes (CNS) limit without resolving the small parameter $\varepsilon$.
} 

{Another novelty of the present work is that  {we show numerically the uniform accuracy of specific type of IMEX-RK methods, originally introduced in \cite{Boscarino2017} when} applied to ES-BGK model for a wide range of Knudsen numbers. These IMEX-RK schemes were originally developed and analyzed for a specific class of hyperbolic relaxation systems. {In the context of the ES-BGK model, we show that, under suitable assumptions on the coefficients of such schemes introduced in \cite{Boscarino2017}, consistency with CNS equations  is ensured (see Theorem \ref{thm 2}). Moreover, additionally  order conditions  allow the schemes to maintain the order of accuracy throughout the entire range of Knudsen numbers, thereby guaranteeing uniform accuracy.} These conditions are sufficient but not necessary. However, the price to pay is that the IMEX-RK schemes proposed in \cite{Boscarino2017} require more stages than those commonly used.}

%Another novelty of the present work is that we numerically demonstrate the uniform accuracy of IMEX-RK methods applied to ES-BGK model using methods in \cite{Boscarino2017} for a wide range of Knudsen numbers. We remark that, in \cite{Boscarino2017} IMEX-RK schemes of type I and II have been developed and analyzed for hyperbolic relaxation systems. The authors carried out an asymptotic expansion of the numerical method up to $\mathcal{O}(\varepsilon)$ and imposed additional order conditions on the IMEX-RK schemes to achieve a consistent and accurate discretization at the diffusion limit. 

The outline of this paper is the following. In Section 2, we explain IMEX-RK methods for ES-BGK model. In Section 3, we perform asymptotic analysis at the Navier-Stokes level. Then in Section 4, we give several numerical tests in order to validate our theoretical findings. Finally, conclusion will be provided.

\section{IMEX-RK schemes for ES-BGK equation}\label{Sec:IMEXRK}
For the time discretization of \eqref{ES-BGK} we consider an implicit-explicit (IMEX) Runge-Kutta (RK) scheme because the convection term
in \eqref{ES-BGK} is not stiff and the collision term is stiff when Knudsen number is small.

An IMEX-RK scheme usually can be represented by the double Butcher tables:
\begin{equation}\label{BT}
	\begin{array}{c|c}
		\tilde{c} & \tilde{A}\\
		\hline \\[-0.3cm]
		& \tilde{b}^\top
	\end{array},
	\quad 
	\begin{array}{c|c}
		c & {A}\\
		\hline \\[-0.3cm]
		& {b}^\top
	\end{array}.
\end{equation}
Here $\tilde{A} = (\tilde{a}_{ij})$ with $\tilde{a}_{ij} = 0$ for $j \ge i$ and $A = ({a}_{ij})$ with ${a}_{ij} = 0$ for $j > i$ are $s \times s$ matrices, which are associated to the explicit and implicit time discretizations, respectively. The coefficients vectors ${\tilde{b}} = (\tilde{b}_1, ..., \tilde{b}_s)^\top$ and  ${{b}} = ({b}_1, ..., {b}_s)^\top$ represent the weights, and the vectors ${\tilde{c}} = (\tilde{c}_1, ..., \tilde{c}_s)^\top$  and  ${{c}} = ({c}_1, ..., {c}_s)^\top$ are the nodes defined as 
\begin{equation}\label{c_relation}
	\tilde{c}_i = \sum_{j = 1}^{i-1} \tilde{a}_{ij}, \quad {c}_i = \sum_{j = 1}^{i} {a}_{ij}, \quad i = 1,...,s. 
\end{equation}

Now we give some preliminary definitions. Based on the structure of matrix $A$ in the implicit table, the IMEX schemes can be classified into the following types \cite{Boscarino2007, Boscarino10, IMEX_book}:
\begin{defn}\label{TypeI}
	An IMEX-RK method is of {\bf type I} if the matrix $A$ is invertible.
\end{defn}
\begin{defn}\label{TypeII}
	An IMEX-RK method is of {\bf type II} if the matrix $A$ can be written as
	\begin{equation}\label{CK2}
		A=\left(\begin{array}{cc}
			0 & 0\\
			a &  \hat{A}
		\end{array}\right),
	\end{equation}
	where $a=(a_{21},..., a_{s1})^\top \in \mathbb{R}^{s-1}$ and the submatrix $\hat{A} \in \mathbb{R}^{(s-1) \times (s-1)}$ is invertible \cite{CK}; in particular if $a = {\bf 0} \in \mathbb{R}^{s-1}$, $b_1 = 0$, the scheme is said of type ARS, \cite{ARS}.
\end{defn}

We note that IMEX-RK schemes of type II are very attractive because they allow some simplifying assumptions that make order conditions easier to treat, therefore permitting the construction of higher order IMEX-RK schemes \cite{Boscarino10, Boscarino09}. On the other hand, schemes of type I are more amenable to a theoretical analysis since the matrix $A$ of the implicit scheme is invertible. Later, we start our analysis with the latter scheme, type I.

In the following, we will also make use of the following representation of the matrix $\tA$ in the explicit part 
\beq
\tA=
\left(
\begin{array}{ll}
	0 & 0  \\
	\ta & \hat{\tA}\end{array}
\right),
\label{CK2bis}\eeq
where $\ta=(\ta_{21},\ldots,\ta_{s 1})^\top\in\R^{s-1}$ and $\hat{\tA}\in\R^{(s-1)\times (s-1)}$. This representation of matrix $\Tilde{A}$ is useful for the analysis of IMEX-RK methods of type II.

Finally, we give the following definition \cite{Boscarino13_1, Boscarino13_2}:
\begin{defn}\label{GSA}
	If $b_i = a_{si}$ for $i = 1, ..., s$, the scheme is said to be {\bf stiffly accurate} (SA) in the implicit tableau. Moreover, if $\tilde{b}_i = \tilde{a}_{si}$ for $i = 1, ..., s$, the scheme is said to be {\bf globally stiffly accurate} (GSA).
\end{defn}
The first condition in Definition \ref{GSA} guarantees that an $A$-stable implicit tableau is $L$-stable\footnote{A Runge-Kutta method is $L$-stable if $\lim_{z \to \infty} R(z) = 0$, where $R(z)$ is the stability function of the method. If the $A$-stable RK-scheme is SA, i.e., $b^\top A^{-1}{\bf e} = 1$, it follows that $R(\infty) = \lim_{z \to \infty} R(z) = 1-b^\top A^{-1}{\bf e} = 0$, \cite{Hairer}. Here ${\bf e}=(1,1,...,1)^\top$ is a vector of length $s$.}. 
The AP property of IMEX-RK schemes are strongly related to the $L$-stability of the implicit part of the scheme.
Finally, if the IMEX-RK scheme is GSA we have $f^{n+1} = f^{(s)}$, i.e., the numerical solution coincides with the last stage of the method. 

Below, we review IMEX-RK methods of type I and type II applied to ES-BGK model \eqref{ES-BGK}.% in \cite{Hu}.

\begin{itemize}
	\item \textbf{IMEX-RK method of type I.} Applying an IMEX-RK method of type I to \eqref{ES-BGK}, we get in vector form
	\begin{align}\label{ES_BGK}
		\begin{split}
			&\bF = f^n {\bf e} -  \Delta t \tilde{A} L(\bF) + \frac{\Delta t}{\varepsilon} A {\bf \bar{\tau}}(\mathcal{G}[\bF] - \bF),\\
			&f^{n+1} = f^{n} - \Delta t \tilde{b}^\top L(\bF) + \frac{\Delta t}{\varepsilon} b^\top {\bf \bar{\tau}}(\mathcal{G}[\bF] - \bF),
		\end{split}
	\end{align}
	where $\bF = (f^{(1)}, f^{(2)},..., f^{(s)})^\top  \in \mathbb{R}^{s}$, $L(\bF) = (L(f^{(1)}),...,L(f^{(s)}))^\top \in \mathbb{R}^{s}$, being $L(f^{(k)}) = v \cdot \nabla_x f^{(k)}$ for all $k =1 ,..., s$, $\mathcal{G}[{\bf F}] := (\mathcal{G}(f^{(1)}),...,\mathcal{G}(f^{(s)}))^\top \in \mathbb{R}^{s}$. {A diagonal matrix $\bar{{\bf \tau}} := diag(\tau^{(1)},...,  \tau^{(s)})$ is defined with the relaxation time  $\tau^{(i)}$ associated to $f^{(i)}$} and ${\bf e}=(1,1,...,1)^\top$ is a  vector of length $s$. 
	
	From the first equation in (\ref{ES_BGK}), we get
	\begin{equation}\label{eps}
		\Delta t {\bar \tau} (\mathcal{G}[\bF] - \bF) = \varepsilon  A^{-1}\left( \bF - f^n{\bf e} + \Delta t \tilde{A} L(\bF)\right)
	\end{equation}
	and inserting in the numerical solution, we obtain
	\begin{equation}\label{GSAeq}
		f^{n+1} = (1- b^\top A^{-1}{\bf e})f^n + b^\top  A^{-1}\bF -\Delta t(\tilde{b}^\top- b^\top A^{-1}\tilde{A}) L(\bF).
	\end{equation}
	Therefore, for IMEX-RK schemes of type I the numerical solution is independent on $\varepsilon$ and we are able to pass to the limit $\varepsilon \to 0$ in \eqref{ES_BGK}.
\item \textbf{IMEX-RK method of type II.}
Regarding the IMEX-RK method of type II applied to \eqref{ES-BGK}, with the notations in  (\ref{CK2}) and (\ref{CK2bis}), the scheme reads in vector form
\begin{align}\label{ES_BGK2_CK}
\begin{split}
	& f^{(1)} = f^n,\\
	&\hat{\bF} = f^n \hat{{\bf e}} -  \Delta t \tilde{a} L({f^n}) - \Delta t \hat{\tilde{A}} L({\hat\bF}) + \Delta t a \frac{\tau^n}{\varepsilon}(\mathcal{G}[f^n] - f^n) +\Delta t \hat{A} \frac{\hat{{\bf \bar{\tau}}}}{\varepsilon}(\mathcal{G}[\hat{\bF}] - \hat{\bF}),\\
	&f^{n+1} = f^{n} -\Delta t \tilde{b}_1 L({f^n}) - \Delta t \hat{\tilde{b}}^\top L(\hat{\bF}) + \Delta t b_1 \frac{\tau^n}{\varepsilon}(\mathcal{G}[f^n] - f^n) + \Delta t \hat{b}^\top \frac{\hat{{\bf \bar{\tau}}}}{\varepsilon}(\mathcal{G}[\bF] - \bF),
\end{split}
\end{align}
{where $\h \bF = (f^{(2)}, ..., f^{(s)})^\top \in \mathbb{R}^{s-1}$, $L(\h\bF) = (L(f^{(2)}),...,L(f^{(s)}))^\top \in \mathbb{R}^{s-1}$, $\h \be=(1,1,...,1)^\top\in \mathbb{R}^{s-1}$, $\hat{\tilde{b}}=(\tilde{b}_2,\tilde{b}_3,...,\tilde{b}_s)^\top$, $\tau^n=\tau^{(1)}$,   $\h \bta := diag(\tau^{(2)},...,\tau^{(s)})$ and $\mathcal{G}[{\bf \h{F}}] := (\mathcal{G}(f^{(2)}),...,\mathcal{G}(f^{(s)}))^\top \in \mathbb{R}^{s-1}$.} 
%$\hat{{\bf e}}=(1,1,...,1)^\top$ a unit vector of length $s-1$.
From the second equation in (\ref{ES_BGK2_CK}) we get 
\begin{equation}
\Delta t \hat{\bar{\tau}}(\mathcal{G}[\hat{\bF}] - \hat{\bF})=\varepsilon
\hat{A}^{-1}\left[\hat\bF-f^{n}\hat{\bf e}+\Delta t \ta
L(f^{n})+\Delta t
\hat{\tA}L(\hat\bF)\right]- \Delta t 
\hat{A}^{-1}a {\tau^n}(\mathcal{G}[f^n] - f^n),
\label{eq:ll1}
\end{equation}
and substituting into the numerical solution in  (\ref{ES_BGK2_CK}), we obtain
\begin{align}\label{numsolII}
\begin{split}
	f^{n+1} &= (1-\hat{b}^\top\hat{A}^{-1}\hat{{\bf e}})f^{n} -\Delta t \left(\tilde{b}_1 - \hat{b}^\top\hat{A}^{-1}\tilde{a}\right) L({f^n}) - \Delta t \left(\hat{\tilde{b}}^\top - \hat{b}^\top\hat{A}^{-1}\hat{\tilde{A}}\right) L(\hat{\bF}) \\
	&+ \Delta t \left(b_1 - \hat{b}^\top\hat{A}^{-1}a\right)\frac{\tau^n}{\varepsilon}(\mathcal{G}[f^n] - f^n) +\hat{b}^\top\hat{A}^{-1} {\hat\bF}.
\end{split}
\end{align}
Unfortunately, here the numerical solution depends on $\varepsilon$. Now, in order to be able to pass to the limit $\varepsilon \to 0$, 
%Therefore, for IMEX-RK schemes of type I the numerical solution is independent on $\varepsilon$ and we are able to pass to the limit $\varepsilon \to 0$ in the numerical solution.
we can either require that the initial conditions are well-prepared\footnote{The initial data for equation \eqref{Beq} are said well-prepared if $f_0(x.v) = \mathcal{M}[f_0](x,y) + g_{\varepsilon}(x,v)$, $\lim_{\varepsilon \to0}g_{\varepsilon}(x,v) = 0$.} \cite{IMEX_book}, or impose that the additional condition $b_1-\hat{b}^\top\hat{A}^{-1}a=0$ has to be satisfied, which allows us to take the limit $\varepsilon \to 0$ in \eqref{ES_BGK2_CK}. Note that the condition  $b_1-\hat{b}^\top\hat{A}^{-1}a=0$, is automatically satisfied if the scheme is SA, i.e., $\hat{b}^\top \hat{A}^{-1}={\bf \hat{e}}^\top_{s-1}$, with $\hat{{\bf e}}^\top_{s-1}=(0,\cdots,0,1)$ vector of length $s-1$. 
%i.e., $f^n = \mathcal{G}[f^n] + g^{\varepsilon}$ and in the limit $\varepsilon \to 0$ we have $f^n = \mathcal{M}[f^n]$.
Alternatively, this condition also holds for IMEX schemes of type ARS, where having $a = 0$ and $b_1 = 0$ condition $b_1-\hat{b}^\top\hat{A}^{-1}a=0$ is also satisfied. 
%Similarly, for type CK IMEX schemes in the limit $\varepsilon \to 0$ the distribution function is projected over the equilibrium, i.e. $f^{n+1} = \mathcal{M}(f^{n+1})$. if the scheme is GSA (see proposition \ref{prop 2} below and for details  \cite{IMEX_book, DiPa}).\\
%However, if a CK-type scheme is not generally GSA but only satisfies the SA condition for a given well-prepared initial condition, a similar analysis as did previously for IMEX-RK schemes of type I, can be carried out to the numerical solution $f^{n+1}$ in the asymptotic limit.

%and $\hat{b}^\top \hat{A}^{-1}\hat{\tilde{A}} = \hat{\tilde{b}}^\top$
%$\hat{b}^\top\hat{A}^{-1} = \hat{\bf e}^\top_{s-1}$, with $\hat{\bf e}^\top_{s-1} = (0,...,0,1)$, vector of length $s-1$, we get $\hat{b}^\top\hat{A}^{-1}\hat{{\bf e}} = 1$, $\hat{b}^\top\hat{A}^{-1}\tilde{a} = \tilde{b}_1$, $\hat{\tilde{b}}^\top = \hat{b}^\top\hat{A}^{-1}\hat{\tilde{A}}$, $\hat{b}^\top\hat{A}^{-1}a = b_1$, and again that the numerical solution $f^{n+1}$ coincides with the last internal stages $f^{(s)}$.
\end{itemize}
In both types, the methods require implicit solver for computing relaxation terms. However, the use of collision invariants allows us to treat the implicit terms explicitly. The details will be provided in Section \ref{Sect:3}.

\section{Asymptotic Properties of the IMEX RK Schemes}\label{Sect:3}

In this section, we discuss in detail the asymptotic properties of the IMEX-RK schemes of type I and II with respect to the Euler and Navier-Stokes limits.  We begin with a brief overview and proof of the results related to the Euler limit, highlighting the preservation of leading-order asymptotic. More details can be found in \cite{DiPa, IMEX_book, JinFilbet,Hu}. Next, we will analyze the asymptotic properties of IMEX-RK schemes of type I and II in the Navier-Stokes limit.

\subsection{Preserving the Euler Limit.}
We start with an arbitrary IMEX-RK scheme written for each stage $k$:
\begin{equation}\label{IMEXcomp}
f^{(k)} = f^{n} - \Delta t\sum_{\ell=1}^{k-1}\tilde{a}_{k\ell} v \cdot \nabla_x f^{(\ell)} + \frac{\Delta t}{\varepsilon}\sum_{\ell=1}^{k}a_{k\ell}\tau^{(\ell)}\left(\mathcal{G}[f^{(\ell)}] - f^{(\ell)}\right), \quad k = 1,...,s
\end{equation}
with numerical solution
\begin{equation}\label{solnum}
f^{n+1} = f^{n} - \Delta t\sum_{k=1}^{s}\tilde{b}_{k} v \cdot \nabla_x f^{(k)} + \frac{\Delta t}{\varepsilon}\sum_{k=1}^{s}{b}_{k}\tau^{(k)}\left(\mathcal{G}[f^{(k)}] - f^{(k)}\right).
\end{equation}
At every $k$th stage of (\ref{IMEXcomp}), the computation of $f^{(k)}$ requires implicit treatment of  $\tau^{(k)}$ and $\mathcal{G}[f^{(k)}]$. This difficulty can be circumvented by approximating $\tau^{(k)}$ and $\mathcal{G}[f^{(k)}]$ explicitly. For this, we rewrite (\ref{IMEXcomp}) as 
\begin{equation}\label{fK}
f^{(k)} = \frac{\varepsilon}{\varepsilon + \tau^{(k)}\Delta t a_{kk}} f_*^{(k)} + \frac{\Delta t \tau^{(k)} a_{kk}}{\varepsilon + \tau^{(k)}\Delta t a_{kk}}\mathcal{G}[f^{(k)}] 
\end{equation}
with
$$
f_*^{(k)} = f^{n} - \Delta t\sum_{\ell=1}^{k-1}\tilde{a}_{k\ell} v \cdot \nabla_x f^{(\ell)} + \frac{\Delta t}{\varepsilon}\sum_{\ell=1}^{k-1}a_{k\ell} \tau^{(\ell)}\left(\mathcal{G}[f^{(\ell)}] - f^{(\ell)}\right).
$$
Taking the moments $\langle \cdot, \phi \rangle := \int_{\mathbb{R}^{d_v}} \cdot\, \phi(v) dv$ with 
$
\phi(v) := \left( 1, v, |v|^2/2\right)^\top 
$
on both sides of (\ref{fK}), and using the conservation properties (\ref{G__U}), the implicit part is canceled, i.e.
$$
\biggl\langle(\mathcal{G}[f^{(k)}] -f^{(k)})\phi \biggr\rangle = 0.
$$
Therefore, one obtains the macroscopic quantities $U:= (\rho, \rho u, E)$ at every stage $k$:
\begin{align}\label{LimitEq_1}
\begin{split}
U^{(k)} &= \int_{\mathbb{R}^{d_v}}   \left(f^{n} - \Delta t\sum_{\ell=1}^{k-1}\tilde{a}_{k\ell} v \cdot \nabla_x f^{(\ell)} \right) \phi(v) dv = \biggl\langle  f^n \phi \biggr\rangle - \Delta t\sum_{\ell=1}^{k-1}\tilde{a}_{k\ell} \biggl\langle  v \cdot \nabla_x f^{(\ell)} \phi \biggr\rangle,
\end{split}
\end{align}
and numerical solution
\begin{align}\label{LimitEq_2}
\begin{split}
U^{n+1} &= \biggl\langle  f^n \phi \biggr\rangle - \Delta t\sum_{k=1}^{s}\tilde{b}_{k} \biggl\langle  v \cdot \nabla_x f^{(k)} \phi \biggr\rangle.
\end{split}
\end{align}
Note that using (\ref{LimitEq_1}) we can obtain $\rho^{(k)}$, $u^{(k)}$ and $T^{(k)}$ at every stage $k$, which enables us to compute explicitly  $\tau^{(k)}$. 
To evaluate $\mathcal{G}[f^{(k)}]$, however, we need to determine the stress tensor $\Theta^{(k)}$ in (\ref{rhoTheta}). To achieve this, we first define the tensor $\Sigma$ as follows
\begin{equation}\label{SigmaAAA}
\Sigma = \int_{\mathbb{R}^{d_v}} {v} \otimes { v} f d { v} = \rho(\Theta + u \otimes u ),
\end{equation}
and consider the stage values: 
\begin{equation}\label{SigmaAAA_bis}
\Sigma^{(k)} = \int_{\mathbb{R}^{d_v}} {v} \otimes { v} f^{(k)} d { v} = \rho^{(k)}(\Theta^{(k)} + u^{(k)} \otimes u^{(k)} ).
\end{equation}
Note that $\Theta^{(k)}$ can be obtained by computing $\Sigma^{(k)}$.
Next, we use (\ref{TTT}) and \eqref{SigmaAAA} to get
$$
\rho \mathcal{T} = \rho[(1-\nu)T Id + \nu \Theta] = \rho(1-\nu)T Id + \nu \Sigma - \nu \rho u \otimes u,
$$
and combine this with
$$
\int_{\mathbb{R}^{d_v}} v \otimes v \, \mathcal{G}[f](v) dv = \rho(\mathcal{T} + u \otimes u),
$$
%and, %using \eqref{identity},
to obtain
\begin{align}\label{identity}
\int_{\mathbb{R}^{d_v}} v \otimes v \,\left( \mathcal{G}[f] -f\right)dv =  (1-\nu)\left(\rho\left({T}Id  +  u \otimes u\right)\ - \Sigma \right).
\end{align}
Now, we multiply the scheme  (\ref{IMEXcomp}) by $v \otimes v$ and integrate it over $v$, and use the relation \eqref{identity} to derive the equation for $\Sigma^{(k)}$:
%Now multiplying the scheme  (\ref{IMEXcomp}) by $v \otimes v$ and integrating it over $v$, and using \eqref{identity}, %we derive the equation for $\Sigma^{(k)}$, \cite{JinFilbet}:
%We multiply the scheme (\ref{IMEXcomp}) by $v \otimes v$ and integrate it over $v$. Using \eqref{identity},
%and the relation $\mathcal{T}= (1-\nu)T Id + \nu \Theta$ we obtain 
%(\ref{TTT})
%$$
%\rho \mathcal{T} = \rho[(1-\nu)T Id + \nu \Theta] = \rho(1-\nu)T Id + \nu %\Sigma - \nu \rho u \otimes u.
%$$
%%(see \cite{Hu})
%we get for $\Sigma^{(k)}$
\begin{align}\label{sigma_comp}
\Sigma^{(k)}   
=   \frac{\varepsilon }{\varepsilon + (1-\nu)\tau^{(k)}\Delta t a_{kk}}\Sigma^*  + \frac{\Delta t (1-\nu)\tau^{(k)} a_{kk}}{\varepsilon  + (1-\nu)\tau^{(k)}\Delta t a_{kk}} \rho^{(k)}\left({T}^{(k)} Id +  u^{(k)} \otimes u^{(k)}\right),
\end{align}
where
$$
\Sigma^*  = \Sigma^{n} - \Delta t \sum_{\ell = 1}^{k-1} \tilde{a}_{k, \ell} \nabla_x \cdot \biggl\langle v \otimes v \,  v f^{(\ell)} \biggr\rangle  + \frac{\Delta t}{\varepsilon}\sum_{\ell = 1}^{k-1} a_{k\ell}(1-\nu)\tau^{(\ell)}[ \rho^{(\ell)}(T^{(\ell)}Id + u^{(\ell)}\otimes u^{(\ell)}) - \Sigma^{(\ell)}].
$$
Then, the IMEX-RK scheme reads
\begin{align}\label{SCHEME_imp}
\begin{split}
\Sigma^{(k)}  & = \frac{\varepsilon }{\varepsilon + (1-\nu)\tau^{(k)}\Delta t a_{kk}}\Sigma^*  + \frac{\Delta t (1-\nu)\tau^{(k)} a_{kk}}{\varepsilon  + (1-\nu)\tau^{(k)}\Delta t a_{kk}} \rho^{(k)}\left({T}^{(k)} Id +  u^{(k)} \otimes u^{(k)}\right),\\
f^{(k)} & = \frac{\varepsilon}{\varepsilon + \tau^{(k)}\Delta t a_{kk}} f_*^{(k)} + \frac{\Delta t \tau^{(k)} a_{kk}}{\varepsilon + \tau^{(k)}\Delta t a_{kk}}\mathcal{G}[f^{(k)}], 
\end{split}
\end{align}
and
\begin{align}\label{SCHEME_imp_bis}
\begin{split}
\Sigma^{n+1} &= \Sigma^{n} - \Delta t\sum_{k=1}^{s}\tilde{b}_{k} \nabla_x \cdot \biggl\langle v \otimes v \,  v f^{(k)} \biggr\rangle + \frac{\Delta t}{\varepsilon}\sum_{k=1}^{s}{b}_{k}\biggl\langle v \otimes v \,  \left(\tau^{(k)}\left(\mathcal{G}[f^{(k)}] - f^{(k)}\right)\right)\biggr\rangle,\\
f^{n+1} &= f^{n} - \Delta t\sum_{k=1}^{s}\tilde{b}_{k} v \cdot \nabla_x f^{(k)} + \frac{\Delta t}{\varepsilon}\sum_{k=1}^{s}{b}_{k}\tau^{(k)}\left(\mathcal{G}[f^{(k)}] - f^{(k)}\right).
\end{split}
\end{align}
Regarding the leading order limit, i.e., Euler equation, the following results for IMEX-RK schemes of type I and II have been proved in \cite{IMEX_book,Hu}.
\begin{prop}
\label{prop 1} 
Consider the IMEX-RK method (\ref{IMEXcomp})-(\ref{solnum}) of type I.  Then in the limit $\varepsilon \to 0$, for a fixed $\Delta t$, the scheme becomes the explicit RK scheme characterized by the pair $(\tilde{A},\tilde{b})$ applied to the limit Euler equations (\ref{Eq:Euler}). 
\end{prop}

\begin{cor}\label{Cor_typeA}
Furthermore, if the scheme is GSA, then
\begin{equation}\label{limitE}
\lim_{\varepsilon \to 0} f^{n+1} = \lim_{\varepsilon \to 0} \mathcal{M}[f^{n+1}].
\end{equation}
\end{cor}
%%%%%%%
\begin{rem}\label{Rem_TypeI_b}
According to Proposition \ref{prop 1}, the limit scheme is both AP and AA. In other words, as $\varepsilon \to 0$, the scheme remains stable and accurate, precisely matching the explicit tableau of the IMEX-RK method. \\
Corollary \ref{Cor_typeA} claims that an important property of the IMEX schemes of type I is obtained if in the limit $\varepsilon \to 0$ the distribution function is projected over the equilibrium, i.e.{,} $f^{n+1} = \mathcal{M}(f^{n+1})$. From \eqref{GSAeq} it is clear that this
property is achieved if the scheme is GSA  i.e., $b^\top A^{-1}={\bf e}^\top_s$, with ${\bf e}^\top_s = (0,...0,1)$ vector of length $s$, and $b^\top A^{-1}\tilde{A} = \tilde{b}^\top$,  \cite{IMEX_book, DiPa}.\\
Note that GSA property is not essential for type I schemes to result in AP and AA for limit Euler equations, %we arrive at the same conclusions as in Proposition \ref{prop 1}, even when assuming that the scheme %has same weights, $b_i = \tilde{b}_i$, for $i = 1,...,s$,
(see \cite{IMEX_book, Boscarino2007, Boscarino09,  Boscarino10}). 
\end{rem}

%Note that the GSA condition for IMEX-RK methods is particularly useful in the limit case of the Knudsen number, i.e. when $\eps \approx 0$, for preserving the Euler limit, (see Proposition \ref{prop 1} and \ref{prop 2}). \\
%This condition guarantees that the scheme is both Asymptotic Preserving (AP) and Asymptotically Accurate (AA) (for details, see \cite{IMEX_book, DiPa}). 
%%%%%%%%
%In the following proposition, we claim that we do not need to assume that the scheme is GSA for type II to attain the AP and AA properties. 
\begin{prop}
\label{prop 2}
Consider a GSA IMEX-RK method (\ref{IMEXcomp})-(\ref{solnum}) of type II,  then in the limit $\varepsilon \to 0$, for fixed $\Delta t$ and well-prepared initial data, the scheme becomes the explicit RK scheme characterized by the pair $(\tilde{A},\tilde{b})$ applied to the limit Euler equations (\ref{Eq:Euler}).  %Furthermore, \begin{equation}\label{limitE}
%\lim_{\varepsilon \to 0} f^{n+1} = %\lim_{\varepsilon \to 0} \mathcal{M}[f^{n+1}].
%\end{equation}
\end{prop}
\begin{rem}\label{GSAnot}
Since the scheme is GSA, it ensures that the initial value remains consistent at the next time step, i.e., $ f^{n+1} = \mathcal{M}[f^{n+1}]$. Without the assumption of GSA and the consistency of the initial data in Proposition \ref{prop 2},  as discussed for the type II case in Section \ref{Sec:IMEXRK}, IMEX-RK schemes of type ARS ($a = 0$ and $b_1 = 0$), remain asymptotic-preserving (AP) without requiring additional conditions (in (2.23) the quantity $(b_1 - \hat{b}^TA^{-1} a) = 0$), although they are not necessarily asymptotically accurate (AA). However, imposing the additional condition $\tilde{b}_1 = 0$ ensures asymptotic accuracy (see for more details \cite{IMEX_book}).
\end{rem}

\QQ{%Note that 
}
%Note that for IMEX-RK scheme of type II, in the limit as $\varepsilon \to 0$, from (\ref{eq:ll1}) we get
%$$ 
%\hat{\bF} = \mathcal{M}[\hat{\bF}] -\Delta t \hat{A}^{-1} a \tau^n(\mathcal{M}[f^n]-f^n).
%$$
%and assuming that the initial data are {\em well prepared}, i.e,  for $\varepsilon \to 0$, $f^n(x,v) = \mathcal{M}%[f^n]$, 
%\quad \lim_{\varepsilon \to 0} f^0(x,v) = \lim_{\varepsilon \to 0}\mathcal{G}[f^0(x,v)]
%we get $\hat{\bF} = \mathcal{M}[\hat{\bF}]$. Moreover, assuming that the scheme is also GSA from (\ref{numsolII})  we have $f^{n+1} = \mathcal{M}[f^{n+1}]$ and the moments system reads as (\ref{ExS2})-(\ref{ExS3}), i.e., the IMEX RK scheme of type II becomes the explicit RK scheme applied to the limit compressible Euler system (\ref{Eq:Euler}). For a detailed proof of this, see the paper \cite{DiPa,Hu}.

Usually, to guarantee that an IMEX-RK scheme of type I or II remains robust and stable for any value of $\varepsilon$, that is, performs correctly across both the rarefied and fluid regimes, assuming GSA condition may achieve this task, particularly in the fluid regime, without additional order conditions, \cite{IMEX_book, Boscarino2017}. However, as noted in Remark \ref{GSAnot}, we can consider Type ARS schemes that are not GSA, but characterized by some  assumptions on the coefficients of the scheme in order to
guarantee both AP and AA properties.

We conclude this section by considering an explicit RK scheme characterized by $(\tilde{A}, \tilde{b})$ applied to the limit compressible Euler equations (\ref{Eq:Euler}) with $T = p/\rho$, which takes the form:
\begin{align}\label{ExS2}
\begin{split}
\rho^{(k)} &=  \rho^{n} - \Delta t \sum_{\ell = 1}^{k-1} \tilde{a}_{k\ell}\nabla_x \cdot  (\rho^{(\ell)}u^{(\ell)}),\\
(\rho u)^{(k)} &=  (\rho u)^{n} - \Delta t \sum_{\ell = 1}^{k-1} \tilde{a}_{k\ell}\nabla_x \cdot \left((\rho u)^{(\ell)}\otimes u^{(\ell)} + \rho^{(\ell)} T^{(\ell)}Id\right) ,\\
E^{(k)} &=  E^{n} - \Delta t \sum_{\ell = 1}^{k-1}\tilde{a}_{k\ell} \nabla_x  \cdot \left((E^{(\ell)} + \rho^{(\ell)}T^{(\ell)}) u^{(\ell)}\right),
\end{split}
\end{align}
for $k=1,...,s$, with numerical solution
\begin{align}\label{ExS3}
\begin{split}
\rho^{ n+1} &=  \rho^{n} - \Delta t \sum_{k = 1}^{s} \tilde{b}_k \nabla_x \cdot  (\rho^{(k)}u^{(k)})  ,\\
(\rho u)^{n+1} &=  (\rho u)^{n} - \Delta t \sum_{k = 1}^{s} \tilde{b}_k \nabla_x \cdot \left((\rho u)^{(k)}\otimes u^{(k)} + \rho^{(k)} T^{(k)}Id\right)  ,\\
E^{n+1	} &=  E^{n} - \Delta t \sum_{k = 1}^{s} \tilde{b}_k \nabla_x \cdot \left((E^{(k)} + \rho^{(k)}T^{(k)}) u^{(k)}\right).
\end{split}
\end{align}
Note that in (\ref{ExS2})-(\ref{ExS3}), to the leading-order, we can exchange the energy equation with the temperature one (\ref{Temperature}), so we get
\begin{align}\label{Temp}
T^{(k)} = T^n -  \Delta t\sum_{\ell = 1}^{k-1} \tilde{a}_{k\ell}\left( u^{(\ell)}\cdot \nabla_x T^{(\ell)} +  \frac{2}{d_v}T^{(\ell)}\nabla_x \cdot u^{(\ell)} \right),
\end{align}
\begin{align}\label{Temp2}
T^{n+1} = T^n -\Delta t \sum_{k = 1}^{s} \tilde{b}_{k}\left( u^{(k)}\cdot \nabla_x T^{(k)} + \frac{2}{d_v} T^{(k)}\nabla_x \cdot u^{(k)}\right).
\end{align}

%and an expected order reduction phenomenon may occur.\\
%On the contrary of considering schemes that satisfy the GSA condition, we assume that the schemes have same weights, i.e, $\tilde{b}_i = b_i$ for all $i$, in  the fluid-dynamic limit  ($\varepsilon \to 0$)
%similar conclusions regarding the AP and AA properties can be drawn, see Section \ref{Sect:3} and \cite{Boscarino2007, Boscarino09,  Boscarino10}.  

\subsection{Preserving the Navier-Stokes Limit.}
In this section, we analyze the asymptotic behavior of IMEX-RK schemes of type I and II for ES-BGK equations (\ref{ES-BGK}), and prove that for small value of $\varepsilon$, these schemes asymptotically capture the NS limit without resolving $\varepsilon$, providing a numerical scheme for the corresponding CNS equations (\ref{NSeq}). 
%{The two main Theorems below extend the results of \cite{Boscarino2017} to asymptotically Navier-Stokes consistent schemes}.
{The two main Theorems below not only extend the asymptotic analysis of \cite{Hu}  by considering genereal IMEX RK methods, but also result in explicit-type RK schemes for CNS equations which are consistent with the results of \cite{Boscarino2017}.
}

At the Navier-Stokes level, (for small but non-negligible Knudsen numbers, $\varepsilon   \ll 1$),  the GSA condition is not crucial to ensure consistency with CNS equations for both types of IMEX-RK schemes. Therefore, for the subsequent analysis, we do not  impose the GSA condition on the scheme.

{\noindent$\bullet$ {\bf IMEX-RK scheme of type I}. We first analyze the IMEX-RK method of type I. The following shows that the result in \cite{Boscarino2017} can be generalized to ES-BGK model.}
\begin{thm}\label{thm 1}
{For small values of $\varepsilon$ and with $\displaystyle \Delta t^p  + \varepsilon\Delta t +  \frac{\varepsilon^2}{\Delta t}= {o}(\varepsilon)$, the IMEX-RK of type I (\ref{ES_BGK})	asymptotically becomes   a consistent macroscopic explicit-type RK scheme of order $p$ charterized by the pair $(\tilde{A}, \tilde{b})$ and $(B,\omega)$  for the CNS equations (\ref{NSeq}) with
$$
B = \tilde{A} A^{-1} \tilde{A}, \quad   \omega^\top = \tilde{b}^T A^{-1} \tilde{A}
.$$
}
\end{thm}
{Note that the pair $(\tilde{A}, \tilde{b})$ is the explicit table of the IMEX-RK method of type I in \eqref{BT}, and $A$ corresponds to the implicit part.
Moreover, in the theorem we assume that the explicit-type RK method is of order $p$. 
We remark that the order $p$ could be ensured by enforcing additional order conditions derived from the associated two pair of Butcher tableaux. Such conditions are derived in \cite{Boscarino2017} up to order $p=3$.}
\begin{proof}
We start by considering the vector notation of the IMEX-RK scheme (\ref{ES_BGK}).  By the first order discrete Chapman-Enskog expansion in $\varepsilon$ of $f^n$ and $\bF$, we get  
\begin{equation}\label{Exf_1}
f^n = \mathcal{M}[f^n] + \varepsilon f^n_1, \quad \bF = \mathcal{M}[\bF] + \varepsilon {\bff}_1,
\end{equation}
where the vector function $\bff_1$ satisfies the so-called compatibility conditions $ \langle \phi  \bff_1\rangle = 0$ in (\ref{Comp}).
%Now we follow the same idea developed in the Appendix \ref{App1_bis}. 
Inserting this expansions (\ref{Exf_1})  
into (\ref{ES_BGK}), then multiplying by $\phi(v)$ function and integrating on $v$, we get
$$
\langle \phi  {\M}[\bF]  \rangle =   \langle \phi  \mathcal{M}[f^n]\be\rangle - \Delta t {\bf \tilde{A}}\langle\phi  L(\M[\bF])\rangle  - \varepsilon  \Delta t {\bf \tilde{A}}\nabla_x \cdot \langle v \phi \bff_1\rangle,
$$
{where ${\bf \tA}:= \tA \otimes_K I_{d_v+2}\in \mathbb{R}^{(2+d_v)s\times (2+d_v)s}$, $I_{d_v+2}$ is the $(d_v+2)\times (d_v+2)$ identity matrix, and the symbol $\otimes_K$ denotes the Kronecker product\footnote{{We use the symbol $\otimes_K$ to denote the Kronecker product: If $A$ is a $m\times n$ matrix and $B$ is a $p\times q$ matrix, then the Kronecker
		product of $A,B$ is a $mp \times nq$ matrix that can be written in block form as
		$A\otimes_K B= \begin{pmatrix}
			a_{11}B & \cdots &a_{1n}B\\
			\vdots &  &\vdots\\
			a_{m1}B &\cdots &a_{mn}B\\	
		\end{pmatrix}$}}.
Defining the flux vector $\nabla_x \cdot F(U^{(i)}) = \langle \phi  L(\mathcal{M}[F^{(i)}])\rangle$ with $U^{(i)} = \langle \phi  \mathcal{M}[F^{(i)}]\rangle$, it follows
\begin{equation}\label{fin_bis}
	\bU=   \be \otimes_K U^n - \Delta t {\bf\tilde{A}} \nabla_x\cdot  F(\bU)  - \varepsilon \Delta t {\bf \tilde{A}}\nabla_x \cdot \langle  v\phi \bff_1\rangle,
\end{equation}
where  $ \be=(1,1,...,1)^\top\in \mathbb{R}^{s}$,
\begin{align}\label{UFU def}
	\begin{split}
		{\bf U}&=\begin{pmatrix}
			U^{(1)}, U^{(2)},\cdots, U^{(s)} 
		\end{pmatrix}^\top \in \mathbb{R}^{(2+d_v)s},\\ 
		U^{(i)}&=(\rho^{(i)}, \rho^{(i)}u^{(i)}, E^{(i)})^\top\in \mathbb{R}^{2+d_v}\\
		%			{\bf F}(\bU)&=(F(U^{(1)}),F(U^{(2)}),...,F(U^{(s)}))^\top\in \mathbb{R}^{(2+d_v)s},\\
		%			F(U^{(i)}) &= \left(
		%			\begin{array}{c}
			%				\rho^{(i)} u^{(i)}\\
			%				\rho^{(i)} u^{(i)} \otimes u^{(i)} + \rho^{(i)} T^{(i)} Id \\
			%				(E^{(i)} + \rho^{(i)} T^{(i)})u^{(i)}
			%			\end{array}
		%			\right)\in \mathbb{R}^{d_v}\times \mathbb{R}^{d_v\times d_v}\ \times \mathbb{R}^{d_v}\\
		\nabla_x\cdot F(\bU)&=\left(\nabla_x\cdot F(U^{(1)}),\nabla_x\cdot F(U^{(2)}),...,\nabla_x\cdot F(U^{(s)})\right)^\top\in \mathbb{R}^{(2+d_v)s},\\
		\nabla_x\cdot F(U^{(i)}) &= \left(
		\begin{array}{c}
			\nabla_x\cdot (\rho^{(i)} u^{(i)})\\
			\nabla_x\cdot (\rho^{(i)} u^{(i)} \otimes u^{(i)} + \rho^{(i)} T^{(i)} Id) \\
			\nabla_x\cdot ((E^{(i)} + \rho^{(i)} T^{(i)})u^{(i)})
		\end{array}
		\right)\in \mathbb{R}^{2+d_v}\,  (i=1,...,s).
	\end{split}
\end{align}}	
By Lemma \ref{Lemma1} in Appendix \ref{App1_bis}, %(see and (\ref{Hu})), 
we obtain 
\begin{equation}\label{DiscCL}
\bU=   {\be \otimes_K U^n} - \Delta t {\bf\tilde{A}} \nabla_x\cdot  F(\bU)  - \varepsilon \Delta t {\bf\tilde{A}} \nabla_x\cdot {\bf H}(\bU), 
\end{equation}
and, similarly, for the numerical solution we have
\begin{equation}\label{DiscCL_numSol}
U^{n+1}=   U^n - \Delta t {\bf \tilde{b}} \nabla_x\cdot  F(\bU)  - \varepsilon \Delta t {\bf \tilde{b}}  \nabla_x\cdot {\bf H}(\bU),
\end{equation}
{where ${\bf{\tb}}:= \tb^\top \otimes_K I_{d_v+2}$ and 
%	, ${\bf A}:= A \otimes_K I_{d_v+2}$, 	 ${\bf{b}}:= b^\top \otimes I_{d_v+2}$ 
\begin{align*}
	%		H^{(i)} &= \left(
	%		\begin{array}{c}
		%			{\bf 0}\\
		%			\rho^{(i)} {\Theta_1^{(i)}}\\
		%			{\mathbb{Q}_1^{(i)}} + \rho^{(i)} {\Theta_1^{(i)}} { u^{(i)}}
		%		\end{array}
	%		\right)\in \mathbb{R}^{d_v} \times \mathbb{R}^{d_v\times d_v} \times \mathbb{R}^{d_v}\\
	%		{\bf H}(\bU)&=(H^{(1)},H^{(2)},...,H^{(s)})^\top\in \mathbb{R}^{(2+d_v)s}\\
	\nabla_x\cdot {\bf H}(\bU)&=(\nabla_x\cdot H(U^{(1)}),\nabla_x\cdot H(U^{(2)}),...,\nabla_x\cdot H(U^{(s)}))^\top\in \mathbb{R}^{(2+d_v)s},\\
	\nabla_x\cdot H(U^{(i)}) &= \left(
	\begin{array}{c}
		0\\
		\nabla_x\cdot (\rho^{(i)} \Theta_1^{(i)}) \\
		\nabla_x\cdot ({\mathbb{Q}_1^{(i)}} + \rho^{(i)} {\Theta_1^{(i)}} { u^{(i)}})
	\end{array}
	\right)\in \mathbb{R}^{2+d_v}\,  (i=1,...,s).
\end{align*}}
Now we evaluate the quantities {$\Theta_1^{(i)}$ and $\mathbb{Q}_1^{(i)}$, for $i=1,...,s$}. To achieve this, we expand the anisotropic Gaussian $\mathcal{G}$ with respect to $\varepsilon$, i.e.,
\begin{equation}\label{Gex}
\quad \mathcal{G}[\bF] =  \mathcal{M}[\bF] + \varepsilon \bg,
\end{equation}
with vector $\bg$ defined as (\ref{gVect}). Now inserting (\ref{Exf_1}), (\ref{Gex}) into (\ref{ES_BGK}), we get
\begin{equation}\label{eq:GIMEXv1_bis} 
\M[\bF]  = \displaystyle 	\M[f^n] \be - \Delta t \tA\,L(\M[\bF]) -\varepsilon \left({\bff}_1 - f_1^n\be  + \Delta t \tA\,L(\bff_1)\right) + \Delta t A\bta (\bg-{\bff}_1).
\end{equation}
which implies
\begin{equation}\label{eq:GIMEXv1_5_I} 
\bar{\tau} (\bff_1 -\bg) = -A^{-1}\left(\displaystyle \frac{\M[\bF] - \mathcal{M}[f^n] \be}{\Delta t} + \tA\,L(\M[\bF])\right)  + 
\mathcal{O}\left(\frac{\varepsilon}{\Delta t} \right).
\end{equation} 
%From \cite{Golse, Mieussens}  and (\ref{M_lim}), following the same idea developed for the linear multistep  methods in \cite{PareschiDimarco}, 
%to ensure that the scheme is consistent with the Navier-Stokes limit in time, it is essential that Eq. (\ref{eq:GIMEXv1_5_I}) satisfies a relation analogous to (\ref{M_lim}) at discrete levels.\\
%%%%%%%
% \SB{
% Now we note that from \eqref{M_lim}, and integrating such a problem by an  explicit RK method we get for the internal stages vector:
% %\begin{equation}\label{M_lim}
% %\partial_t \mathcal{M} + v\cdot \nabla_x \mathcal{M} = 
% %\mathcal{M}\left(A(V):\frac{\sigma(u)}{2} + 2B(V)\cdot \nabla_x \sqrt{T}\right) +\mathcal{O}(\varepsilon).
% %\end{equation}
% %%%%%%%
% %Now, we note that 
% \begin{equation}\label{RR_2}
	%  \frac{\M[\bF] - \mathcal{\M}[f^n] \be}{\Delta t} + \tA\,\left(L(\M[\bF]) \right) =\tilde{A}\mathcal{M}[\bF]\left(A(V):\frac{\sigma(\bu)}{2} + 2B(V)\cdot \nabla_x \sqrt{\bT}\right),
	%   \end{equation}
%   with $\bu = (u^{(1)}, ...,u^{(s)})^\top$,  $\bT = (T^{(1)}, ...,T^{(s)})^\top$ and 
% \begin{align}\label{VAB def}
	% V = \frac{v-\bu}{\sqrt{T}}, \quad A(V) = V \otimes V - \frac{1}{d_v}|V|^2 Id, \quad B(V) = \frac{1}{2}V(|V|^2 - (d+2)).
	% \end{align}
% }
Now, applying only the explicit part of IMEX-RK scheme to \eqref{M_lim}, we get
\begin{equation}\label{RR_2_CHO}
\frac{\M[{\bG}] - \mathcal{\M}[f^n] \be}{\Delta t} + \tA\,\left(L(\M[\bG]) \right) =\tilde{A} {diag\left(\mathcal{M}[\bG]\right)}\left(A({\textbf{V}}):\frac{\sigma(\bu)}{2} + 2B({\textbf{V}})\cdot \nabla_x \sqrt{\bT}\right)+\mathcal{O}(\varepsilon),
\end{equation}
{where $\bG$ denotes the vector of internal stage for the equation \eqref{M_lim} and
\begin{align*}
	\begin{split}
		\bu &= (u^{(1)}, ...,u^{(s)})^\top\in \mathbb{R}^{d_vs},  \bT = (T^{(1)}, ...,T^{(s)})^\top\in \mathbb{R}^{s},\\ \textbf{V} &= (V^{(1)},V^{(2)},...,V^{(s)})^\top\in \mathbb{R}^{d_vs} \quad V^{(i)}=\frac{v-u^{(i)}}{\sqrt{T^{(i)}}}\in \mathbb{R}^{d_v},\\
		%			 \\ {\bf \underline{A}}(\textbf{V}) &= \textbf{V} \otimes \textbf{V} - \frac{1}{d_v}|\textbf{V}|^2 Id, \quad B(\textbf{V}) = \frac{1}{2}\textbf{V}(|\textbf{V}|^2 - (d_v+2))),
		&A(\textbf{V}):\frac{\sigma(\bu)}{2} \in \mathbb{R}^{s},\quad \left(A(\textbf{V}):\frac{\sigma(\bu)}{2}\right)_i = \left(V^{(i)} \otimes V^{(i)} - \frac{1}{d_v}|V^{(i)}|^2 Id\right) : \sigma(u^{(i)})\\ 
		&B(\textbf{V})\cdot \nabla_x \sqrt{\bT}\in \mathbb{R}^{s},\quad \left(B(\textbf{V})\cdot \nabla_x \sqrt{\bT}\right)_i = \frac{1}{2}V^{(i)}(|V^{(i)}|^2 - (d_v+2)) \cdot \nabla_x \sqrt{\bT^{(i)}}, \quad i=1,...,s, 
	\end{split}
\end{align*}}
Integrating both sides of \eqref{RR_2_CHO} macroscopic variables are associated to $\bG$ are obtained as follows: 
%\textcolor{red}{\begin{align}
	%		\begin{split}    
		%			\label{DiscCL_g}
		%			\bU_\bG&=   {\be \otimes_K U^n} - \Delta t {\bf \tilde{A}} \nabla_x\cdot F(\bU_\bG)  +\mathcal{O}(\varepsilon \Delta t)\cr
		%			&=   {\be \otimes_K U^n} - \Delta t {\bf \tilde{A}} \nabla_x
		%			\cdot F(\bU)  +\mathcal{O}(\varepsilon \Delta t), 
		%		\end{split}
	%	\end{align}where $\bU_\bG$ is the macroscopic quantities associated to the stage values $\bG$ and \textcolor{magenta}{$\nabla_x\cdot- F(\bU)=
	%		\nabla_x\cdot F(\bU_\bG)+\mathcal{O}(\varepsilon\Delta t)$} follows from the fact that first stage values of $\bF$ and $\bG$ are the same $f^n$. This together with \eqref{DiscCL} gives $\|\bU-\bU_\bG\|_\infty=\mathcal{O}(\varepsilon\Delta t)$, and 
%	Lemma \ref{lem discre} further implies that 
%$\mathcal{M}[\bG]=\mathcal{M}[\bF]+\mathcal{O}(\varepsilon\Delta t)$. }
{\begin{align}
	\begin{split}    
		\label{DiscCL_g}
		\bU_\bG&=   {\be \otimes_K U^n} - \Delta t {\bf \tilde{A}} \nabla_x\cdot F(\bU_\bG)  +\mathcal{O}(\varepsilon \Delta t),
	\end{split}
\end{align}where $\bU_\bG$ is the macroscopic quantities associated to the stage values. Since the first stage values of $\bF$ and $\bG$ are the same $f^n$, $\bU$ and $\bU_G$ are obtained from the same $\bU^n$ and this implies that $\|\bU-\bU_\bG\|_\infty=\mathcal{O}(\varepsilon\Delta t)$. Then, by 
Lemma \ref{lem discre} we have 
$\mathcal{M}[\bG]=\mathcal{M}[\bF]+\mathcal{O}(\varepsilon\Delta t)$. }
%\begin{align*}
%\begin{split}
% \rho^{(k)} &=  \rho^{n} - \Delta t \sum_{\ell = 1}^{k-1} \tilde{a}_{k\ell}\textrm{div}_x (\rho^{(\ell)}u^{(\ell)})+\mathcal{O}(\varepsilon\Delta t)
%,\\
%(\rho u)^{(k)} &=  (\rho u)^{n} - \Delta t \sum_{\ell = 1}^{k-1} \tilde{a}_{k\ell}\textrm{div}_x ((\rho u)^{(\ell)}\otimes u^{(\ell)} + \rho^{(\ell)} T^{(\ell)}I) +\mathcal{O}(\varepsilon\Delta t),\\
%E^{(k)} &=  E^{n} - \Delta t \sum_{\ell = 1}^{k-1}\tilde{a}_{k\ell} \textrm{div}_x ([E^{(\ell)} + \rho^{(\ell)}T^{(\ell)}] u^{(\ell)})+\mathcal{O}(\varepsilon\Delta t),\\
% T^{(k)} &= T^n -  \Delta t\sum_{\ell = 1}^{k-1} \tilde{a}_{k\ell}\left( u^{(\ell)}\cdot \nabla_x T^{(\ell)} +  \frac{2}{d_v}T^{(\ell)}\textrm{div}_x u^{(\ell)} \right)+\mathcal{O}(\varepsilon\Delta t).
%\end{split}
%\end{align*}
Inserting this into \eqref{eq:GIMEXv1_5_I}, we get
\begin{align}\label{substitution thm 1}
\begin{split}
	\bar{\tau} (\bff_1 -\bg) &= -A^{-1}\left(\displaystyle \frac{\M[\bG] - \mathcal{M}[f^n] \be}{\Delta t} + \tA\,L(\M[\bG]) +\mathcal{O}(\varepsilon) +\mathcal{O}(\varepsilon\Delta t)\right)  + 
	\mathcal{O}\left(\frac{\varepsilon}{\Delta t} \right)\\
	&= -A^{-1}\left(\tilde{A}{diag\left(\mathcal{M}[\bG]\right)}\left(A({\textbf{V}}):\frac{\sigma(\bu)}{2} + 2B({\textbf{V}})\cdot \nabla_x \sqrt{\bT}\right)+\mathcal{O}(\varepsilon)\right)  + 
	\mathcal{O}\left(\frac{\varepsilon}{\Delta t} \right)+\mathcal{O}(\varepsilon\Delta t)
\end{split}
\end{align}
where the second line comes from \eqref{RR_2_CHO}. Now, we again use $\mathcal{M}[\bG]=\mathcal{M}[\bF]+\mathcal{O}(\varepsilon\Delta t)$ to derive
\begin{align*}
\bar{\tau} (\bff_1 -\bg)&= -A^{-1}\left(\tilde{A}{diag\left(\mathcal{M}[\bF]\right)}\left(A({\textbf{V}}):\frac{\sigma(\bu)}{2} + 2B({\textbf{V}})\cdot \nabla_x \sqrt{\bT}\right)\right)  + 
\mathcal{O}\left(\frac{\varepsilon}{\Delta t} \right)+\mathcal{O}(\varepsilon)+\mathcal{O}(\varepsilon\Delta t)
\end{align*} 
%Thus, we can rewrite \eqref{eq:GIMEXv1_5_I} as
% \begin{equation*} 
%  \bar{\tau} (\bff_1 -\bg) = -A^{-1}\left(\displaystyle \left(\partial_t \mathcal{M}[F] + L(\mathcal{M}[F])\right)_{t = t^n + \tilde{c}\Delta t} \right)  + 
%  \mathcal{O}\left(\frac{\varepsilon}{\Delta t} \right) + \SB{\mathcal{O}(\Delta t^{{q}})}.
%  \end{equation*} 
Then, we have
\begin{equation}\label{eq:GIMEXv1_4}
\bff_1 =\bg -\bta^{-1}A^{-1} \left( \tilde{A}{diag\left(\mathcal{M}[\bF]\right)}\left(A({\textbf{V}}):\frac{\sigma(\bu)}{2} + 2B({\textbf{V}})\cdot \nabla_x \sqrt{\bT}\right)\right)  + \mathcal{O}\left(\frac{\varepsilon}{\Delta t} \right) + \mathcal{O}(\varepsilon).
\end{equation} %with $\bu = (u^{(1)}, ...,u^{(s)})^\top$,  $\bT = (T^{(1)}, ...,T^{(s)})^\top$ and 

{Multiplying by $v\otimes v$ or $v|v|^2/2$ both side in (\ref{eq:GIMEXv1_4}) and taking the integration over $v$, we get from (\ref{Hu}), (\ref{Gu}) and (\ref{Fu}) that
\begin{align*}%\label{eq:HuSD}
	\begin{split}
		{\rho^{(i)}\Theta_1^{(i)}} &= \nu{\rho^{(i)}\Theta_1^{(i)}}  -   {{ \frac{1}{\tau^{(i)}} \sum_{j=1}^{s}(A^{-1}\tilde{A})_{ij}}} \rho^{(j)} T^{(j)}\sigma(u^{(j)}) +\mathcal{O}\left(\frac{\varepsilon}{\Delta t} \right),\quad i=1,...,s\\
		\mathbb{Q}_1^{(i)} + \rho^{(i)}\Theta_1^{(i)} u^{(i)} &= \nu{\rho^{(i)}\Theta_1^{(i)}}u^{(i)}  -   {{ \frac{1}{\tau^{(i)}} \sum_{j=1}^{s}(A^{-1}\tilde{A})_{ij}}} \left(\rho^{(j)} T^{(j)}\sigma(u^{(j)})u^{(j)} + \frac{d_v+2}{2}\rho^{(j)}T^{(j)}\nabla_x T^{(j)}\right) +\mathcal{O}\left(\frac{\varepsilon}{\Delta t} \right).
	\end{split}
\end{align*}
Thus, the approximations of stress tensor and heat flux are given by 
\begin{align*}
	\rho{\bf \Theta_1} &= \begin{pmatrix}
		\rho^{(1)}{ \Theta_1^{(1)}},
		\rho^{(2)}{ \Theta_1^{(2)}},
		\cdots,
		\rho^{(s)}{ \Theta_1^{(s)}}		
	\end{pmatrix}^\top\\ 
	%&= -\frac{1}{1-\nu}\left(\begin{pmatrix}
	%	\frac{1}{\tau^{(1)}} (A^{-1}\tilde{A})_{11}p^{(1)} & \frac{1}{\tau^{(1)}} (A^{-1}\tilde{A})_{12}p^{(2)} & \cdots & \frac{1}{\tau^{(1)}} (A^{-1}\tilde{A})_{1s}p^{(s)}\\
	%	\frac{1}{\tau^{(2)}} (A^{-1}\tilde{A})_{21}p^{(1)} & \frac{1}{\tau^{(2)}} (A^{-1}\tilde{A})_{22}p^{(2)} & \cdots & \frac{1}{\tau^{(2)}} (A^{-1}\tilde{A})_{2s}p^{(s)}\\
	%	\vdots & \vdots & \ddots & \vdots\\
	%	\frac{1}{\tau^{(s)}} (A^{-1}\tilde{A})_{s1}p^{(1)} & \frac{1}{\tau^{(s)}} (A^{-1}\tilde{A})_{s2}p^{(2)} & \cdots & \frac{1}{\tau^{(s)}} (A^{-1}\tilde{A})_{ss}p^{(s)}\\
	%\end{pmatrix}\otimes_K Id\right)\times \begin{pmatrix}
	%	\sigma(u^{(1)})\\
	%	\sigma(u^{(2)})\\
	%	\vdots\\
	%	\sigma(u^{(s)})
	%\end{pmatrix} +\mathcal{O}\left(\frac{\varepsilon}{\Delta t} \right)\\
	&= -\frac{1}{1-\nu}\left( \bta^{-1}A^{-1}\tilde{A} \, diag(\bar{p}) \otimes_K Id\right)\times \sigma(\bu) +\mathcal{O}\left(\frac{\varepsilon}{\Delta t} \right)\\
	&= -\left( \bar{\mu} \otimes_K Id\right)\times \sigma(\bu) +\mathcal{O}\left(\frac{\varepsilon}{\Delta t} \right)
	%&=- \bar{\mu} \sigma(\bu)+\mathcal{O}\left(\frac{\varepsilon}{\Delta t} \right)\\
\end{align*}
\begin{align*}
	{\bf q}= \mathbb{Q}_1  &= \begin{pmatrix}
		\mathbb{Q}_1^{(1)},
		\mathbb{Q}_1^{(2)},
		\cdots ,
		\mathbb{Q}_1^{(s)}		
	\end{pmatrix}^\top\\ 
	%&= -\frac{d_v+2}{2}\left(\begin{pmatrix}
	%		\frac{1}{\tau^{(1)}} (A^{-1}\tilde{A})_{11}p^{(1)} & \frac{1}{\tau^{(1)}} (A^{-1}\tilde{A})_{12}p^{(2)} & \cdots & \frac{1}{\tau^{(1)}} (A^{-1}\tilde{A})_{1s}p^{(s)}\\
	%		\frac{1}{\tau^{(2)}} (A^{-1}\tilde{A})_{21}p^{(1)} & \frac{1}{\tau^{(2)}} (A^{-1}\tilde{A})_{22}p^{(2)} & \cdots & \frac{1}{\tau^{(2)}} (A^{-1}\tilde{A})_{2s}p^{(s)}\\
	%		\vdots & \vdots & \ddots & \vdots\\
	%		\frac{1}{\tau^{(s)}} (A^{-1}\tilde{A})_{s1}p^{(1)} & \frac{1}{\tau^{(s)}} (A^{-1}\tilde{A})_{s2}p^{(2)} & \cdots & \frac{1}{\tau^{(s)}} (A^{-1}\tilde{A})_{ss}p^{(s)}\\
	%	\end{pmatrix}\otimes_K Id\right)\times \begin{pmatrix}
	%		\nabla_x T^{(1)}\\
	%		\nabla_x T^{(2)}\\
	%		\vdots\\
	%		\nabla_x T^{(s)}
	%	\end{pmatrix} +\mathcal{O}\left(\frac{\varepsilon}{\Delta t} \right)\\
&= -\frac{d_v+2}{2}\left( \bta^{-1}A^{-1}\tilde{A}  \, diag(\bar{p}) \otimes_K Id\right)\times \nabla_x \bT +\mathcal{O}\left(\frac{\varepsilon}{\Delta t} \right)\\
&= -\left(\bar{\kappa} \otimes_K Id\right) \times \nabla_x \bT +  \mathcal{O}\left(\frac{\varepsilon}{\Delta t} \right)
\end{align*}
where $\sigma(u^{(i)}) = \nabla_x u^{(i)} + (\nabla_x u^{(i)})^\top - \frac{2}{d_v}\nabla_x \cdot u^{(i)} Id$, and $\bar{p}=(p^{(1)}, p^{(2)},...,p^{(s)})^\top$ with $p^{(i)}=\rho^{(i)}T^{(i)}$.
Here we defined the viscosity and thermal conductivity $s\times s$ matrices ${\bar \mu}=(\mu_{ij})$ and ${\bar \kappa}=(\kappa_{ij})$ as
\begin{equation}\label{mu_k}
{\bar \mu} = \frac{1}{(1-\nu)}\bta^{-1} (A^{-1}\tilde{A})\,diag(\bar{p}), \quad
{\bar\kappa} =\frac{d_v + 2}{2} \, \bta^{-1} (A^{-1}\tilde{A})\, diag(\bar{p}).
\end{equation}
This implies that
\[
\bta^{-1}(A^{-1}\tilde{A})\,diag(\bar{p})= (1-\nu){\bar \mu}   = 
\frac{2}{d_v + 2}{\bar\kappa} 
\]
and hence
\[
\frac{d_v+2}{2}\frac{\mu_{ij}}{\kappa_{ij}}=\frac{1}{1-\nu},\quad 1\leq i,j\leq s.
\]
This form is consistent to the Prandtl number of the ES-BGK model.
}
To sum up, the scheme reads \begin{equation}\label{DiscCL_bis_one}
\bU=   \be \otimes_K U^n - \Delta t {\bf\tilde{A}} \nabla_x\cdot F(\bU)  + \varepsilon \Delta t {\bf \tilde{A}} \nabla_x\cdot \underline{S}(\bU) + \mathcal{O}(\varepsilon^2), 
\end{equation}
and 
\begin{equation}\label{DiscCL_numSol_bis}
U^{n+1}=   U^n - \Delta t {\bf\tilde{b}} \nabla_x\cdot F(\bU)  + \varepsilon \Delta t {\bf \tilde{b}}  \nabla_x\cdot\underline{S}(\bU)+  \mathcal{O}(\varepsilon^2),
\end{equation}
for the numerical solution, where
{
\begin{align*}%\label{Su}
\begin{split}
	\nabla_x\cdot \underline{S}(\bU) &=\left(\nabla_x\cdot \underline{S}(U^{(1)}), \nabla_x\cdot \underline{S}(U^{(2)}),...,\nabla_x\cdot \underline{S}(U^{(s)})
	\right)^\top \in \mathbb{R}^{(d_v+2)s},\\ 
	\nabla_x\cdot \underline{S}(U^{(i)}) &=\left(
	\begin{array}{c}
		\nabla_x \cdot{\bf 0},\\
		\nabla_x \cdot\left(\sum_{j=1}^{i-1}\bar{\mu}_{ij} \sigma(u^{(j)})\right),\\
		\nabla_x \cdot\left(\sum_{j=1}^{i-1} \bar{\mu}_{ij}\sigma(u^{(j)}) u^{(j)} + {\bar{\kappa}_{ij} \nabla_x T^{(j)}}\right)
	\end{array}
	\right)\\
	&=\left(
	\begin{array}{c}
		\nabla_x \cdot{\bf 0},\\
		\nabla_x \cdot\left(\frac{1}{\tau^{(i)}}\sum_{j=1}^{i-1}\frac{1}{(1-\nu)} (A^{-1}\tilde{A})_{ij}p^{(j)} \sigma(u^{(j)})\right),\\
		\nabla_x \cdot\left(\frac{1}{\tau^{(i)}}\sum_{j=1}^{i-1} \frac{1}{(1-\nu)} (A^{-1}\tilde{A})_{ij}p^{(j)}\sigma(u^{(j)}) u^{(j)} + {\frac{d_v+2}{2} (A^{-1}\tilde{A})_{ij}p^{(j)} \nabla_x T^{(j)}}\right)
	\end{array}
	\right)\in \mathbb{R}^{d_v+2}.
\end{split}
\end{align*}
For each $i,j=1,...,s$, we now consider the approximation $\tau^{(i)}=\tau^{(j)} + \mathcal{O}(\Delta t)+\mathcal{O}(\varepsilon^2)$. Inserting this approximation into the $\underline{S}(\bU)$ in \eqref{DiscCL_bis_one} and \eqref{DiscCL_numSol_bis}, we obtain
\begin{align*}
\begin{split}
	\bU&=   \be \otimes_K U^n - \Delta t {\bf\tilde{A}} \nabla_x\cdot F(\bU)  + \varepsilon \Delta t {\bf \tilde{A}} \nabla_x\cdot \underline{\underline{S}}(\bU) + \mathcal{O}(\varepsilon^2)+  \mathcal{O}(\varepsilon \Delta t^2),\\ 
	U^{n+1}&=   U^n - \Delta t {\bf\tilde{b}} \nabla_x\cdot F(\bU)  + \varepsilon \Delta t {\bf \tilde{b}}  \nabla_x\cdot\underline{\underline{S}}(\bU)+  \mathcal{O}(\varepsilon^2) +  \mathcal{O}(\varepsilon \Delta t^2),
\end{split}
\end{align*}
where
\begin{align*}
\begin{split}
	\nabla_x\cdot \underline{\underline{S}}(\bU) &=\left(\nabla_x\cdot \underline{\underline{S}}(U^{(1)}), \nabla_x\cdot \underline{\underline{S}}(U^{(2)}),...,\nabla_x\cdot \underline{\underline{S}}(U^{(s)})
	\right)^\top \in \mathbb{R}^{(d_v+2)s},\\ 
	\nabla_x\cdot \underline{\underline{S}}(U^{(i)}) 
	&=\left(
	\begin{array}{c}
		\nabla_x \cdot{\bf 0},\\
		\nabla_x \cdot\left(\sum_{j=1}^{s}\frac{1}{(1-\nu)\tau^{(j)}} (A^{-1}\tilde{A})_{ij}p^{(j)} \sigma(u^{(j)})\right),\\
		\nabla_x \cdot\left(\sum_{j=1}^{s} \frac{1}{(1-\nu)\tau^{(j)}} (A^{-1}\tilde{A})_{ij}p^{(j)}\sigma(u^{(j)}) u^{(j)} + {\frac{d_v+2}{2\tau^{(j)}} (A^{-1}\tilde{A})_{ij}p^{(j)} \nabla_x T^{(j)}}\right)
	\end{array}
	\right)\\
	&=\left(
	\begin{array}{c}
		\nabla_x \cdot{\bf 0},\\
		\sum_{j=1}^{s} (A^{-1}\tilde{A})_{ij} \nabla_x \cdot\left(\mu^{(j)} \sigma(u^{(j)})\right),\\
		\sum_{j=1}^{s} (A^{-1}\tilde{A})_{ij} \nabla_x \cdot\left( \mu^{(j)}\sigma(u^{(j)}) u^{(j)} + \kappa^{(j)} \nabla_x T^{(j)}\right)
	\end{array}
	\right)\in \mathbb{R}^{d_v+2}\\
\end{split}
\end{align*}
with $\mu^{(j)} = \frac{1}{(1-\nu)\tau^{(j)}} p^{(j)}$, $\kappa^{(j)}=\frac{d_v+2}{2\tau^{(j)}} p^{(j)}$. Defining
\begin{align*}
\begin{split}
	\nabla_x\cdot S(\bU) &=\left(\nabla_x\cdot S(U^{(1)}), \nabla_x\cdot S(U^{(2)}),...,\nabla_x\cdot S(U^{(s)})
	\right)^\top \in \mathbb{R}^{(d_v+2)s},\\ 
	\nabla_x\cdot S(U^{(i)}) 
	&=\left(
	\begin{array}{c}
		\nabla_x \cdot{\bf 0},\\
		\nabla_x \cdot\left(\mu^{(i)} \sigma(u^{(i)})\right),\\
		\nabla_x \cdot\left( \mu^{(i)}\sigma(u^{(i)}) u^{(i)} + \kappa^{(i)} \nabla_x T^{(i)}\right)
	\end{array}
	\right)\in \mathbb{R}^{d_v+2},
\end{split}
\end{align*}
we finally derive
\begin{align}\label{final typeI}
\begin{split}
	\bU&=   \be \otimes_K U^n - \Delta t {\bf \tilde{A}} \nabla_x\cdot F(\bU)  + \varepsilon \Delta t  {\bf B} \nabla_x\cdot  S(\bU) + \mathcal{O}(\varepsilon^2)+  \mathcal{O}(\varepsilon \Delta t^2),\\ 
	U^{n+1}&=   U^n - \Delta t {\bf \tilde{b}} \nabla_x\cdot F(\bU)  + \varepsilon \Delta t {\boldsymbol \omega} \nabla_x\cdot S(\bU)+  \mathcal{O}(\varepsilon^2) +  \mathcal{O}(\varepsilon \Delta t^2),
\end{split}
\end{align}
where ${\bf B}= (\tilde{A}A^{-1}\tilde{A}) \otimes_K I_{d_v+2}$ and ${\boldsymbol \omega}= (\tilde{b}^\top A^{-1}\tilde{A}) \otimes_K I_{d_v+2}$.
}
%\textcolor{red}{SY: Since $\tau^{(i)}$ appears instead of on $\tau^{(j)}$, our theorem only works when $\tau^{(i)}=\tau^{(j)}+ \text{small error}$ for all $i,j$. Since the operator $S$ includes $A^{-1}\tA$, it seems that our theorem is wrong, but only Corollary 3.7 is true.}
Next, we consider the following explicit-type RK method of order $p$ based on the coefficients matrices $\tilde{A}, B$ and the weights $\tilde{b}, w$ with $
{B} = \tilde{A}A^{-1}\tilde{A},\, \omega^\top =  \tilde{b}^\top A^{-1}\tilde{A}$ applied to the CNS equations (\ref{NSeq}):
%$$
%\partial_t U + \nabla_x\cdot F(U) = \varepsilon \nabla_x\cdot S(U)$$ is given by
%\SB{reads as \eqref{DiscCL_bis_one}-\eqref{DiscCL_numSol_bis}, i.e, and explicit RK scheme applied to a hyperbolic system, with convect term $\nabla_x \cdot F(U)$ and a perturbed diffusive term $\varepsilon\nabla_x \cdot S(U)$. 
%i.e.,}
\begin{align*}
\begin{split}
\bU&=   \be \otimes_K U^n - \Delta t {\bf \tilde{A}} \nabla_x\cdot F(\bU)  + \varepsilon \Delta t {\bf B} \nabla_x\cdot { S}(\bU), \\[2mm]% + \mathcal{O}(\varepsilon \Delta t^{{q} + 1}),
U^{n+1}&=   U^n - \Delta t {\bf \tilde{b}} \nabla_x\cdot {\bf F}(\bU)  + \varepsilon \Delta t {\boldsymbol\omega} \nabla_x\cdot { S}(\bU).
\end{split}
\end{align*}
$$
%L.T.E. = \mathcal{O}(\Delta t^p) + \mathcal{O}(\varepsilon \Delta t^{q+1}) + \mathcal{O}\left(\frac{\varepsilon^2}{\Delta t}\right),
$$
%with $p$ the order of the explicit Runge-Kutta scheme characterized by the pairs  $(\tilde{A}, \tilde{b})$.
%}
{Assuming  that $u(t)$ is the true solution of the CNS}, then the local truncation error is $\mathcal{O}(\Delta t^p)$, i.e., 
\begin{align}\label{temp1}
\begin{split}
\mathcal{O}(\Delta t^p)  &= \frac{u(t_{n+1})-u(t_n)}{\Delta t} + \sum_{i = 1}^s \tilde{b}_i \nabla_x\cdot F(u^{(i)})   - \sum_{i = 1}^s \omega_i  \varepsilon  \nabla_x\cdot S(u^{(i)})
\end{split}
\end{align}
where $p$ is the order of the explicit scheme with 
$$
{u}^{(i)} = u(t_n) - \Delta t\sum_{j = 1}^{{i-1}} {\tilde{a}_{ij}}\nabla_x \cdot F(u^{(j)}) - \Delta t\sum_{j = 1}^{{i-1}} {{B}_{ij}}   \varepsilon\nabla_x\cdot S(u^{(j)}),
$$
where $B=(B_{ij})$.
Now, we consider the local trunction of our method \eqref{final typeI}:
\begin{align}
\begin{split}
L.T.E. &= \frac{u(t_{n+1})-u(t_n)}{\Delta t} + \sum_{i = 1}^s \tilde{b}_i \left( \nabla_x F(u^{(i)}) \right) - \sum_{i = 1}^s \omega_i \left(  \varepsilon  \nabla_x S(u^{(i)})\right) + \mathcal{O}\left(\frac{\varepsilon^2}{\Delta t}\right) + \mathcal{O}\left(\varepsilon\Delta t\right).
\end{split}
\end{align}
where $\omega=(\omega_1,\omega_2,...,\omega_s)^\top$. This combined with \eqref{temp1}, gives
\begin{equation}\label{LTE}
L.T.E.= \mathcal{O}(\Delta t^p) + \mathcal{O}\left(\frac{\varepsilon^2}{\Delta t}\right)+ \mathcal{O}\left(\varepsilon\Delta t\right).
\end{equation}

%Question: As in the case of Hu, I don't have a stage order... If you have any comment please let me know.
\end{proof}

{Theorem \ref{thm 1} implies that, to accurately capture the Navier-Stokes equation, it is necessary that the local truncation error satisfies $L.T.E = o(\varepsilon)$, which holds true if $\displaystyle \Delta t^p + \varepsilon\Delta t +  \frac{\varepsilon^2}{\Delta t}= {o}(\varepsilon)$.}

\begin{cor}\label{tau1}
{Setting $\tau = 1$ in (\ref{ES-BGK}) the macroscopic explicit-type RK scheme has a local truncation error: $L.T.E. = \mathcal{O}(\Delta t^p) +\mathcal{O}(\varepsilon^2/\Delta t).$ }
%\textcolor{red}{Setting $\tau = 1$,  in (\ref{ES-BGK}), the numerical scheme (\ref{DiscCL_bis_one})-(\ref{DiscCL_numSol_bis}) becomes 
%	{\begin{align}\label{MainNS}
%			\begin{split}
	%				\bU&=   \be \otimes_K U^n - \Delta t {\bf \tilde{A}} \nabla_x\cdot F(\bU)  + \varepsilon \Delta t {\bf B} \nabla_x\cdot {\bf S}(\bU), \\[2mm]% + \mathcal{O}(\varepsilon \Delta t^{{q} + 1}),
	%				U^{n+1}&=   U^n - \Delta t {\bf \tilde{b}} \nabla_x {\bf F}(\bU)  + \varepsilon \Delta t {\omega}^\top \nabla_x {\bf S}(\bU), % + \mathcal{O}(\varepsilon \Delta t^{{q} + 1}),
	%			\end{split}
%	\end{align}}
%	which is an \emph{explicit-type} RK method for CNS equation (\ref{NSeq}) based on the coefficients matrices $\tilde{A}, B$ and the weights $\tilde{b}, w$, with
%	$$
%	{\bf B} = (\tilde{A}A^{-1}\tilde{A}) \otimes_K I_{d_v+2}, \quad  \omega^\top =  (\tilde{b}^\top A^{-1}\tilde{A})\otimes_K I_{d_v+2},
%	$$
%	i.e. scheme (\ref{MainNS}) is an explicit Runge-Kutta type method for (\ref{NSeq}) based on two explicit tableau $(\tilde{A},\tilde{b})$ and (${B},{w})$. In this case, viscosity and thermal conductivity ${\bar \mu}$ and ${\bar \kappa}$ are defined by
%	\begin{equation}\label{mu_k}
%		{\bar \mu} = \frac{1}{(1-\nu)}\bar{p}, \quad
%		{\bar\kappa} =\frac{d_v + 2}{2} \bar{p},
%	\end{equation}
%	with $\bar{p}=(p^{(1)}, p^{(2)},...,p^{(s)})^\top$. }
\end{cor}

{Similarly, analogous considerations can be made for type II by introducing a new definition of the matrix $B$ and weights $w$ (for details, see paper \cite{Boscarino2017}). However, in practice, $\tau$ is closely related to the viscosity and heat conductivity of gases, and hence the form of $\tau$ should be set carefully (see for example \cite{Golse, Mieussens, JinFilbet, DiPa}).}\\

{\noindent$\bullet$ {\bf IMEX-RK scheme of type II}. In this section  we analyze the asymptotic behavior of type II IMEX-RK schemes. %a similar result can be obtained as follows. %Note that  this result is similar to the theorem 3.7 in \cite{Hu}. 
As type I scheme, at this stage, we do not need to assume that the scheme satisfies the GSA condition.
%For the proof, see appendix \ref{App2}.
}

\begin{thm}\label{thm 2}
{For small values of $\varepsilon$ and with $\displaystyle \Delta t^p  + \varepsilon\Delta t +  \frac{\varepsilon^2}{\Delta t}= {o}(\varepsilon)$, the IMEX-RK of type II (\ref{ES_BGK2_CK}) satisfying 
\begin{equation}\label{22CK}
\h{\tilde{A}}\h{A}^{-1}a=\tilde{a},\quad \h{\tilde{b}}^\top \h{A}^{-1}a=\tilde{b}_1,
\end{equation} 
with initial data $f^n=\mathcal{M}[f^n] + \varepsilon f_{1}^n$, where
\begin{align}\label{assumption thm 2}
f_1^n = g^n - \frac{1}{\tau^{n}}\mathcal{M}[f^n]\left(A(V^n):\sigma(u^n) + 2B(V^n)\cdot \nabla_x \sqrt{T^n}\right) + \mathcal{O}(\varepsilon),
\end{align} 
$V^n=\frac{v-u^n}{\sqrt{T^n}}$, asymptotically becomes a consistent 
macroscopic explicit-type RK scheme of order $p$ charterized by the pair $(\tilde{A}, \tilde{b})$ and $(B, \omega)$  for the CNS equations (\ref{NSeq}) with
\begin{align}\label{TYPE_II}
\begin{split}
B&=\left(\begin{array}{cc}
	0 & 0\\
	\mathbb{\bf b}_1  &  \hat{B}
\end{array}\right),\quad \omega^\top=(\omega_1,\h{\omega}^\top),\\
\mathbb{\bf b}_1  = {\h{\tilde{A}}\h{A}^{-1}\tilde{a}},\quad \hat{B} &= \h{\tilde{A}}\h{A}^{-1}{\h{\tilde{A}}},\quad \omega_1={\h{\tilde{b}}^\top \h{A}^{-1}\tilde{a}},\quad \h{\omega}^\top = {\h{\tilde{b}}}^\top \h{A}^{-1} \h{\tilde{A}}.
\end{split}
\end{align}
}
\end{thm}
{As in the case of Theorem \ref{thm 1}, we remark that the order $p$ could be ensured by enforcing additional order conditions derived up to order $3$ for type II from the associated two pair of Butcher tableaux in \cite{Boscarino2017}.}
%For small values of $\varepsilon$ with $\displaystyle \Delta t^p+ \varepsilon\Delta t + \frac{\varepsilon^2}{\Delta t} = {o}(\varepsilon)$, and $p$,  the order of the explicit RK scheme, the IMEX-RK scheme of type II satisfying
%\begin{equation}\label{22CK}
%	\h\tA \h A^{-1}({a} - \tilde{a})=0, \quad \h{\tilde{b}}^\top \h A^{-1}({a} - \tilde{a}) = 0,
%\end{equation}
%with initial data $f^n=\mathcal{M}[f^n] + \varepsilon f_{1}^n$, where
%\begin{align}\label{assumption thm 2}
%	f_1^n = g^n - \frac{1}{\tau^{n}}\mathcal{M}[f^n]\left(A(V^n):\sigma(u^n) + 2B(V^n)\cdot \nabla_x \sqrt{T^n}\right) + \mathcal{O}(\varepsilon),
%\end{align} 
%{$V^n=\frac{v-u^n}{\sqrt{T^n}}$, and $A(\cdot)$, $B(\cdot)$} are given by \eqref{VAB def}, asymptotically becomes consistent time discretization of the CNS equations (\ref{NSeq}).
%\section{Appendix 1}\label{App2}
%Here we prove Theorem \ref{thm 2}.
\begin{proof}
The IMEX RK scheme of type II applied to (\ref{ES-BGK}) can be
written in compact form {\eqref{ES_BGK2_CK}} as 
%\begin{subequations}
\begin{align}\label{eq:GIMEXv1CK} 
F^{(1)} = f^n, \quad \h \bF = \displaystyle f^{n} \h \be - \Delta t \ta\,L(f^n) - \Delta t \h\tA\,L(\h\bF)+\frac{\Delta t}{\varepsilon} a\tau^n (\mathcal{G}[f^n] - f^n) + \frac{\Delta t}{\varepsilon} \h A {\h\bta} (\mathcal{G}[\h\bF] - \h\bF),
\end{align}
\begin{equation*}\label{eq:GIMEXv2CK}
f^{n+1} = \displaystyle f^{n} -\Delta t
\tb_1 L(f^n) - \Delta t
\h \tb^{T}L(\h \bF)
+\frac{\Delta t}{\varepsilon}b_1(\mathcal{G}[f^n] - f^n)
+\frac{\Delta t}{\varepsilon}\h b^{T}(\mathcal{G}[\h \bF] - \h \bF),
\end{equation*}
%\end{subequations}
%where $\be = (1, \h \be)^\top$, $\h \be{=(1,1,...,1)}\in \mathbb{R}^{s-1}$, $\bF = (F^{(1)}, \h \bF^\top)$ with {$\h \bF^\top = (f^{(2)}, ..., f^{(s)})^\top \in \mathbb{R}^{s-1}$}  and   $\bta = diag(\tau^{(1)},\tau^{(2)},...,\tau^{(s)})$, with  $\h \bta := diag(\tau^{(2)},...,\tau^{(s)})$, and $\tau^{(1)} = \tau^n$. 
Now inserting expansions 
\begin{equation}\label{Exf_1_CK}
f^n = \mathcal{M}[f^n] + \varepsilon f^n_1, \quad \h \bF = \mathcal{M}[\h \bF] + \varepsilon \h{\bff}_1, \quad \mathcal{G}[f^n] =  \mathcal{M}[f^n] + \varepsilon g^n,\quad \mathcal{G}[\h \bF] =  \mathcal{M}[\h \bF] + \varepsilon \h\bg,
\end{equation}
into (\ref{eq:GIMEXv1CK}), multiplying by $\phi(v)$ function and integrating on $v$, we get
{
\begin{align}\label{fin_bis_CK_tris22}
\begin{split}
U^{(1)} &= U^n,\\[2mm]
\h \bU &=  \hat{\be} \otimes_K  U^n   - \Delta t  \left(\tilde{a} \otimes_K \nabla_x \cdot F(U^n)  +  {\bf \h \tA} \nabla_x\cdot  F(\h \bU)\right)
- \varepsilon\Delta t  \left(\tilde{a} \otimes_K   \langle  \phi L(f_1^n) \rangle +  {\bf \h \tA}  \langle  \phi L(\h \bff_1)\rangle \right),\\[2mm]
U^{n+1}  &=   U^n - \Delta t  \left(\tilde{b}_1 \nabla_x\cdot  F(U^n) +  {\bf \h \tb}  \nabla_x\cdot F(\h \bU)\right)  
- \varepsilon\Delta t  \left(\tilde{b}_1   \langle \phi L(f_1^n) \rangle +  {\bf \h \tb}  \langle  \phi L(\h \bff_1)\rangle \right).
\end{split}
\end{align}
}
{Here ${\bf \h{\tA}}:= \h{\tA} \otimes_K I_{d_v+2}$, ${\bf\h{\tb}}:= \h{\tb}^\top \otimes_K I_{d_v+2}$, and we define $\nabla_x\cdot F(U^n)=\langle \phi  L(\mathcal{M}[f^n])\rangle$ and $\nabla_x\cdot F(\h \bU) = \langle \phi  L(\mathcal{M}[\h \bF])\rangle$, i.e.,
\begin{align}\label{UFU def2}
\begin{split}
\h{\bf U}&=\begin{pmatrix}
	U^{(2)}, U^{(3)},\cdots, U^{(s)} 
\end{pmatrix}^\top \in \mathbb{R}^{(2+d_v)(s-1)},\\ 
U^{(i)}&=(\rho^{(i)}, \rho^{(i)}u^{(i)}, E^{(i)})^\top\in \mathbb{R}^{2+d_v}\\
%			{\bf F}(\bU)&=(F(U^{(1)}),F(U^{(2)}),...,F(U^{(s)}))^\top\in \mathbb{R}^{(2+d_v)s},\\
%			F(U^{(i)}) &= \left(
%			\begin{array}{c}
	%				\rho^{(i)} u^{(i)}\\
	%				\rho^{(i)} u^{(i)} \otimes u^{(i)} + \rho^{(i)} T^{(i)} Id \\
	%				(E^{(i)} + \rho^{(i)} T^{(i)})u^{(i)}
	%			\end{array}
%			\right)\in \mathbb{R}^{d_v}\times \mathbb{R}^{d_v\times d_v}\ \times \mathbb{R}^{d_v}\\
\nabla_x\cdot F(\h\bU)&=\left(\nabla_x\cdot F(U^{(2)}),\nabla_x\cdot F(U^{(3)}),...,\nabla_x\cdot F(U^{(s)})\right)^\top\in \mathbb{R}^{(2+d_v)(s-1)},\\
\nabla_x\cdot F(U^{(i)}) &= \left(
\begin{array}{c}
	\nabla_x\cdot (\rho^{(i)} u^{(i)})\\
	\nabla_x\cdot (\rho^{(i)} u^{(i)} \otimes u^{(i)} + \rho^{(i)} T^{(i)} Id) \\
	\nabla_x\cdot ((E^{(i)} + \rho^{(i)} T^{(i)})u^{(i)})
\end{array}
\right)\in \mathbb{R}^{2+d_v}\,  (i=1,...,s).
\end{split}
\end{align}}
%\textcolor{red}{where {$\nabla_x\cdot F(U^n)=\langle \phi  L(\mathcal{M}[f^n])\rangle$} and $\nabla_x\cdot F(\h \bU) = \langle \phi  L(\mathcal{M}[\h \bF])\rangle$ is the flux vector defined with $\h \bU = (U^{(2)},U^{(3)},...,U^{(s)})^\top, \langle \phi  \mathcal{M}[\h \bF]\rangle$.}
Then, by Lemma \ref{Lemma1} we obtain
{\begin{align}\label{fin_bis_CK_tris22_bis}
\begin{split}
U^{(1)} &= U^n,\\[2mm]
\h \bU &=   \hat{\be}\otimes_K U^n  - \Delta t  \left(\tilde{a}\otimes_K \nabla_x\cdot F(U^n)  +  {\bf \h \tA} \nabla_x\cdot F(\h \bU)\right)
- \varepsilon\Delta t  \left(\tilde{a} \otimes_K \nabla_x\cdot H(U^n) +  {\bf \h \tA} \nabla_x\cdot H(\h \bU)\right),\\[2mm]
U^{n+1}  &=   U^n - \Delta t  \left(\tilde{b}_1 \nabla_x\cdot F(U^n) +  {\bf \h \tb}  \nabla_x\cdot F(\h \bU)\right)  
- \varepsilon\Delta t  \left(\tilde{b}_1 \nabla_x\cdot  H(U^n) +  {\bf \h \tb} \nabla_x\cdot H(\h \bU) \right),
\end{split}
\end{align}
where
\begin{align*}
\nabla_x\cdot {\bf H}(\h\bU)&=(\nabla_x\cdot H(U^{(2)}),\nabla_x\cdot H(U^{(3)}),...,\nabla_x\cdot H(U^{(s)}))^\top\in \mathbb{R}^{(2+d_v)(s-1)},\\
\nabla_x\cdot H(U^{(i)}) &= \left(
\begin{array}{c}
0\\
\nabla_x\cdot (\rho^{(i)} \Theta_1^{(i)}) \\
\nabla_x\cdot ({\mathbb{Q}_1^{(i)}} + \rho^{(i)} {\Theta_1^{(i)}} { u^{(i)}})
\end{array}
\right)\in \mathbb{R}^{2+d_v},\quad  i=1,...,s.
\end{align*}
}
%%%%%%%%%%%%%%%%%%%%%%%%
Now we evaluate $\h \bff_1$ in (\ref{fin_bis_CK_tris22}). { Substituting (\ref{Exf_1_CK}) into (\ref{eq:GIMEXv1CK}), we obtain}
\begin{align}\label{eq:GIMEXv1_bis_CK} 
\begin{split}
\M[\h\bF]  = \displaystyle 	\M[f^n] \hat{\be} &- \Delta t  \left( \tilde{a}L(\M[f^n]) + \h \tA\,L(\M[\h \bF]) \right) \\
& -\varepsilon \left({\h\bff}_1 - f_1^n\be  + \Delta t\left( \tilde{a}L(f_1^n) + \h \tA\,L(\h \bff_1)\right)\right)\\
& +\Delta t \left( a \tau^n (g^n - f_1^n) + \h A \h \bta (\h\bg-\h{\bff}_1)\right).
\end{split}
\end{align}
Following the same argument for type I in \eqref{RR_2_CHO}-\eqref{substitution thm 1}, %\textcolor{red}{we obtain
%  \begin{equation}\label{eq:GIMEXv1_5} 
%  \h A  \h \bta (\h \bff_1 -\h \bg) = -\left(\displaystyle \frac{\M[\h \bF] - \mathcal{M}[f^n] \h \be}{\Delta t} + \tilde{a}L(\M[f^n])+ \h \tA\,L(\M[\h \bF])\right) - a \tau^n(f_1^n - g^n) +\mathcal{O}\left(\frac{\varepsilon}{\Delta t} \right),
%   \end{equation}
% and by
%   \begin{equation}\label{RR_2_CK}
%  \frac{\M[\h \bF] - \mathcal{\M}[f^n] \h \be}{\Delta t} + \tilde{a}L(\M[f^n]) +  \h \tA\,L(\M[\h \bF]) =  \left(\partial_t \mathcal{M}[F] + L(\mathcal{M}[F])\right)_{t = t^n + \hat{\tilde{c}}\Delta t} + \mathcal{O}(\Delta t^{{q}}),
%   \end{equation} with $\hat{\tilde{c}} = (\tilde{c}_2 ..., \tilde{c}_s)^\top$ 
% and by (\ref{M_lim})}
%we obtain
{we have}
\begin{align}\label{eq:GIMEXv1_4_CK}
\begin{split}
\h A \h \bta (\h \bff_1-\h \bg ) &={- \tilde{a} \left(\mathcal{M}[f^n]\left(A({V^n}):\frac{\sigma(u^n)}{2} + 2B({V^n})\cdot \nabla_x \sqrt{T^n}\right)\right)} \\ 
&- \h\tA\left(\mathcal{M}[\h \bF]\left(A(\h V):\frac{\sigma(\h \bu)}{2} + 2B(\h V)\cdot \nabla_x \sqrt{\h \bT}\right)\right)  \\
&-a \tau^n ( f_1^n - g^n ) + \mathcal{O}\left(\frac{\varepsilon}{\Delta t} \right) + \mathcal{O}(\varepsilon),
\end{split}
\end{align} 
where $\bu^\top = (u^{(1)},\h \bu^\top)$,  $\bT^\top = (T^{(1)},\h \bT^\top)$ with $\h \bu^\top\in \mathbb{R}^{d_v (s-1)}$, $\h \bT^\top \in \mathbb{R}^{s-1}$ and 
{
\begin{align*}
\ds {\textbf{V}}^\top = (V^{(1)},\h V^\top),\quad		{\h \bV} = (V^{(2)},V^{(2)},...,V^{(s)})^\top\in \mathbb{R}^{d_v(s-1)} \quad V^{(i)}=\frac{v-u^{(i)}}{\sqrt{T^{(i)}}}\in \mathbb{R}^{d_v},
\end{align*}}
%$$
%\ds {\textbf{V}}^\top = (V^{(1)},\h V^\top), \quad \h V^\top = \left(\frac{v-\h \bu}{\sqrt{\h \bT}}\right)^\top,
%$$
with $u^{(1)} = u^n$, $T^{(1)}=T^n$ and $V^{(1)} = V^{n}$.
Now, at the initial step $n$, by the assumption \eqref{assumption thm 2}
%\begin{align}\label{fn1}
%    f_1^n = g^n - \frac{1}{\tau^{n}}\mathcal{M}[f^n]\left(A(V^n):\frac{\sigma(u^n)}{2} + 2B(V^n)\cdot \nabla_x \sqrt{T^n}\right) + \mathcal{O}(\varepsilon),
%\end{align}
we have 
\begin{equation*}%\label{gnfn}
\tau^n (f_1^n - g^n) = -\mathcal{M}[f^n]\left(A(V^n):\frac{\sigma(u^n)}{2} + 2B(V^n)\cdot \nabla_x \sqrt{T^n}\right) + \mathcal{O}(\varepsilon),
\end{equation*} 
and substituting this  into (\ref{eq:GIMEXv1_4_CK}), we get
\begin{align}\label{eq:GIMEXv1_4_CK_bis}
\begin{split}
\h \bff_1 &=  \h \bg  - \h \bta^{-1}\h A^{-1}(\tilde{a}-a) \left(\mathcal{M}[f^n]\left(A({V^n}):\frac{\sigma(u^n)}{2} + 2B({V^n})\cdot \nabla_x \sqrt{T^n}\right)\right) \\ 
&- \h \bta^{-1} \h A^{-1}\h\tA\left(\mathcal{M}[\h \bF]\left(A(\h V):\frac{\sigma(\h \bu)}{2} + 2B(\h V)\cdot \nabla_x \sqrt{\h \bT}\right)\right)  + \mathcal{O}\left(\frac{\varepsilon}{\Delta t} \right) + \mathcal{O}(\varepsilon).
\end{split}
\end{align}
{Recall that $F^{(1)}=f^n$, $U^{(1)}=U^n$, we set  $g^{(1)}=g^n$, $f_1^{(1)}=f_1^n$. Now, we multiplying by $v\otimes v$ or $v|v|^2/2$ both side in (\ref{eq:GIMEXv1_4_CK_bis}), and take the integration over $v$, which together with (\ref{Hu}), (\ref{Gu}), (\ref{Fu}) gives
\begin{align}
\begin{split}
	{\rho^{(1)}\Theta_1^{(1)}} &= \nu{\rho^{(1)}\Theta_1^{(1)}}   - \frac{1}{\tau^{n}} \rho^{n} T^{n}\sigma(u^{n}) +\mathcal{O}\left({\varepsilon} \right),\\
	\mathbb{Q}_1^{(1)} + \rho^{(1)}\Theta_1^{(1)} u^{(1)} &= \nu{\rho^{(1)}\Theta_1^{(1)}}u^{(1)}  -   { \frac{1}{\tau^{n}} } \left(\rho^{n} T^{n}\sigma(u^{n})u^{n} + \frac{d_v+2}{2}\rho^{n}T^{n}\nabla_x T^{n}\right) +\mathcal{O}\left({\varepsilon} \right),
\end{split}
\end{align}
and
\begin{align*}
\begin{split}
	{\rho^{(i)}\Theta_1^{(i)}} &= \nu{\rho^{(i)}\Theta_1^{(i)}}  - \frac{1}{\tau^{(i)}} (\h{A}^{-1}(\tilde{a}-a))_{i-1} \rho^{n} T^{n}\sigma(u^{n})\\
	&- {{ \frac{1}{\tau^{(i)}} \sum_{j=1}^{i-2}(\h{A}^{-1}\h{\tilde{A}})_{ij}}} \rho^{(j+1)} T^{(j+1)}\sigma(u^{(j+1)}) +\mathcal{O}\left(\frac{\varepsilon}{\Delta t} \right),\\
	\mathbb{Q}_1^{(i)} + \rho^{(i)}\Theta_1^{(i)} u^{(i)} &= \nu{\rho^{(i)}\Theta_1^{(i)}}u^{(i)}  -   { \frac{1}{\tau^{(i)}} (\h{A}^{-1}(\tilde{a}-a))_{i-1}} \left(\rho^{n} T^{n}\sigma(u^{n})u^{n} + \frac{d_v+2}{2}\rho^{n}T^{n}\nabla_x T^{n}\right)\\
	&-   { \frac{1}{\tau^{(i)}} \sum_{j=1}^{i-2}(\h{A}^{-1}\h{\tilde{A}})_{ij}} \left(\rho^{(j+1)} T^{(j+1)}\sigma(u^{(j+1)})u^{(j+1)} + \frac{d_v+2}{2}\rho^{(j+1)}T^{(j+1)}\nabla_x T^{(j+1)}\right) +\mathcal{O}\left(\frac{\varepsilon}{\Delta t} \right), \quad i=2,...,s
\end{split}
\end{align*}
with $\h{A}^{-1}(\tilde{a}-a) \in \mathbb{R}^{s-1}$.  
%Considering (\ref{fin_bis_CK_tris}),  from (\ref{fin_bis_CK_tris22_bis}) we get,
%{
%\begin{align}\label{HaHA}
%	\begin{split}
	%		H(U^n) &= G(U^n) - (\tau_n)^{-1} \mathcal{F}( U^n) {+\mathcal{O}\left(\varepsilon\right)},\\
	%		%H(\h \bU) &= G(\h \bU) -\h\tau^{-1} \h A^{-1}\h\tA \mathcal{F}(\h \bU){+\mathcal{O}\left(\frac{\varepsilon}{\Delta t}\right)}.
	%		H(U^{(i)}) &= G(U^{(i)}) -\frac{1}{\tau^{(i)}} \left(\h A^{-1}(\ta -a)\right)_{i-1} \mathcal{F}( U^{n}) -\frac{1}{\tau^{(i)}} \sum_{j=1}^{i-1}\left(\h A^{-1}\h\tA\right)_{ij} \mathcal{F}( U^{(j+1)}){+\mathcal{O}\left(\frac{\varepsilon}{\Delta t}\right)},\quad i=2,...,s.
	%	\end{split}
%\end{align}
%}
Thus, the approximations of stress tensor and heat flux are given by 
\begin{align*}
\rho^{(1)}{ \Theta_1^{(1)}} &= \frac{1}{(1-\nu)\tau^{n}} \rho^{n} T^{n}\sigma(u^{n}) +\mathcal{O}\left({\varepsilon}\right),\quad \mathbb{Q}_1^{(1)}=\frac{d_v+2}{2\tau^{n}} \rho^{n}T^{n}\nabla_x T^{n} +\mathcal{O}\left({\varepsilon}\right),\\
\rho{ \bf \h\Theta_1}
&= \begin{pmatrix}
	\rho^{(2)}{ \Theta_1^{(2)}},
	\rho^{(3)}{ \Theta_1^{(3)}},
	\cdots,
	\rho^{(s)}{ \Theta_1^{(s)}}		
\end{pmatrix}^\top,\\
&=-\frac{1}{1-\nu} \left(\h \bta^{-1} \h A^{-1}(\ta -a)p^n\right) \otimes_K \sigma( U^{n}) -\frac{1}{1-\nu}\left( \h \bta^{-1}\h A^{-1}\h{\tilde{A}} \, diag(\h{\bar{p}}) \otimes_K Id\right)\times \sigma(\h\bu) +\mathcal{O}\left(\frac{\varepsilon}{\Delta t} \right)\\
&= - \bar{\mu}^{(1)}\otimes_K \sigma(U^n) -\left( \h{\bar{\mu}} \otimes_K Id\right)\times \sigma({\h\bu}) +\mathcal{O}\left(\frac{\varepsilon}{\Delta t} \right)\\
%		&=  \textcolor{red}{-\left( {\bar{\mu}} \otimes_K Id\right)\times \sigma({\bu}) +\mathcal{O}\left(\frac{\varepsilon}{\Delta t} \right)}\\	
{\bf q}= \mathbb{Q}_1  &= \begin{pmatrix}
	\mathbb{Q}_1^{(2)},
	\mathbb{Q}_1^{(3)},
	\cdots ,
	\mathbb{Q}_1^{(s)}		
\end{pmatrix}^\top,\\ 
&= -\frac{d_v+2}{2}\left( {\h\bta}^{-1} {\h A}^{-1} (\ta -a) p^n \right)  \otimes_K \nabla_x {T^n}
-\frac{d_v+2}{2}\left( {\h\bta}^{-1} {\h A}^{-1}\h{\tilde{A}}  \, diag(\h{\bar{p}}) \otimes_K Id\right)\times \nabla_x \h{\bT} +\mathcal{O}\left(\frac{\varepsilon}{\Delta t} \right)\\
&= -\bar{\kappa}^{(1)} \otimes_K \nabla_x {T^n} -(\h{\bar{\kappa}} \otimes_K Id) \times  \nabla_x \h{\bT} +  \mathcal{O}\left(\frac{\varepsilon}{\Delta t} \right)
%		\\
%		&=\textcolor{red}{-(\bar{\kappa}\otimes_K Id) \nabla_x {\bT} +  \mathcal{O}\left(\frac{\varepsilon}{\Delta t} \right)}
\end{align*}}
where {$\sigma(\h{\bu})= ( \sigma(u^{(2)}),\,...,\,\sigma(u^{(s)}))^\top$, $\nabla_x\h{\bT}= (\nabla_xT^{(2)},\,...,\,\nabla_xT^{(s)})^\top$}, 
$\sigma(u^{(i)}) = \nabla_x u^{(i)} + (\nabla_x u^{(i)})^\top - \frac{2}{d_v}\nabla_x \cdot u^{(i)} Id$ ($i = 2,...,s$) and $\h{\bar{p}}=(p^{(2)}, p^{(3)},...,p^{(s)})^\top$ with $p^{(i)}=\rho^{(i)}T^{(i)}$ ($i = 2,...,s$).
%with $\textrm{b} = \h\tA \h A^{-1}\tilde{a}$ and $\hat{B} =\h\tA \h A^{-1}\h\tA$. 
% Now considering (\ref{fin_bis_CK_tris}),  from (\ref{fin_bis_CK_tris22_bis}) we get,
% \begin{align*}
% \begin{split}
% \tilde{a} H(U^n) &= \tilde{a}G(U^n) - (\tau_n)^{-1} \textrm{b} \mathcal{F}( U^n),\\
% \h {\tilde{A}} H(\h \bU) &= \h{\tilde{A}}G(\h \bU) -\h \bta^{-1} \hat{B} \mathcal{F}(\h \bU),
% \end{split}
% \end{align*}
% with $\textrm{b} = \h\tA \h A^{-1}\tilde{a}$ and $\hat{B} =\h\tA \h A^{-1}\h\tA$. 
%\textcolor{red}{Then, from \eqref{HaHA} we see that the viscosity and thermal conductivity is given by
%	\begin{equation}\label{mu_k_1}
%		\mu^{(1)} = \frac{1}{(1-\nu)\tau^n}\,p^n\textrm{b}, \quad
%		\kappa^{(1)} =\frac{d_v + 2}{2 \tau^n}{p}^n \textrm{b},
%	\end{equation}
%	with $\textrm{b} =\h A^{-1}\tilde{a}$, 
%	and $s-1\times s-1$ matrices ${\h{\bar\mu}}=(\mu_{ij})$ and ${\h{\bar\kappa}}=(\kappa_{ij})$
%	% \begin{equation}\label{mu_k_bis}
%		% \h{{\bar \mu}} = \frac{1}{(1-\nu)}\h \bta^{-1} (\h A^{-1} \h{\tilde{A}})\,\h{\bar{p}}e^\top, \quad
%		% \h{{\bar\kappa}} =\frac{d_v + 2}{2} \, \h \bta^{-1} ( \h A^{-1} \h{\tilde{A}})\, \h{ \bar{p}}e^\top,
%		% \end{equation}
%	\begin{equation}\label{mu_k_bis_22}
%		\h{{\bar \mu}} = \frac{1}{(1-\nu)} \h \tau^{-1} (\h A^{-1} \h{\tilde{A}}) \,\h{\bar{p}}e^\top, \quad
%		\h{{\bar\kappa}} =\frac{d_v + 2}{2} \, \h \tau^{-1}  (\h A^{-1} \h{\tilde{A}}) \h{ \bar{p}}e^\top,
%	\end{equation}
%	with  $\h{\bar{p}}=(p^{(2)},...,p^{(s)})^\top$.}
{Now we introduce the  viscosity and thermal conductivity matrices given by} 
\begin{equation}\label{mukappa}
\bar{\mu}=
\left(
\begin{array}{ll}
0 & 0  \\
\bar{\mu}^{(1)} & \hat{\bar{\mu}}\end{array}
\right),\quad \bar{\kappa}=
\left(
\begin{array}{ll}
0 & 0  \\
\bar{\kappa}^{(1)} & \hat{\bar{\kappa}}\end{array}
\right),
\end{equation}
where
\begin{equation}\label{mu_k_1}
\bar{\mu}^{(1)} =\frac{1}{1-\nu} \left(\h \bta^{-1} \h A^{-1}(\ta -a)p^n\right), \quad
\bar{\kappa}^{(1)} =\frac{d_v+2}{2}\left( {\h\bta}^{-1} {\h A}^{-1} (\ta -a) p^n \right),
\end{equation}
and ${\h{\bar\mu}}=(\mu_{ij})$ and ${\h{\bar\kappa}}=(\kappa_{ij})$
are $s-1\times s-1$ matrices such that
% \begin{equation}\label{mu_k_bis}
% \h{{\bar \mu}} = \frac{1}{(1-\nu)}\h \bta^{-1} (\h A^{-1} \h{\tilde{A}})\,\h{\bar{p}}e^\top, \quad
% \h{{\bar\kappa}} =\frac{d_v + 2}{2} \, \h \bta^{-1} ( \h A^{-1} \h{\tilde{A}})\, \h{ \bar{p}}e^\top,
% \end{equation}
\begin{equation}\label{mu_k_bis_22}
\h{{\bar \mu}} = \frac{1}{(1-\nu)} \h \tau^{-1} (\h A^{-1} \h{\tilde{A}}) \,diag(\h{\bar{p}}), \quad
\h{{\bar\kappa}} =\frac{d_v + 2}{2} \, \h \tau^{-1}  (\h A^{-1} \h{\tilde{A}})\,diag(\h{\bar{p}}).
\end{equation}
Therefore, from \eqref{fin_bis_CK_tris22_bis} we obtain

{
\begin{align}\label{fin_tris}
\begin{split}
	U^{(1)} &= U^n,\\
	\h \bU &=   \h \be \otimes_K U^n - \Delta t \tilde{a} \otimes_K \nabla_x\cdot F(U^n)  - \Delta t {\bf\h{\tilde{A}}} \nabla_x\cdot F(\h \bU)  + \varepsilon \Delta t \tilde{a} \otimes_K \nabla \cdot S(U^n) + \varepsilon \Delta t {\bf \h{\tilde{A}}} \nabla_x \cdot \overline{S}(\h \bU) + 
	%\mathcal{O}(\varepsilon \Delta t^2)
	\mathcal{O}(\varepsilon^{2} )\\[2mm]
	U^{n+1} &=   U^n - \Delta t \tilde{b}_1 \nabla_x \cdot F(U^n)  - \Delta t {\bf\h{\tb}} \nabla_x \cdot F(\h \bU)  + \varepsilon \Delta t \tilde{b}_1 \nabla_x \cdot S(U^n) + \varepsilon \Delta t {\bf\h{\tb}} \nabla_x \cdot \overline{S}(\h \bU) 
	+%\mathcal{O}(\varepsilon \Delta t^{2})+
	\mathcal{O}(\varepsilon^2 ) 
\end{split}
\end{align}
}
with 
%\[
%S(U^n) =\left(
%\begin{array}{c}
%0,\\
%{\mu^{(1)}} \sigma(u^n),\\
%{ \mu^{(1)}} \sigma(u^n) u^n + {\kappa^{(1)}} \nabla_x T^n
%\end{array}
%\right),
%\quad
%\overline{S}(\h \bU) =\left(
%\begin{array}{c}
%0,\\
%\h{\bar\mu} \sigma(\h \bu),\\
%\h{\bar\mu} \sigma(\h \bu) \h \bu + \h{\bar\kappa} \nabla_x \h \bT
%\end{array}
%\right),
%\]
\begin{align*}
\begin{split}
\nabla_x \cdot S(U^n) &=\left(
\begin{array}{c}
	\nabla_x \cdot  {\bf 0},\\
	\nabla_x \cdot ({\mu^{n}} \sigma(u^n)),\\
	\nabla_x \cdot (
	{ \mu^{n}} \sigma(u^n) u^n + {\kappa^{n}} \nabla_x T^n)
\end{array}
\right),\quad \mu^{n} = \frac{1}{(1-\nu)\tau^{n}} p^{n},\quad \kappa^{n}=\frac{d_v+2}{2\tau^{n}} p^{n},\cr
\nabla_x\cdot \overline{S}(\h{\bU}) &=\left(\nabla_x\cdot \overline{S}(U^{(2)}), \nabla_x\cdot \overline{S}(U^{(3)}),...,\nabla_x\cdot \overline{S}(U^{(s)})
\right)^\top \in \mathbb{R}^{(d_v+2)(s-1)},\\ 
\nabla_x\cdot \overline{S}(U^{(i)}) &=\left(
\begin{array}{c}
	\nabla_x \cdot{\bf 0},\\
	\nabla_x \cdot\left({\bar{\mu}^{(1)}_{i-1}} \sigma(u^{n})\right)+ \nabla_x \cdot\left(\sum_{j=1}^{i-2}\h{\bar{\mu}}_{ij} \sigma(u^{(j+1)})\right)\\
	\nabla_x \cdot\left( \bar{\mu}^{(1)}_{i-1}\sigma(u^{n}) u^{n} + {\bar{\kappa}^{(1)}_{i-1} \nabla_x T^{n}}\right) + \nabla_x \cdot\left(\sum_{j=1}^{i-2} \h{\bar{\mu}}_{ij}\sigma(u^{(j+1)}) u^{(j)} + {\h{\bar{\kappa}}_{ij} \nabla_x T^{(j+1)}}\right)
\end{array}
\right)\in \mathbb{R}^{d_v+2}.
%					&=\textcolor{red}{We can remove this!}\left(
%					\begin{array}{c}
	%						\nabla_x \cdot{\bf 0},\\
	%						\nabla_x \cdot \left(\frac{1}{(1-\nu)\tau^{(i)}}  \left(\h A^{-1} (\tilde{a}-a)\right)_{i-1} p^n  \sigma(u)^n\right)+ 
	%						\nabla_x \cdot\left(\sum_{j=1}^{i-2}\frac{1}{(1-\nu)\tau^{(i)}} (\h{A}^{-1}\h{\tilde{A}})_{ij}p^{(j+1)} \sigma(u^{(j+1)})\right),\\
	%						\nabla_x\cdot \left(\left(\frac{1}{(1-\nu)\tau^{(i)}}  \left(\h A^{-1} (\tilde{a}-a)\right)_{i-1} p^n  \sigma(u^n)u^n\right) + \frac{d_v + 2}{2\tau^{(i)}} \left(A^{-1}(\tilde{a}-a)_{i-1}\right)p^n\nabla_x T^n\right) + \nabla_x \cdot\left(\sum_{j=1}^{i-2} \frac{1}{(1-\nu)\tau^{(i)}} (\h{A}^{-1}\h{\tilde{A}})_{ij}p^{(j+1)}\sigma(u^{(j+1)}) u^{(j+1)} + \frac{d_v+2}{2\tau^{(i)}} (\h{A}^{-1}\h{\tilde{A}})_{ij}p^{(j+1)} \nabla_x T^{(j+1)}\right)
	%					\end{array}
%					\right)
%					\in \mathbb{R}^{d_v+2}.
\end{split}
\end{align*}
From $\tau^{(i)}=\tau^{(j)} + \mathcal{O}(\Delta t)+\mathcal{O}(\varepsilon^2)$ for $i,j=1,...,s$, we further get
{
\begin{align*}
\begin{split}
	U^{(1)} &= U^n,\\
	\h \bU &=   \h \be \otimes_K U^n - \Delta t \tilde{a} \otimes_K \nabla_x\cdot F(U^n)  - \Delta t {\bf\h{\tilde{A}}} \nabla_x\cdot F(\h \bU) \\ 
	&+ \varepsilon \Delta t \tilde{a} \otimes_K \nabla \cdot S(U^n) + \varepsilon \Delta t {\bf \h{\tilde{A}}} \nabla_x \cdot \overline{\overline{S}}(\h \bU) +
	\mathcal{O}(\varepsilon^{2} )+ 
	\mathcal{O}(\varepsilon \Delta t^2) \\[2mm]
	U^{n+1} &=   U^n - \Delta t \tilde{b}_1 \nabla_x \cdot F(U^n)  - \Delta t {\bf\h{\tb}} \nabla_x \cdot F(\h \bU)\\  &+ \varepsilon \Delta t \tilde{b}_1 \nabla_x \cdot S(U^n) + \varepsilon \Delta t {\bf\h{\tb}} \nabla_x \cdot \overline{\overline{S}}(\h \bU) 
	+\mathcal{O}(\varepsilon^2 ) +\mathcal{O}(\varepsilon \Delta t^{2})
\end{split}
\end{align*}
}
with 
\begin{align*}
\begin{split}
\nabla_x\cdot \overline{\overline{S}}(\bU) &=\left(\nabla_x\cdot \overline{\overline{S}}(U^{(2)}), \nabla_x\cdot \overline{\overline{S}}(U^{(3)}),...,\nabla_x\cdot \overline{\overline{S}}(U^{(s)})
\right)^\top \in \mathbb{R}^{(d_v+2)(s-1)},\\ 
\nabla_x\cdot \overline{\overline{S}}(U^{(i)}) 
%		&=\left(
%		\begin{array}{c}
	%			\nabla_x \cdot{\bf 0}\\
	%			\nabla_x \cdot\left({\bar{\mu}^{(1)}_{i-1}} \sigma(u^{n})\right) + \nabla_x \cdot\left(\sum_{j=1}^{i-2}\frac{1}{(1-\nu)\tau^{(j+1)}} (\h{A}^{-1}\h{\tilde{A}})_{ij}p^{(j+1)} \sigma(u^{(j+1)})\right)\\
	%			\nabla_x \cdot\left(\sum_{j=1}^{i-2} \frac{1}{(1-\nu)\tau^{(j+1)}} (\h{A}^{-1}\h{\tilde{A}})_{ij}p^{(j+1)}\sigma(u^{(j+1)}) u^{(j+1)} + \frac{d_v+2}{2\tau^{(j+1)}} (\h{A}^{-1}\h{\tilde{A}})_{ij}p^{(j+1)} \nabla_x T^{(j+1)}\right)
	%		\end{array}
%		\right)\\
&=\left(
\begin{array}{c}
	\nabla_x \cdot{\bf 0}\\
	\left(\h A^{-1} (\tilde{a}-a)\right)_{i-1}\nabla_x \cdot \left( \mu^n     \sigma(u^n)\right) \\
	\left(\h A^{-1} (\tilde{a}-a)\right)_{i-1} \nabla_x\cdot \left(\left(  \mu^n  \sigma(u^n)u^n\right) + \kappa^n\nabla_x T^n\right))
\end{array}
\right)\\
&+\left(
\begin{array}{c}
	\nabla_x \cdot{\bf 0}\\
	\sum_{j=1}^{i-2}(\h{A}^{-1}\h{\tilde{A}})_{ij} \nabla_x \cdot\left(\mu^{(j+1)} \sigma(u^{(j+1)})\right)\\
	\sum_{j=1}^{i-2} (\h{A}^{-1}\h{\tilde{A}})_{ij} \nabla_x \cdot\left( \mu^{(j+1)}\sigma(u^{(j+1)}) u^{(j+1)} + \kappa^{(j+1)} \nabla_x T^{(j+1)}\right)
\end{array}
\right)
\end{split}
\end{align*}
where $\mu^{(j)} = \frac{1}{(1-\nu)\tau^{(j)}} p^{(j)}$, $\kappa^{(j)}=\frac{d_v+2}{2\tau^{(j)}} p^{(j)}$. Defining
%			\textcolor{red}{
%				\begin{align*}
%					\begin{split}
	%						\nabla_x\cdot S(U^n)&=\left(
	%						\begin{array}{c}
		%							\nabla_x \cdot{\bf 0}\\
		%							\nabla_x \cdot \left( \mu^n     \sigma(u^n)\right) \\
		%							 \nabla_x\cdot \left(\left(  \mu^n  \sigma(u^n)u^n\right) + \kappa^n\nabla_x T^n\right))
		%						\end{array}
	%						\right)
	%					\end{split}
%				\end{align*}
%			}
%			and
\begin{align*}
\begin{split}
\nabla_x\cdot S(\h{\bU}) &=\left(\nabla_x\cdot S(U^{(2)}), \nabla_x\cdot S(U^{(3)}),...,\nabla_x\cdot S(U^{(s)})
\right)^\top \in \mathbb{R}^{(d_v+2)(s-1)},\\ 
\nabla_x\cdot S(U^{(i)}) 
&=\left(
\begin{array}{c}
	\nabla_x \cdot{\bf 0},\\
	\nabla_x \cdot\left(\mu^{(i)} \sigma(u^{(i)})\right),\\
	\nabla_x \cdot\left( \mu^{(i)}\sigma(u^{(i)}) u^{(i)} + \kappa^{(i)} \nabla_x T^{(i)}\right)
\end{array}
\right)\in \mathbb{R}^{d_v+2},\quad i=2,...,s,
\end{split}
\end{align*}
we finally obtain
{
\begin{align}\label{EqNS}
\begin{split}
	U^{(1)} &= U^n,\\
	\h \bU &=   \h \be \otimes_K U^n - \Delta t \tilde{a} \otimes_K \nabla_x\cdot F(U^n)  - \Delta t {\bf\h{\tilde{A}}} \nabla_x\cdot F(\h \bU) \\
	&+ \varepsilon \Delta t \tilde{a} \otimes_K \nabla_x \cdot S(U^n) + \varepsilon \Delta t {\h{\tilde{A}}\h{A}^{-1}(\tilde{a}-a)} \otimes_K \nabla_x \cdot S(U^n) + \varepsilon \Delta t {\bf \h{B}} \nabla_x \cdot S(\h \bU) +
	\mathcal{O}(\varepsilon^{2} )+ 
	\mathcal{O}(\varepsilon \Delta t^2)\\
	U^{n+1} &=   U^n - \Delta t \tilde{b}_1 \nabla_x \cdot F(U^n)  - \Delta t {\bf\h{\tb}} \nabla_x \cdot F(\h \bU)  \cr
	&+ \varepsilon \Delta t \tilde{b}_1 \nabla_x \cdot S(U^n) +  \varepsilon \Delta t {\h{\tilde{b}}^\top \h{A}^{-1}(\tilde{a}-a)} \nabla_x \cdot S(U^n)
	+
	\varepsilon \Delta t \boldsymbol{\omega} \nabla_x \cdot S(\h \bU) 
	+\mathcal{O}(\varepsilon^2 ) +\mathcal{O}(\varepsilon \Delta t^{2})
\end{split}
\end{align}
}
where ${\bf \h{B}}= \h{B} \otimes_K I_{d_v+2}$ with $\h{B}=\h{\tilde{A}}\h{A}^{-1}\h{\tilde{A}}$ and $\h{\boldsymbol{\omega}}= \h{\omega} \otimes_K I_{d_v+2}$ with $\h{\omega}= \h{\tilde{b}}^\top \h{A}^{-1}\h{\tilde{A}}$.
By the extra condition \eqref{22CK}, we get
\begin{align*}
\begin{split}
U^{(1)} &= U^n,\\
\h \bU &=   \h \be \otimes_K U^n - \Delta t \tilde{a} \otimes_K \nabla_x\cdot F(U^n)  - \Delta t {\bf\h{\tilde{A}}} \nabla_x\cdot F(\h \bU) \\
&+  \varepsilon \Delta t {\mathbb{\bf b}_1}  \otimes_K \nabla_x \cdot S(U^n) + \varepsilon \Delta t {\bf \h{B}} \nabla_x \cdot S(\h \bU) +
\mathcal{O}(\varepsilon^{2} )+ 
\mathcal{O}(\varepsilon \Delta t^2)\\
U^{n+1} &=   U^n - \Delta t \tilde{b}_1 \nabla_x \cdot F(U^n)  - \Delta t {\bf\h{\tb}} \nabla_x \cdot F(\h \bU)  \cr
&+ \varepsilon \Delta t \omega_1  \nabla_x \cdot S(U^n)
+
\varepsilon \Delta t \h{\boldsymbol{\omega}} \nabla_x \cdot S(\h \bU) 
+\mathcal{O}(\varepsilon^2 ) +\mathcal{O}(\varepsilon \Delta t^{2})
\end{split}
\end{align*}
where and {$\mathbb{\bf b}_1  = {\h{\tilde{A}}\h{A}^{-1}\tilde{a}}$} and $\omega_1={\h{\tilde{b}}^\top \h{A}^{-1}\tilde{a}}$.
Next, we consider the following explicit-type RK method of order $p$ 
%			based on the coefficients matrices $B$ and the weights $\omega$
%			\begin{align}\label{TYPE_II}
%				B=\left(\begin{array}{cc}
%					0 & 0\\
%					\mathbb{\bf b}_1  &  \hat{B}
%				\end{array}\right),\quad \omega^\top=(\omega_1,\h{\omega}^\top). 
%			\end{align}
applied to the CNS equations (\ref{NSeq}):
%$$
%\partial_t U + \nabla_x\cdot F(U) = \varepsilon \nabla_x\cdot S(U)$$ is given by
%\SB{reads as \eqref{DiscCL_bis_one}-\eqref{DiscCL_numSol_bis}, i.e, and explicit RK scheme applied to a hyperbolic system, with convect term $\nabla_x \cdot F(U)$ and a perturbed diffusive term $\varepsilon\nabla_x \cdot S(U)$. 
%i.e.,}
\begin{align}\label{fin_tris2}
\begin{split}
U^{(1)} &= U^n,\\
\h \bU &=   \h \be \otimes_K U^n - \Delta t \tilde{a} \otimes_K \nabla_x\cdot F(U^n)  - \Delta t {\bf{\h{\tilde{A}}}} \nabla_x\cdot F(\h \bU) \\
&+  \varepsilon \Delta t \mathbb{\bf b}_1  \otimes_K \nabla_x \cdot S(U^n) + \varepsilon \Delta t {\bf \h{B}} \nabla_x \cdot S(\h \bU)\\
U^{n+1} &=   U^n - \Delta t \tilde{b}_1 \nabla_x \cdot F(U^n)  - \Delta t {\bf\h{\tb}} \nabla_x \cdot F(\h \bU)  \cr
&+ \varepsilon \Delta t \omega_1  \nabla_x \cdot S(U^n)
+
\varepsilon \Delta t \h{\boldsymbol{\omega}} \nabla_x \cdot S(\h \bU) 
\end{split}
\end{align}		
{charterized by the pair $(\tilde{A}, \tilde{b})$ and $(B, \omega)$ given by \eqref{TYPE_II} for type II schemes.}

%			\textcolor{red}{This yields an macroscopic explicit-type additive Runge-Kutta method limit to (30) characterized by the pairs $(\tilde{A}, \tilde{b})$ and $(B,\omega)$ given by \eqref{TYPE_II} for type II schemes.}
%\begin{equation}\label{fin_tris2}
%\bU=   U^n\be  - \Delta t \tilde{A} \nabla_x F(\bU)  + \varepsilon \Delta t {\tilde{A}} S(\bU)  +
%%\mathcal{O}(\varepsilon \Delta t^{2})+
%\mathcal{O}(\varepsilon^2 ) 
%\end{equation}
%\begin{equation}\label{fin_tris22}
%U^{n+1} =   U^n  - \Delta t  \tilde{b}^\top \nabla_x F(\bU)  + \varepsilon \Delta t  \tilde{b}^\top S( \bU) + 
%%\mathcal{O}(\varepsilon \Delta t^2)+
%\mathcal{O}(\varepsilon^2) 
%\end{equation}
%i.e., an explicit discretization of the CNS equation characterized by the explicit scheme, i.e.,$ (\tilde{A}, \tilde{b})$. 
The same conclusions regarding the local truncation error presented in Theorem \ref{thm 1} can be applied here for the case of type II.
\end{proof}

\subsection{Uniform accuracy}

{
{In the stiff case, i.e., $\Delta t \gg \varepsilon$, the phenomenon of order reduction may occur, particularly in the worst case when $\Delta t \approx \varepsilon$  (intermediate regime of $\varepsilon$). In this scenario, the classical order of the method can drop off, leading to a loss of accuracy in IMEX-RK schemes  within this regime.}
%{In \eqref{LTE}, in the worse case scenario $  \Delta t \approx \varepsilon $, the maximum error is $\mathcal{O}(\Delta t)$ for $p >1$. For instance, this result states that the 
	%	a third-order schemes $p = 3$ may exhibit only first-order accuracy in the worst case. 
	%	However, this is a lower bound, and in practice, better behavior is often observed (see numerical results).}
{This phenomenon was investigated in detail in \cite{Boscarino2007, Boscarino2017, IMEX_book}. In particular in \cite{Boscarino2017} the authors considered a prototype hyperbolic relaxation system and performed an asymptotic expansion up to $\mathcal{O}(\varepsilon)$ for IMEX-RK methods of type I and II.  They showed that, under additional order conditions, these schemes effectively reduce to explicit-type RK methods at the $\mathcal{O}(\varepsilon)$ level.  These methods were carefully designed to accurately match the $\mathcal{O}(\varepsilon)$ terms up to a desired order $p$, ensuring that effectively behave as explicit RK schemes of order $p$ for the convection-diffusion equation when $\varepsilon$ is small but not negligible. 
}
{However, constructing high-order IMEX-RK schemes of type I requires additional stages to satisfy extra-order conditions, their construction is particularly challenging, \cite{Boscarino2017, IMEX_book}.  In this work, we therefore focus on IMEX-RK schemes of type II,\cite{Boscarino2017}, since their structures make them easier to construct.
	\\
	We note that if conditions \eqref{22CK} are not satisfied, order reduction is likely to occur in these intermediate regimes of $\varepsilon$, for example, in the  numerical tests (Test 1 and Test 2), a classical ARS scheme of type II in \eqref{A3:ARS2-IMEX3} show this phenomenon of order reduction. Alternatively, we adopt IMEX-RK schemes of type II introduced in  \cite{Boscarino2017} that satisfy  conditions \eqref{22CK} ensuring the consistency of the schemes (Theorem  \ref{thm 2}). Furthermore, these schemes satisfy additional order conditions, preventing order reduction in the intermediate regime and ensuring uniform accuracy over the full range of Knudsen numbers. In particular, we examine two IMEX-RK schemes of type II called IMEX-II-GSA3 \eqref{GSA:CK-IMEX3} and IMEX-II-ISA3 \eqref{A3:CK-IMEX3}. These two schemes are specifically designed in \cite{Boscarino2017} to improve the resolution of hyperbolic relaxation systems, especially in the intermediate regime, i.e., $\varepsilon \approx \Delta t$.
	\\
	In Section \ref{sec: Numerocal results}, we compare these methods with other IMEX-RK schemes of type I or II existing in literature.
}}

{Finally, we note that IMEX-II-ISA3 scheme fulfills  $\tilde{b} = b$ and the quantity in \eqref{EqNS} becomes}
\begin{equation}\label{bb_b}
\h{\tilde{b}}^\top \h A^{-1}({a} - \tilde{a}) = \h{{b}}^\top \h A^{-1}({a} - \tilde{a}) = \h \be_s^\top (a - \tilde{a})
\end{equation}
which equals zero if $\tilde{a}_{s1} = a_{s1}$. {It is worth noting that the assumption 
$\tilde{b} = b$, plays an important rule in mitigating the lost of accuracy, especially in regimes where $\varepsilon \approx \Delta t$, as demonstrated in \cite{Boscarino2007}. In Section \ref{sec: Numerocal results}, we will show that IMEX-II-GSA3 scheme, which is GSA and $\tilde{b} \neq b$ leads to a mild order reduction for small values of $\varepsilon$. On the contrary, IMEX-II-ISA3 with $\tilde{b} = b$, does not exhibit order reduction, (as confirmed by the numerical results in the next section)}.

\section{Numerical Results}\label{sec: Numerocal results}
In this section, we test the performance of different types of IMEX RK schemes using two distinct models: the 1 + 1 BGK model and the 1 + 2 ES-BGK model. We use the fifth-order finite difference WENO method \cite{Shu} to approximate spatial derivatives. The discretization of the velocity domain is achieved through uniform grid points within a sufficiently large interval or domain. The choice of free parameter $\nu$ and collision frequency $\tau$ will be explained in each problem.

Throughout this section, we consider IMEX-RK schemes of type I and II. In particular, we take into account the GSA IMEX-II-GSA3 scheme and IMEX-II-ISA3 with $\tilde{b} = b$. %These schemes have been derived by imposing additional order conditions to achieve an accurate time discretization in the diffusion limit.
%These new IMEX-RK schemes found in \cite{Boscarino2007} require more stages than the commonly used ones (see below).
For comparison, we consider the {SI-IMEX(4,4,3)} scheme of type I  not GSA in \cite{IMEX_book}, and the scheme GSA ARS(4,4,3)  of type II in \cite{ARS}.

%In both types, the methods require implicit solver for computing relaxation terms.
%However, the use of collision invariants allows us to treat the implicit terms explicitly.

In the following, these schemes are represented, as usual, by the double Butcher tableau.

\begin{itemize}

\item  Third-order {SI-IMEX(4,4,3)} scheme of type I in \cite{IMEX_book}. (Top: explicit method. Bottom: implicit method).

\begin{align}
	\begin{split}
		\begin{array}{c|cccc}
			0& 0 & 0 & 0 & 0\\
			\gamma & \gamma & 0 & 0 & 0\\
			c_3 &  1.243893189483362 &  -0.525959928729133 & 0& 0\\
			1 & 0.630412558152867 & 0.786580740199155  & -0.416993298352022 & 0\\
			\hline
			& 0 & 1.208496649176010 & -0.644363170684468 & \gamma
		\end{array} \\[4mm]  \nonumber 
		\begin{array}{c|cccc}
			\gamma &  \gamma & 0 & 0& 0\\
			\gamma & 0 & \gamma & 0& 0\\
			c_3 & 0 & 0.282066739245771 & \gamma & 0 \\
			1  & 0 & 1.208496649176010 & -0.644363170684468 & \gamma\\
			\hline 
			& 0 & 1.208496649176010 & -0.644363170684468 & \gamma
		\end{array}
	\end{split}
\end{align}

$\gamma = 0.435866521508459, \, c_3 = 0.7179332607542294$.

% \QQ{\item  Third-order SA ARS(3,4,3) in  \cite{ARS}. (Top: explicit method. Bottom: implicit method).
	%  \begin{align}
		% \begin{split}
			%     \begin{array}{c|cccc}
				%               0& 0 & 0 & 0 & 0\\
				%               \gamma & \gamma & 0 & 0 & 0\\
				%               c_3 & 0.3212788860286277 &  0.3966543747256017 & 0& 0\\
				%               1 & -0.1058582960718796 & 0.5529291480359398 & 0.5529291480359398 & 0\\
				%               \hline
				%               & 0 & 1.208496649176010 & -0.644363170684468 & \gamma
				% \end{array} \\[4mm]  \nonumber 
			% \begin{array}{c|cccc}
				%                0 & 0 & 0 & 0& 0\\
				%                \gamma & 0 & \gamma & 0& 0\\
				%                c_3 & 0 & 0.282066739245771 & \gamma & 0 \\
				%               1  & 0 & 1.208496649176010 & -0.644363170684468 & \gamma\\
				%               \hline 
				%                  & 0 & 1.208496649176010 & -0.644363170684468 & \gamma
				% \end{array}
			% \end{split}
		% \end{align}
	% $\gamma = 0.435866521508459, \, c_3 = 0.7179332607542294$.
	% }

\item Third-order GSA ARS(4,4,3) in \cite{ARS}:

\begin{equation}\label{A3:ARS2-IMEX3}
	\begin{array}{c|ccccc}
		0 & 0 & 0 & 0 &0 &0 \\
		1/2& 1/2& 0 & 0 & 0 &0\\
		2/3 & 11/18 & 1/18 & 0 & 0 &0 \\
		1/2 & 5/6 &-5/6 & 1/2 & 0 &0\\
		1 & 1/4 &7/4 & 3/4 & -7/4 &0\\
		\hline \\
		& 1/4 &7/4 & 3/4 & -7/4 &0
	\end{array} \qquad
	\begin{array}{c|ccccc}
		0 & 0 & 0 & 0 &0 &0 \\
		1/2&  0&1/2 & 0 & 0 &0\\
		2/3 & 0 & 1/6 & 1/2 & 0 &0 \\
		1/2 & 0 &-1/2 & 1/2 & 1/2 &0\\
		1 & 0 &3/2 & -3/2 & 1/2 &1/2\\
		\hline \\
		&0&  3/2 &-3/2 & 1/2 & 1/2 
	\end{array}
\end{equation}

\item Third-order GSA IMEX-II-GSA3 of type ARS, \cite{Boscarino2017}. (Top: explicit method. Bottom: implicit method): 
\begin{align}\label{GSA:CK-IMEX3}
	\begin{split}
		\begin{array}{c|ccccccc}
			0 & 0 & 0 & 0 &0 &0 &0 &0 \\
			43/100& 43/100& 0 & 0 & 0 &0&0 &0 \\
			336/929 & 0 & 336/929 & 0 & 0 &0 &0 &0 \\
			-29/42 & 0 & -29/42 & 0 & 0 &0&0 &0 \\
			581/527 & 0 &-1213/770 & 2491/956 & 267/3701 &0&0 &0 \\
			2/3 & 0 & -197/1238 & 499/743 & 0 &581/3768 &0 &0 \\
			1 & 0 & 263/620 & 134/16589 & 1040/22119 &0  &4777/9174 &0 \\
			\hline \\
			& 0 & 263/620 & 134/16589 & 1040/22119 &0  &4777/9174 &0 \\
		\end{array}\\
		\begin{array}{c|ccccccc}
			0 & 0 & 0 & 0 &0 &0 &0 &0 \\
			43/100& 0 & 43/100 & 0 & 0 &0&0 &0 \\
			336/929 & 0 & -168/2459 & 43/100 & 0 &0 &0 &0 \\
			-29/42 & 0 & -2353/2100 & 0 & 43/100 &0&0 &0 \\
			581/527 & 0 &889/1322 & 0 & 0 &43/100 &0 &0 \\
			2/3 & 0 & 247/2416 & 0 &408/3035 &0 & 43/100&0 \\
			1 & 0 & 872/1201 & 0 & 139/4081 & -50/237  &434/20817 &43/100 \\
			\hline \\
			& 0 & 872/1201 & 0 & 139/4081 & -50/237  &434/20817 &43/100 \\
		\end{array}
	\end{split}
\end{align}

\item Third-order IMEX-II-ISA3, \cite{Boscarino2017}. (Top: explicit method. Bottom: implicit method):

\begin{align}\label{A3:CK-IMEX3}
	\begin{split}
		\begin{array}{c|ccccccc}
			0 & 0 & 0 & 0 &0 &0 &0 &0 \\
			1/5& 1/5& 0 & 0 & 0 &0&0 &0 \\
			1/3 & 0 & 1/3 & 0 & 0 &0 &0 &0 \\
			2/3 & 0 & 557/867 & 7/289 & 0 &0&0 &0 \\
			3/4 & 0 &16/289 & 803/1156 & 0 &0&0 &0 \\
			1 & 0 & 13348/3993 & -9355/3993 & 0 & 0 &0 &0 \\
			1 & 0 & 75/154 & 0 & -3/14 &8/11&0 &0 \\
			\hline \\
			& 0 & -155/112  & 251/80  &-547/280 &2/3 &1/3& 1/5 \\ 
		\end{array}\\
		\begin{array}{c|ccccccc}
			0 & 0 & 0 & 0 &0 &0 &0 &0 \\
			1/5& 0 & 1/5 & 0 & 0 &0&0 &0 \\
			1/3 & 0 & 2/15 & 1/5 & 0 &0 &0 &0 \\
			2/3 & 0 & 7/15 & 0 & 1/5 &0&0 &0 \\
			3/4 & 0 &1137/1004 & -731/1255 & 0 &1/5 &0 &0 \\
			1 & 0 & 447/565 & 0 & -636/613 &519/496 & 1/5&0 \\
			1 & 0 & -155/112  & 251/80  &-547/280 &2/3 &1/3& 1/5  \\
			\hline \\
			& 0 & -155/112  & 251/80  &-547/280 &2/3 &1/3& 1/5 \\ 
		\end{array}
	\end{split}
\end{align}

\end{itemize}

\subsection{Accuracy test.} 
{\em Test 1.} In this test we investigate numerically the convergence rate solving the BGK model for a smooth solution. The initial data is taken as
\begin{align}\label{init well-prepared 1D}
\begin{split}
	\rho_0(x)&=1+0.2\sin(\pi x),\quad {u_0(x)=1},\quad T_0(x)=\frac{1}{\rho_0(x)},\cr
	f_0(x,v)&=\mathcal{M}(x,v) - \frac{\varepsilon}{\tau}\left(I-\Pi_\mathcal{M}\right)(v_1\partial_x \mathcal{M})
\end{split}
\end{align}
where
\[ 
\mathcal{M}(x,v)=\frac{\rho_0(x)}{2\pi T_0(x)}\exp\left(-\frac{|v-u_0(x)|^2}{2T_0(x)}\right).
\]
For an explicit definition of the term $\left(I-\Pi_\mathcal{M}\right)(v_1\partial_x \mathcal{M})$, please refer to Remark \ref{B1} in the appendix. \\
We use $32$ uniform points for velocity in the interval $v \in [-10, 10]$, and $N_x$ uniform
points for space in the interval $x \in [0, 2]$  with periodic boundary conditions. We set $\Delta t = \text{CFL}\Delta x /10 $ with $\text{CFL} = 0.9$ and $\tau=1$.  

Since the exact solution is not available for this test, as well as for the subsequent test, the $L^1$ errors are computed as the difference of the numerical solutions on two consecutive meshes in space, i.e., we use the numerical solution on a finer mesh with mesh size $\Delta x/2$ as the reference solution to compute the error for solutions on the mesh size of $\Delta x$. 

Table \ref{1d accuracy} shows results for third order accurate IMEX-RK schemes at $t = 0.25$. First we can see order reduction in the intermediate regime for ARS(4,4,3) and {SI-IMEX(4,4,3)} scheme. Although we observe order reductions for IMEX-RK at small Knudsen numbers, such as $10^{-4}$, it is a typical issue of IMEX-RK methods. In general, it is highly nontrivial to avoid order reduction of IMEX-RK schemes in the intermediate regime (for example, see papers \cite{Boscarino2007, Boscarino09, Boscarino10}).

For comparison, we test the two third-order schemes introduced in \cite{Boscarino2017} IMEX-II-GSA3 (\ref{GSA:CK-IMEX3}) and IMEX-II-ISA3 (\ref{A3:CK-IMEX3}) for this example. The third order accuracy is achieved for IMEX-II-ISA3 for all  values of $\varepsilon$ even in the intermediate regime, i.e., the scheme maintains uniform accuracy, on the other hand IMEX-II-GSA3 shows a slight order reduction for the value $\varepsilon = 10^{-4}$. Note that, as discussed in Section \ref{Sec:IMEXRK}, since the IMEX-II-ISA3 scheme has same weights, it performs better than the other methods for intermediate values of $\varepsilon$.

\begin{table}[!htbp]
\centering	
{\small
	\begin{tabular}{ |p{3cm}||p{3cm}|p{3cm}|p{3cm}|p{3cm}|}
		\hline
		\multicolumn{5}{ |c| }{Relative $L^1$ error and order of density} \\ 
		\hline
		\multicolumn{5}{|c|}{{SI-IMEX(4,4,3)}-WENO35} \\
		\hline
		& 
		$ \varepsilon = 1$ &
		$ \varepsilon = 10^{-2}$ &
		$ \varepsilon = 10^{-4}$ & 
		$ \varepsilon = 10^{-6}$\\
		\hline
		($N_x$,$2N_x$) & $L_1$ \, \, \, \, \, Order& $L_1$ \, \, \, \, \, Order & $L_1$ \, \, \, \, \, Order &  $L_1$ \, \, \, \, \, Order\\
		\hline
		(40,80)& 5.5435e-06
		\, \, \,  &  2.4034e-07
		& 2.4841e-07
		&2.5178e-07 \\
		(80,160)    & 1.9598e-07
		\, \, \, 4.82
		&1.5473e-08
		\, \, \, 3.96
		&8.0955e-09
		\, \, \,4.94
		& 8.1595e-09 \, \, \, 4.95\\
		(160,320) & 8.6077e-09
		\, \, \, 4.51
		& 1.9379e-09
		\, \, \, 3.00
		&2.9470e-10
		\, \, \,4.78
		&  2.5731e-10   \, \, \,  4.99 \\
		(320,640) & 6.0318e-10
		\, \, \, 3.84
		&  2.5092e-10
		\, \, \, 2.95
		& 3.4689e-11
		\, \, \,3.09
		&8.1584e-12 \, \, \, 4.98\\
		(640,1280) &  6.2088e-11 \, \, \, 3.28 & 3.2067e-11 \, \, \, 2.97& 9.1516e-12   \, \, \,1.92
		& 3.3436e-13 \, \, \, 4.61\\
		\hline
		\hline
		\multicolumn{5}{|c|}{GSA ARS(4,4,3)-WENO35} \\
		\hline
		& 
		$ \varepsilon = 1$ &
		$ \varepsilon = 10^{-2}$ &
		$ \varepsilon = 10^{-4}$ & 
		$ \varepsilon = 10^{-6}$\\
		\hline
		($N_x$,$2N_x$) & $L_1$ \, \, \, \, \, Order& $L_1$ \, \, \, \, \, Order & $L_1$ \, \, \, \, \, Order &  $L_1$ \, \, \, \, \, Order\\
		\hline
		(40,80)& 1.6273e-05  \, \, \,  &  1.2738e-06  &1.1499e-06 &1.1482e-06 \\
		(80,160)    & 8.5850e-07 \, \, \, 4.24 &4.2936e-08 \, \, \, 4.89 &3.5346e-08\, \, \,   5.02& 3.5193e-08 \, \, \, 5.02\\
		(160,320) &      5.9885e-08  \, \, \,   3.84 & 3.5424e-09 \, \, \, 3.60&1.3376e-09   \, \, \,4.72 &  1.2315e-09   \, \, \,  4.83 \\
		(320,640) &    6.2342e-09  \, \, \, 3.26 &  4.7790e-10 \, \, \, 2.88& 8.0292e-10 \, \, \,0.73 &5.6966e-11 \, \, \, 4.43\\
		(640,1280) &  7.4782e-10 \, \, \, 3.06 & 6.2449e-11 \, \, \, 2.93& 4.7953e-10   \, \, \,0.74& 6.3566e-12 \, \, \, 3.16\\
		\hline
		\hline
		\multicolumn{5}{|c|}{IMEX-II-GSA3-WENO35} \\
		\hline
		($N_x$,$2N_x$) & $L_1$ \, \, \, \, \, Order& $L_1$ \, \, \, \, \, Order & $L_1$ \, \, \, \, \, \, \, Order &  $L_1$ \, \, \, \, \, Order\\
		\hline 
		(40,80) & 1.3047e-05  & 1.2245e-06&  1.1307e-06&1.1301e-06\\
		(80,160)    &4.2711e-07 \, \, \, 4.93& 3.5404e-08 \, \, \, 5.11&3.2615e-08\, \, \, 5.11   & 3.2616e-08 \, \, \,   5.11 \\
		(160,320) &  1.4238e-08 \, \, \, 4.90   & 1.0951e-09 \, \, \, 5.01 &  8.8667e-10\, \, \, 5.20& 8.8936e-10 \, \, \,  5.20\\
		(320,640) & 2.4517e-09 \, \, \,  2.53    & 1.1533e-10 \, \, \, 3.24 & 2.3656e-11\, \, \, 5.22  &  2.1240e-11 \, \, \,  5.38\\
		(640,1280) &  3.2531e-10 \, \, \, 2.92    &1.5141e-11  \, \, \, 2.93  & 5.8846e-12\, \, \, 2.00 & 3.5767e-12 \, \, \,  2.57\\
		\hline
		\hline
		\multicolumn{5}{|c|}{IMEX-II-ISA3-WENO35} \\
		\hline
		($N_x$,$2N_x$) & $L_1$ \, \, \, \, \, Order& $L_1$ \, \, \, \, \, Order & $L_1$ \, \, \, \, \, \, \, Order &  $L_1$ \, \, \, \, \, Order\\
		\hline 
		(40,80) &  1.5114e-05  &1.2600e-06 &  1.1486e-06 &  1.1426e-06\\
		(80,160)    &6.8701e-07 \, \, \, 4.45& 3.9869e-08 \, \, \, 4.98&3.4483e-08 \, \, \,5.05 & 3.4369e-08 \, \, \, 5.05 \\
		(160,320) &   3.6504e-08 \, \, \, 4.23   & 1.4366e-09 \, \, \,  4.79 &  1.1093e-09 \, \, \,4.95&  1.1169e-09 \, \, \, 4.94\\
		(320,640) & 3.2821e-09 \, \, \,   3.47     &   7.6254e-11 \, \, \, 4.23&  4.9755e-11 \, \, \,4.47  &    4.4882e-11 \, \, \, 4.63\\
		(640,1280) &  3.7817e-10 \, \, \, 3.11&  7.0441e-12  \, \, \, 3.43  & 4.4539e-12 \, \, \,3.48& 4.7061e-12 \, \, \, 3.25\\
		\hline
	\end{tabular}
	\caption{Accuracy test for the BGK equation. Initial data is given in \eqref{init well-prepared 1D},  $L^1$ error on the density $\rho$ at $T = 0.25$ and $N_v = 32$.}\label{1d accuracy}
}
\end{table}

{\em Test 2.} In this test we investigate numerically the convergence rate solving the ES-BGK model for a smooth solution considering $d_v = 2$ with $v = (v_1, v_2)$, and  the initial well-prepared initial data (\ref{init well-prepared 1D}),  where
\begin{equation}\label{init well-prepared 2D}
\mathcal{M}(x,v)=\frac{\rho_0(x)}{2\pi T_0(x)}\exp\left(-\frac{(v_1-u_0(x))^2 + v_2^2}{2T_0(x)}\right).
\end{equation}
For the velocity discretization, we use $32 \times 32$ uniform points in the domain $v \in [-10, 10] \times [-10,10]$. Note that, for the {SI-IMEX(4,4,3)} scheme in order to achieve the expected order of convergence, we required more points in velocity. In this test we use $N_v = 80$ uniform points for the velocity. We use $N_x$ uniform points for space in the interval $x \in  [0,2]$  with periodic boundary conditions. The time step is taken as $\Delta t = \text{CFL} \Delta x/10$ with $\text{CFL} = 0.9$, $\tau=1$ and $\nu = -1/2$. 
\begin{table}[htbp]
\centering	
{\small
	\begin{tabular}{ |p{3cm}||p{3cm}|p{3cm}|p{3cm}|p{3cm}|}
		\hline
		\multicolumn{5}{ |c| }{Relative $L^1$ error and order of density} \\ 
		\hline   
		\multicolumn{5}{|c|}{{SI-IMEX(4,4,3)}-WENO35 $N_v=80$} \\
		\hline
		& 
		$ \varepsilon = 1$ &
		$ \varepsilon = 10^{-2}$ &
		$ \varepsilon = 10^{-4}$ & 
		$ \varepsilon = 10^{-6}$\\
		\hline
		\hline
		($N_x$,$2N_x$) & $L_1$ \, \, \, \, \, Order& $L_1$ \, \, \, \, \, Order & $L_1$ \, \, \, \, \, Order &  $L_1$ \, \, \, \, \, Order\\
		\hline
		(40,80)&3.1641e-05 
		& 1.4962e-06
		&   1.1980e-06
		& 1.1999e-06
		\, \, \,  
		\\
		(80,160)    
		& 1.8285e-06
		\, \, \, 4.11   
		&4.4260e-08
		\, \, \,5.07
		& 3.5725e-08
		\, \, \, 5.06
		& 3.6019e-08
		\, \, \, 5.05
		\\
		(160,320) 
		& 8.6363e-08
		\, \, \,  4.40
		&1.7243e-09
		\, \, \,4.68
		&  1.0440e-09
		\, \, \, 5.09
		&      1.1081e-09
		\, \, \,   5.02
		\\
		(320,640) 
		& 6.2789e-09
		\, \, \, 3.78
		& 1.9614e-10
		\, \, \,3.13
		&  5.0039e-11
		\, \, \, 4.38
		&    3.4509e-11
		\, \, \,  5.00
		\\
		(640,1280) &   4.1337e-10
		\, \, \, 3.92
		&  2.5833e-11   \, \, \,2.92
		& 1.0886e-11 \, \, \, 2.20
		&  3.4304e-12 \, \, \, 3.33\\
		\hline
		\hline
		\multicolumn{5}{|c|}{GSA ARS(4,4,3)-WENO35 $N_v = 32$} \\
		\hline
		& 
		$ \varepsilon = 1$ &
		$ \varepsilon = 10^{-2}$ &
		$ \varepsilon = 10^{-4}$ & 
		$ \varepsilon = 10^{-6}$\\
		\hline
		\hline
		($N_x$,$2N_x$) & $L_1$ \, \, \, \, \, Order& $L_1$ \, \, \, \, \, Order & $L_1$ \, \, \, \, \, Order &  $L_1$ \, \, \, \, \, Order\\
		\hline
		(40,80)& 5.8625e-06   \, \, \,  &   8.7440e-08  & 6.1776e-08 &6.2235e-08 \\
		(80,160)    & 4.0040e-07 \, \, \, 3.87 & 5.0511e-09  \, \, \, 4.11 &2.3576e-09\, \, \,   4.71& 2.4420e-09 \, \, \, 4.67\\
		(160,320) &      4.0934e-08  \, \, \,   3.29 &  5.0181e-10 \, \, \, 3.33&2.5187e-10   \, \, \,3.22 &  1.4842e-10   \, \, \,  4.04 \\
		(320,640) &    4.8923e-09 \, \, \,  3.06 &  6.0480e-11  \, \, \, 3.05& 1.1787e-10 \, \, \,1.09 &1.4216e-11 \, \, \, 3.38\\
		(640,1280) &  6.0623e-10 \, \, \, 3.01 & 7.5418e-12 \, \, \, 3.00&  4.6731e-11   \, \, \,1.33&   1.8633e-12 \, \, \, 2.93\\
		\hline
		\hline
		\multicolumn{5}{|c|}{IMEX-II-GSA3-WENO35 $N_v = 32$} \\
		\hline
		($N_x$,$2N_x$) & $L_1$ \, \, \, \, \, Order& $L_1$ \, \, \, \, \, Order & $L_1$ \, \, \, \, \, \, \, Order &  $L_1$ \, \, \, \, \, Order\\
		\hline 
		(40,80) & 3.3509e-06   &  6.5315e-08& 1.1307e-06&4.5811e-07\\
		(80,160)    &1.1158e-07 \, \, \, 4.90&  1.4444e-09 \, \, \, 5.49&3.2615e-08\, \, \, 5.11   &  1.2718e-08 \, \, \,   5.17 \\
		(160,320) &  1.5075e-08 \, \, \, 2.88   & 1.4088e-10 \, \, \, 3.35 &  8.8667e-10\, \, \, 5.20& 3.8152e-10 \, \, \,  5.05\\
		(320,640) & 3.7304e-09  \, \, \,  2.01    & 3.6735e-11 \, \, \, 1.93& 2.3656e-11 \, \, \,5.22  &   1.1960e-11 \, \, \, 5.00\\
		(640,1280) &   2.6209e-10 \, \, \, 3.83    &3.2301e-12 \, \, \, 3.50  & 5.8846e-12 \, \, \,2.00 &  1.7831e-12 \, \, \,  2.75\\
		\hline
		\hline
		\multicolumn{5}{|c|}{IMEX-II-ISA3-WENO35 $N_v = 32$} \\
		\hline
		($N_x$,$2N_x$) & $L_1$ \, \, \, \, \, Order& $L_1$ \, \, \, \, \, Order & $L_1$ \, \, \, \, \, \, \, Order &  $L_1$ \, \, \, \, \, Order\\
		\hline 
		(40,80) &   4.9602e-06  &7.9449e-08 &   5.9187e-08 &  4.6149e-07\\
		(80,160)    &  2.6934e-07 \, \, \, 4.20& 3.4997e-09 \, \, \, 4.50& 2.0718e-09 \, \, \,4.83 &  1.3395e-08\, \, \, 5.10 \\
		(160,320) &    2.3668e-08 \, \, \, 3.50   & 2.5659e-10 \, \, \,  3.76 &  1.0526e-10 \, \, \,4.29&  4.7348e-10\, \, \, 4.82\\
		(320,640) &   2.7055e-09 \, \, \,   3.12     &   2.7401e-11 \, \, \, 3.22&   9.1594e-12 \, \, \,3.52  &    2.0956e-11 \, \, \,4.50\\
		(640,1280) &   3.3206e-10  \, \, \, 3.02&  3.2314e-12 \, \, \, 3.08  & 1.1581e-12\, \, \, 2.98& 2.6896e-12\, \, \, 2.96\\
		\hline
		%%%%%%%
	\end{tabular}
	\caption{Accuracy test for the ES-BGK equation. Initial data is given in \eqref{init well-prepared 1D} with \eqref{init well-prepared 2D},  $L^1$ error on the density $\rho$ at $T = 0.25$.}\label{2d_bis accuracy}
}
\end{table}

Table \ref{2d_bis accuracy} shows the result for the third order accurate IMEX-RK schemes at $t = 0.25$. As the one dimensional case we can see order reduction in the intermediate regime ($\varepsilon=10^{-4}$) for the ARS(4,4,3) scheme. 
The third order scheme IMEX-II-ISA3 achieves the third order accuracy for all the values of $\varepsilon$ even in the intermediate regime, whereas IMEX-II-GSA3 and {SI-IMEX(4,4,3)} schemes show a slight order reduction for the value $\varepsilon = 10^{-4}$.

{
% We conclude this section with the following remark. 
At the theoretical level, the ES-BGK model approximates the NSE equations with an error of order $\mathcal{O}(\varepsilon^2)$, as is well known from the Chapman-Enskog expansion. However, verifying this convergence numerically is extremely challenging. Indeed, the numerical method approximates the ES-BGK model, not the NSE system directly. Hence, the observed error will always be a combination of the {modeling error}, between ES-BGK and NSE, and the {numerical discretization error} in approximating both models. 
With this view point, to show the error between ES-BGK and NSE is $\mathcal{O}(\varepsilon^2)$, we consider the same initial data used in Test 2,	compute very accurate solutions for each model for a range of Knudsen numbers, and check whether the discrepancy of the numerical solutions show expected order $\mathcal{O}(\varepsilon^2)$. In view of this, we compute numerical solutions with sufficiently small $\Delta t$ and $\Delta x$ (for both ES-BGK and NSE) and large number of velocity grid points (for ES-BGK), and measure the discrepancy in relative $L^1$-norm. In the following figure \ref{fig ESNSE}, it appears that the convergence rates are close to $2$ when $10^{-4} \leq \varepsilon \leq 10^{-2}$. On the other cases, the second order is not obtained.
Note that when $\varepsilon$ is relatively small $(\varepsilon < 10^{-4})$, the result is improved as we increase the accuracy of numerical solutions. However, when $\varepsilon$ is relatively large $(\varepsilon>10^{-2})$,  the errors are not improved even for high accurate solutions. This can be interpreted as modelling errors between ES-BGK and  NSE.}

\begin{figure}[!htbp]
\centering
\includegraphics[width=0.46\textwidth]{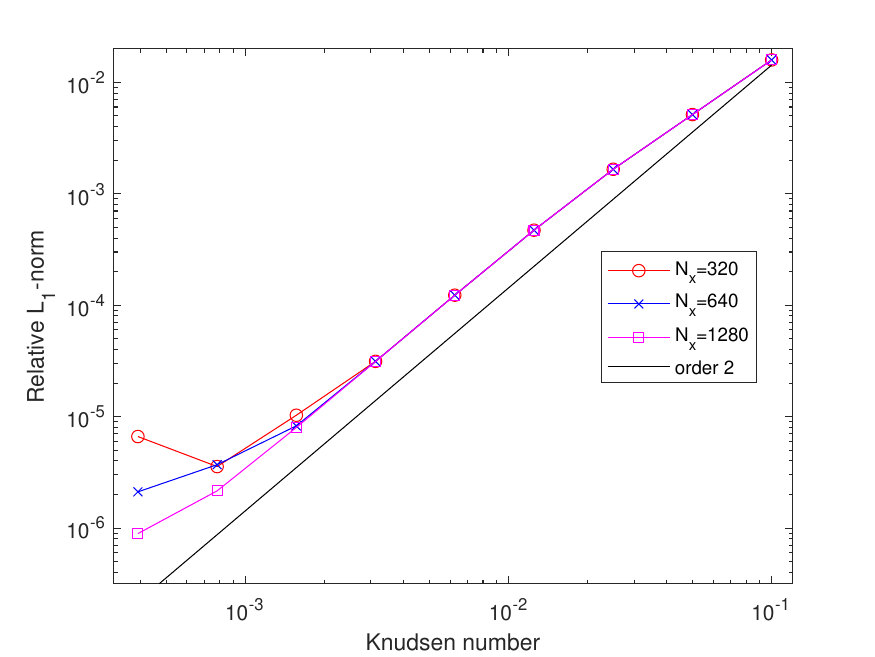}
\caption{{Knudsen number($\varepsilon$) vs relative $L_1$-norm. For the comparison, we used the same initial data for test 2, and compute numerical solutions up to $t=0.5$.}}
\label{fig ESNSE}
\end{figure}

{\em Test 3. Riemann Problem.}
In this problem, we consider a Riemann problem in 1D space and 2D velocity domain. The same test has been adopted in \cite{JinFilbet} to show the consistency between the ES-BGK model and BTE or NSE. As initial macroscopic states, we use
\[
\left(\rho_0, u_{x0}, u_{y0}, T_0 \right)
= \begin{cases}
(1, M\sqrt{2}, 0,1),    & -1 \le x \le 0.5,\\[2mm]
(\frac{1}{8},0,0,\frac{1}{4}), & \hspace{4mm}\textrm{otherwise},
\end{cases}
\]
where $M = 2.5$ is the Mach number. Here we impose the free-flow boundary condition in space $x \in [-1, 2]$. We truncate the velocity domain with $v_{max} = 15$.
We compare numerical solutions at time $t=0.15,\,0.2,\,0.4$ 
%The final time is $t = 0.4$. 
with the time step that corresponds to $CFL = 0.5$. Here we set $\nu = -1$ and $\tau = 0.9\pi \rho/2$ following \cite{JinFilbet}, which gives the following viscosity and heat conductivity:
$$
\mu = \frac{1}{0.9 \pi}T, \quad \kappa = \frac{1}{0.9\pi}T.
$$
We consider these transport coefficients when computing the reference {solutions} to NSE. We remark that this choice matches the viscosity of ES-BGK model and the one derived from BTE for Maxwellian molecules \cite{JinFilbet}.

In Figs. \ref{fig:riemann 1}-\ref{fig:riemann 2}, we present numerical results for BTE, NSE and ES-BGK model for different Knudsen numbers $\varepsilon = 0.5, 0.1$ respectively. For BTE, we use an explicit fourth-order semi-Lagrangian scheme \cite{Boscarino:2024}, while for NSE, we use the classical RK4 and WENO23. For ES-BGK, we use the combination of IMEX-II-ISA3 for the time discretization and WENO23 for the space one. For $\varepsilon=0.5$, the numerical solution of the ES-BGK model is very close to that of BTE, while solutions to NSE are deviated from the other solutions. On the other hand, for $\varepsilon=0.1$, the discrepancy between three solutions become smaller, which confirms the consistency of the three models.
\begin{figure}[!htbp]
\centering
\includegraphics[width=0.46\textwidth]{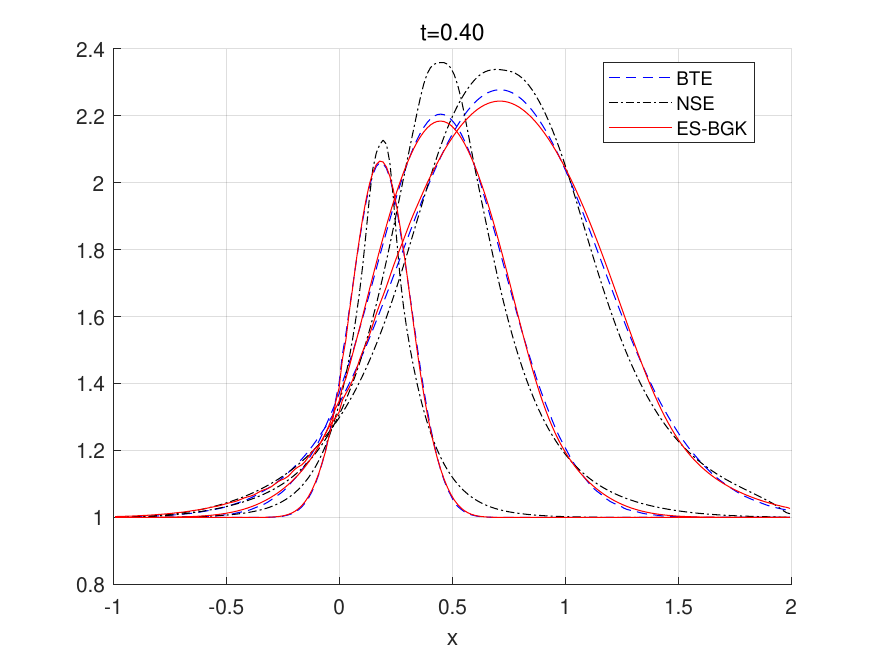}
\includegraphics[width=0.46\textwidth]{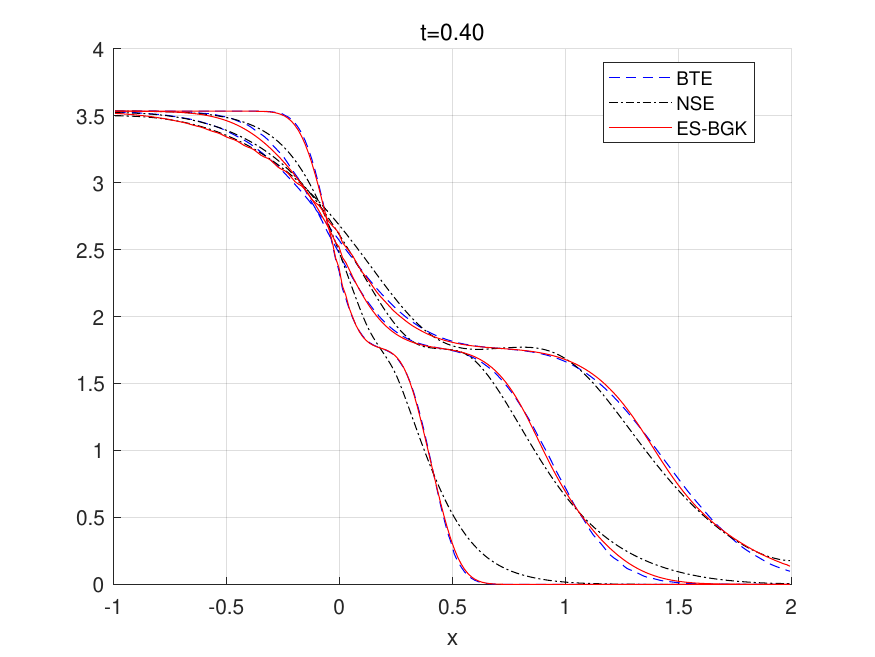}
\includegraphics[width=0.46\textwidth]{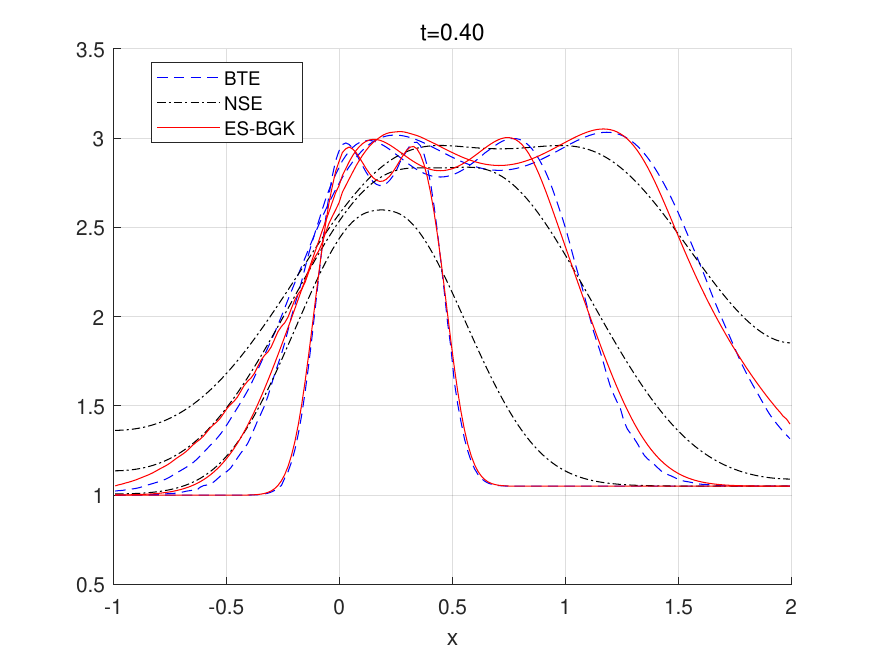}
\includegraphics[width=0.46\textwidth]{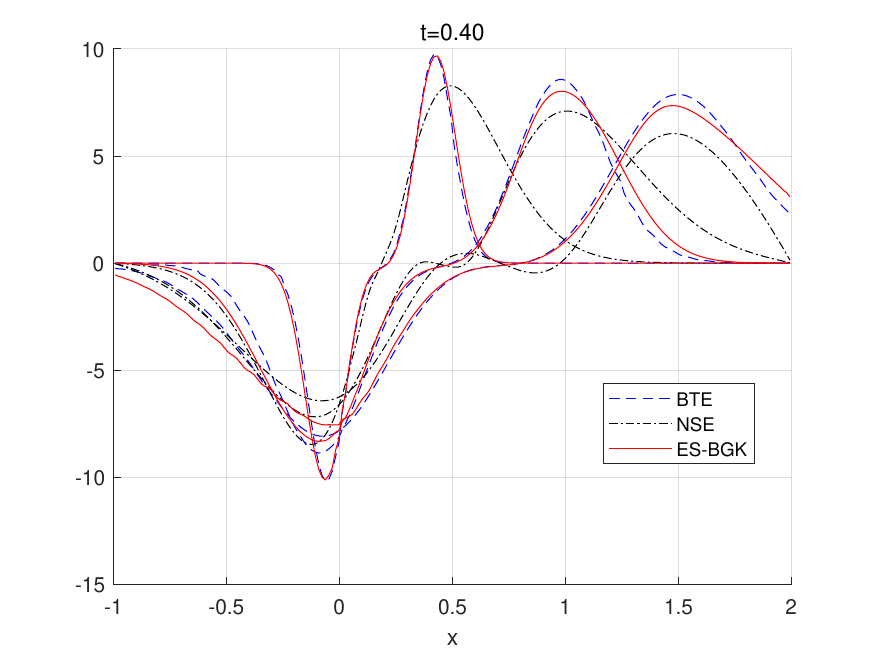}
\caption{Comparison between reference solutions of BTE and NSEs with the numerical solutions for ES-BGK model. We report the case of density, velocity, temperature and heat flux at time $t=0.1,\,0.25,\,0.4$. The {Knudsen} number is $\varepsilon=0.5$.}
\label{fig:riemann 1}
\end{figure}
\begin{figure}[!htbp]
\centering
\includegraphics[width=0.46\textwidth]{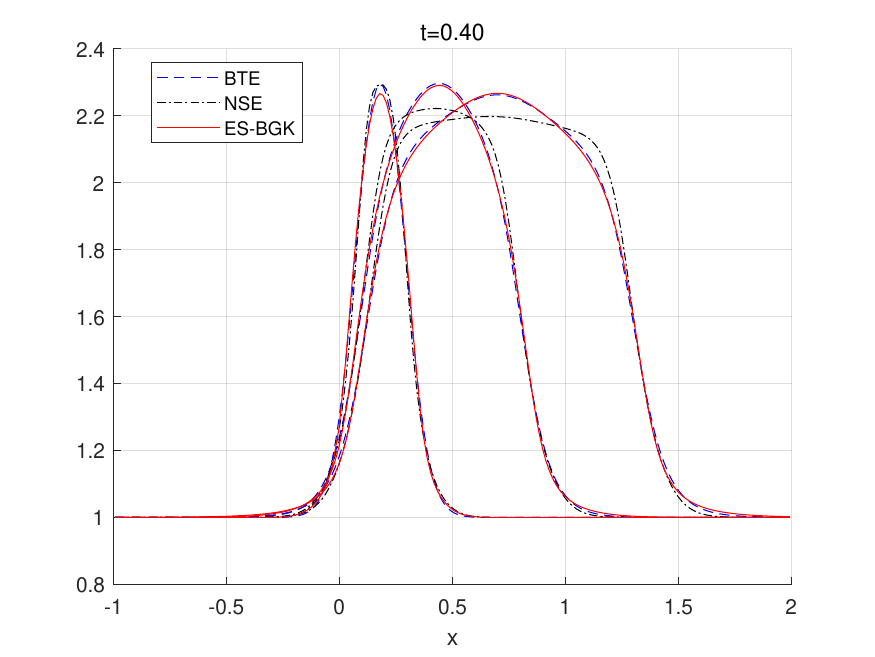}
\includegraphics[width=0.46\textwidth]{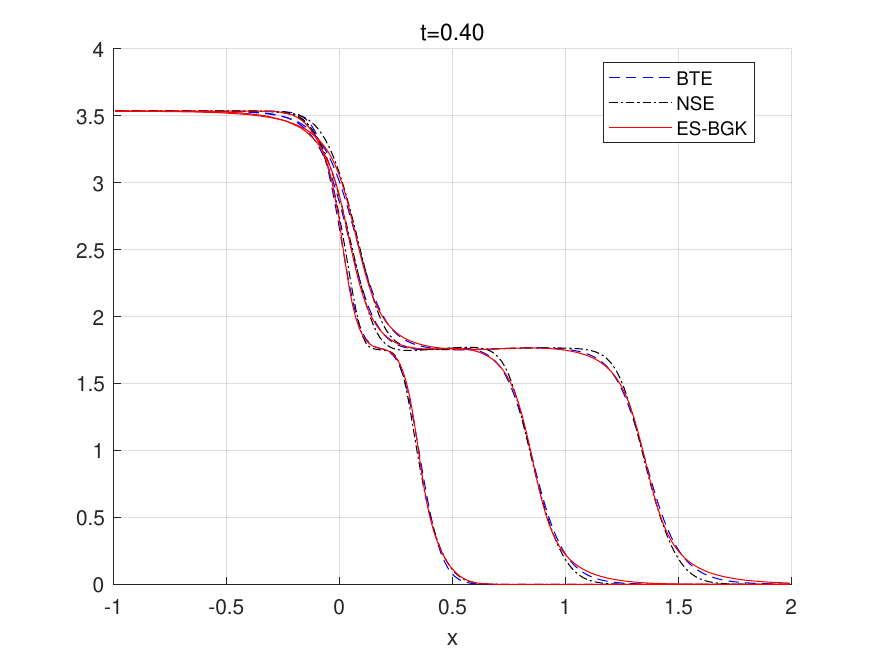}
\includegraphics[width=0.46\textwidth]{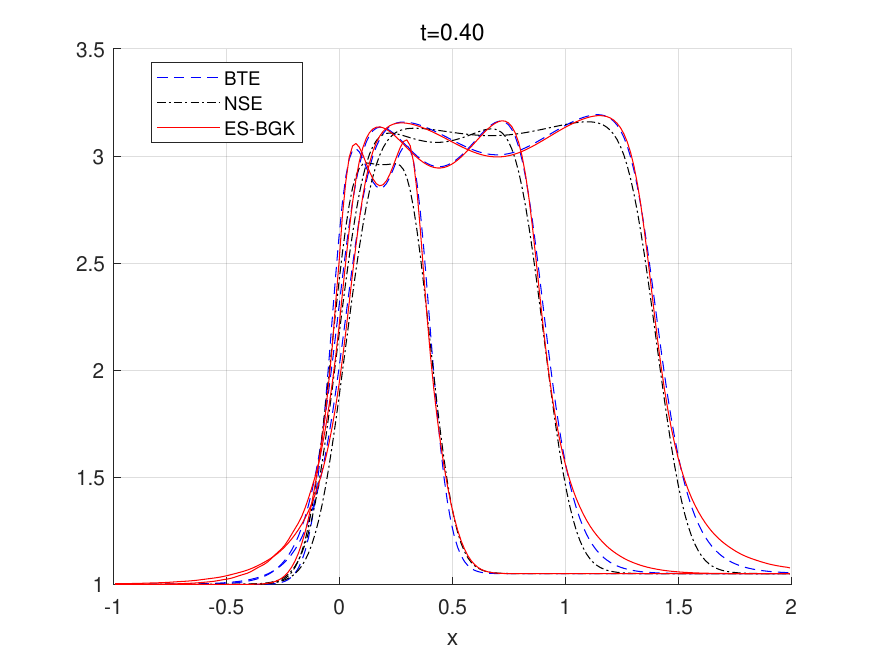}
\includegraphics[width=0.46\textwidth]{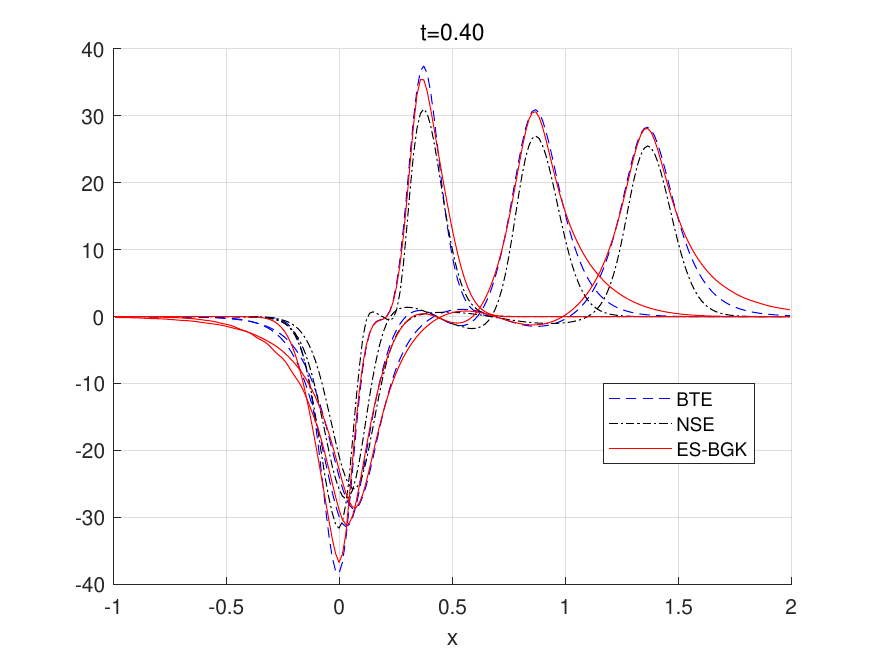}
\caption{Comparison between reference solutions of BTE and NSEs with the numerical solutions for ES-BGK model. We report the case of density, velocity, temperature and heat flux at time $t=0.1,\,0.25,\,0.4$. The Knudsern number is $\varepsilon=0.1$.}
\label{fig:riemann 2}
\end{figure}

{\em Test 4. The Lax Shock Tube Problem.}
Finally we test IMEX-RK schemes solving the ES-BGK model for the Lax shock tube problem, \cite{Hu}.  We set the collision frequency $\tau$ for ES-BGK model as 
$$
\tau = \frac{2}{3}\rho \sqrt{T},
$$
which yields the viscosity and heat conductivity:
$$
\mu = \sqrt{T}, \quad \kappa = \frac{15}{4}\sqrt{T},
$$
where $\nu = -1/2$, and this implies that the Prandtl number is $2/3$. We consider the initial macroscopic variables 
\[
\left(\begin{array}{c}
\rho \\
u \\
p
\end{array}
\right) = \Biggl\{
\begin{array}{cc}
(0.445, 0.698, 3.528)^\top,    & -0.5 \le x \le 0,\\[2mm]
(0.5,0,0.571)^\top, & 0 < x \le 0.5,
\end{array}
\]
on the free-flow condition $x \in [-5,5]$. We take well-prepared initial conditions
$$
f_0(x,v) =\mathcal{M}(x,v) - \frac{\varepsilon}{\tau}\left(I-\Pi_\mathcal{M}\right)(v_1\partial_x \mathcal{M}),
$$
where $I$ is the identity operator and $\Pi_\mathcal{M}$ is the projection operator defined as \eqref{PiMf}.
We first use $80 \times 80$ uniform points for the velocity domain $[-20, 20] \times [-20, 20]$, and $N_x = 200$ uniform points for the spatial discretization. The CFL number is taken as $0.2$ and final time $1.3$. We consider IMEX-II-ISA3 scheme coupled with WENO35 for solving the ES-BGK model.
\begin{figure}[!htbp]
\centering
\includegraphics[width=0.46\textwidth]{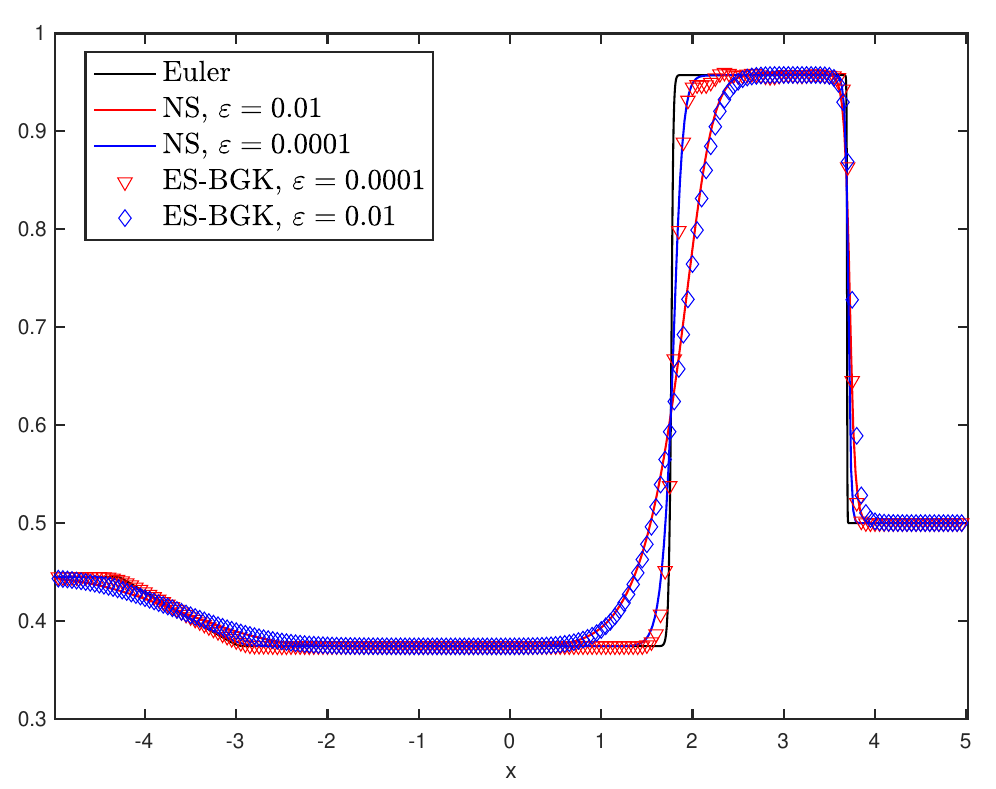}
\includegraphics[width=0.46\textwidth]{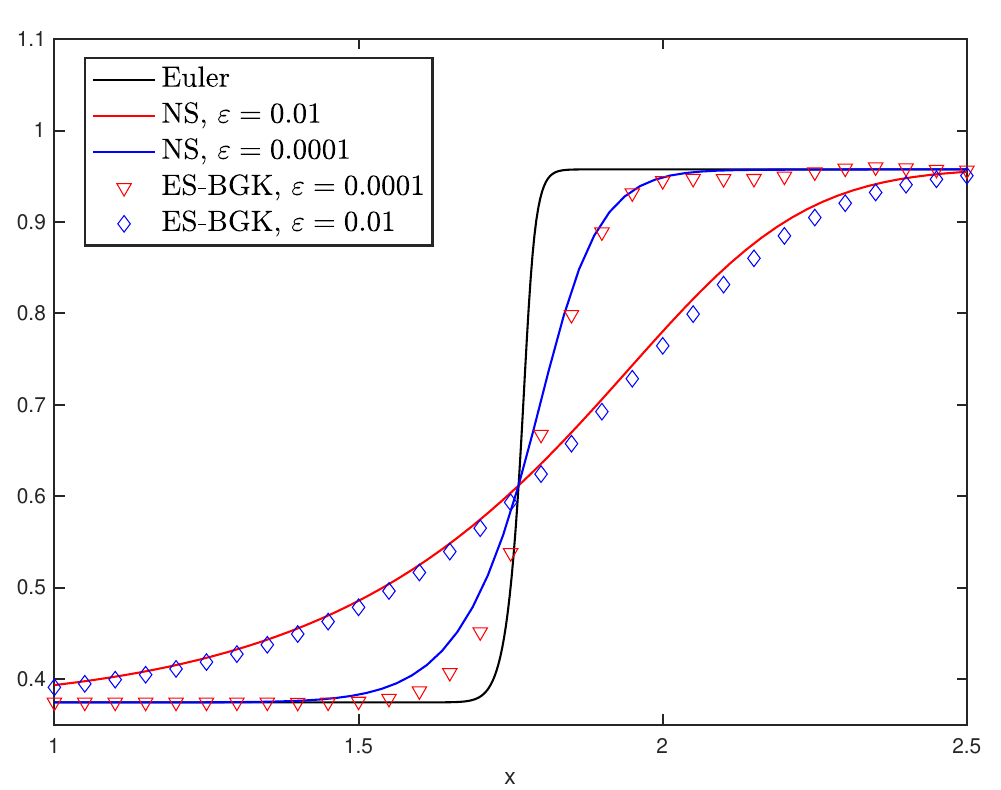}
\caption{Comparison between reference solutions of Navier-Stokes equations and the numerical solutions of IMEX-II-GSA3 (left) and a zoom (right). Uniform $40 \times 40$ points for the velocity domain $[-20, 20] \times [-20, 20]$}
\label{fig:Density}
\end{figure}
In Fig. \ref{fig:Density}, we compare the numerical solution of the ES-BGK model with those of the compressible Navier-Stokes equations and the compressible Euler equations. The reference solution of the Euler equations and of the Navier-Stokes equations was generated by using spectral method for the {spatial} derivatives and fourth-order RK method for the time on a grid of 3200 for $\varepsilon = 0, 10^{-2}, 10^{-4}$ respectively. As we can see from the Fig. \ref{fig:Density} the numerical solution of the ES-BGK model is very close to that of the NS equations. Furthermore for  a better visualization, we zoomed the two solutions on a portion of the domain.
\begin{figure}[htbp]
\centering
\includegraphics[width=0.49\textwidth]{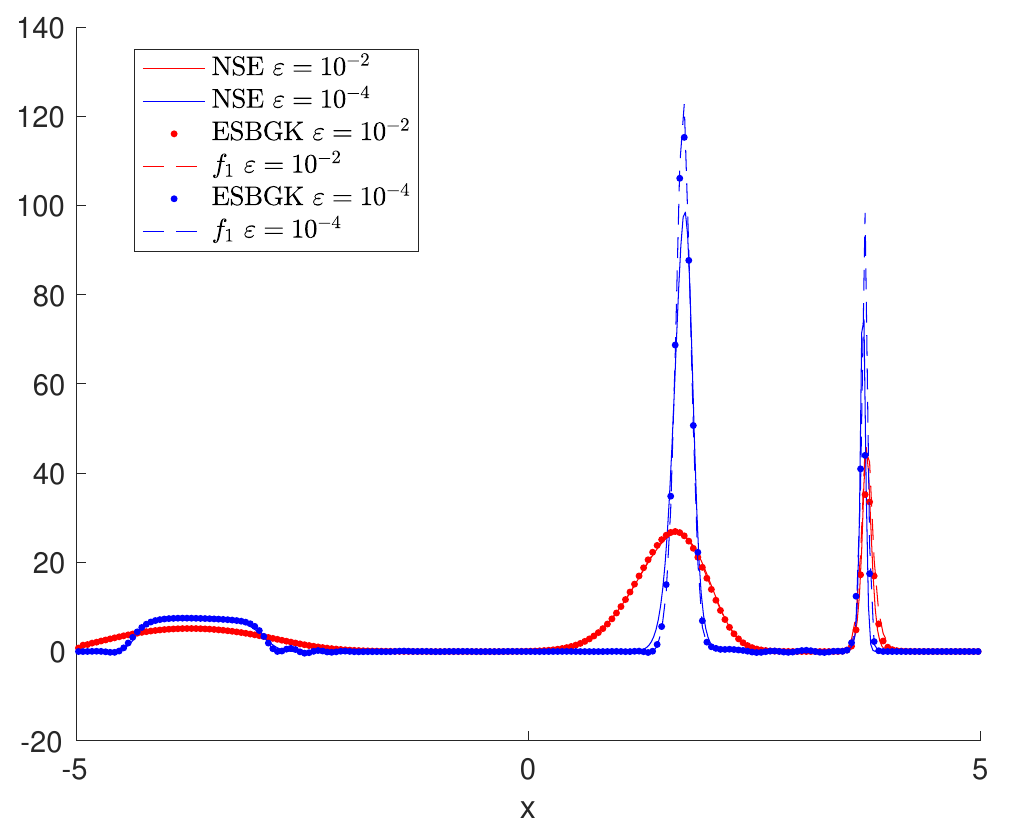}
\includegraphics[width=0.49\textwidth]{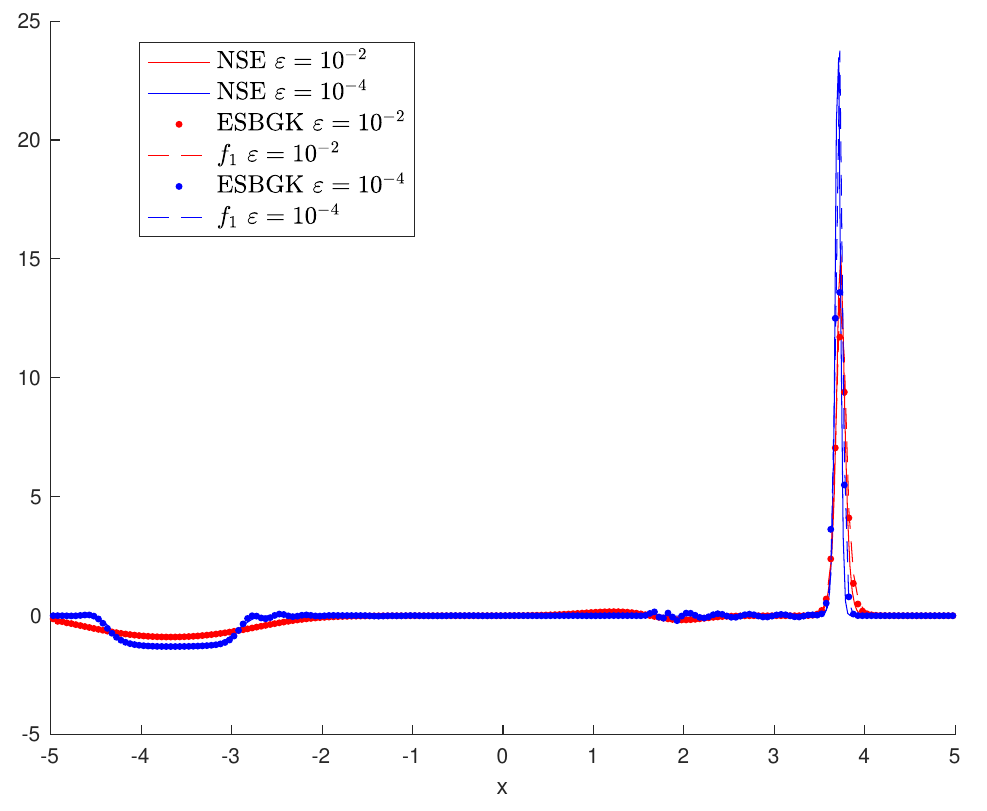}
\caption{The shear stress and heat flux for $\varepsilon = 10^{-2}, 10^{-4}$. Uniform $40 \times 40$ points for the velocity domain $[-20, 20] \times [-20, 20].$}
\label{fig:other}
\end{figure}
In Fig. \ref{fig:other} we reported the moments 
\begin{align}\label{1_A}
\begin{split}
	-\mu \sigma(u) &=  \int_{\R^{d_v}} f_1 (v-u) \otimes  (v-u)dv = \avg{f_1(v -u)\otimes(v-u)},\\
	-\kappa \nabla T &= \int_{\R^{d_v}} f_1 \frac{1}{2}(v-u) |v-u|^2 dv =  \avg{f_1(v -u)|v-u|^2},
\end{split}
\end{align}
i.e., the shear stress and heat flux, computed (1) from the function $f_1 = \frac{\mathcal{G}-f}{\varepsilon}$ (2) from the macroscopic quantities and 3) from the NS equation for the case of $\varepsilon = 10^{-2}$ and $\varepsilon = 10^{-4}$.

\section{Conclusions}
In this work, we investigated 
the asymptotic behavior at the Navier-Stokes level, reproducing theoretical results, Theorems \ref{thm 1} and Theorem \ref{thm 2}. We showed that existing IMEX-RK schemes including type I and II can accurately capture the NS limit without resolving the small scales determined by the Knudsen number. To address the issues on order reduction of these schemes at the Navier-Stokes level, we propose IMEX-RK schemes of type II developed in \cite{Boscarino2017}. In particular, the IMEX-II-ISA3 scheme of type ARS satisfying additional conditions \eqref{22CK} guarantees uniform accuracy avoiding order reduction. Finally, we provided numerical examples that validate the theoretical results.

\section*{Acknowledgements}
Sebastiano Boscarino is supported for this work by 1) the Spoke 1 "FutureHPC $\&$ BigData" of the Italian Research Center on High-Performance Computing, Big Data and Quantum Computing (ICSC) funded by MUR Missione 4 Componente 2 Investimento 1.4: Potenziamento strutture di ricerca e creazione di "campioni nazionali di R$\&$S (M4C2-19 )", (CN00000013);
by 2) the Italian Ministry of Instruction, University and Research (MIUR) to support this research with funds coming from PRIN Project 2022  (2022KA3JBA), entitled "Advanced numerical methods for time dependent parametric partial differential equations and applications"; 3) from Italian Ministerial grant PRIN 2022 PNRR "FIN4GEO: Forward and Inverse Numerical Modeling of hydrothermal systems in volcanic regions with application to geothermal energy exploitation.", (No. P2022BNB97). S. Boscarino is a member of the INdAM Research group GNCS. S. Y. Cho was supported by the National Research Foundation of Korea (NRF) grant funded by the Korea government (MSIT) (No. RS-2022-00166144), and Learning \& Academic research institution for Master's$\cdot$PhD students, and Postdocs (LAMP) Program of the National Research Foundation of Korea (NRF) grant funded by the Ministry of Education (No. RS-2023-00301974).

\appendix
\section{Appendix}\label{App1_bis}
\subsection{Basic properties of Gaussian distribution function}
By definition, $\mathcal{G}[f]$ satisfies
\begin{align}\label{G__U}
\begin{split}
	&\biggl\langle f \biggr\rangle = \int_{\mathbb{R}^{d_v}} f(v) dv = \int_{\mathbb{R}^{d_v}} \mathcal{G}[f](v) dv  = \biggl\langle \mathcal{G}[f] \biggr\rangle  = \rho,\\
	&\biggl\langle v f \biggr\rangle = \int_{\mathbb{R}^{d_v}} vf(v) dv = \int_{\mathbb{R}^{d_v}} v\mathcal{G}[f](v) dv  = \biggl\langle v \mathcal{G}[f](v)\biggr\rangle = \rho u,\\
	&\biggl\langle \frac{|v|^2}{2} f\biggr\rangle = \int_{\mathbb{R}^{d_v}}\frac{|v^2|}{2} f(v) dv = \int_{\mathbb{R}^{d_v}} \frac{|v^2|}{2}\mathcal{G}[f](v) dv=  \biggl\langle \frac{|v|^2}{2} \mathcal{G}[f](v)\biggr\rangle = E.
\end{split}
\end{align}
{Note that from
\begin{align}\label{G__Um}
	\begin{split}
		&\int_{\mathbb{R}^{d_v}} (v-u) \otimes (v-u) f(v) dv   = \rho \Theta,\\
		& \int_{\mathbb{R}^{d_v}} (v-u) \otimes (v-u) \mathcal{G}[f](v) dv  = \rho \mathcal{T},
	\end{split}
\end{align}
we further have
\begin{align}\label{Prop1}
	\begin{split}
		f &= \mathcal{G}[f] \Longleftrightarrow f = \mathcal{M}[f];\\[3mm]
		f &= \mathcal{G}[f] + \mathcal{O}(\varepsilon) \quad \textrm{implies}  \quad  \mathcal{G}[f] = \mathcal{M}[f] + \mathcal{O}(\varepsilon).
	\end{split}
\end{align}
The proof is straightforward. In fact, on both side of $f = \mathcal{G}[f]$ using (\ref{G__Um})  we get $\Theta =  \mathcal{T}$ and from (\ref{TTT}) we get  $\mathcal{T} =  TId$ hence $\mathcal{G}[f]$ is just the isotropic Maxwellian $\mathcal{M}[f]$. Similarly if on both side of $f = \mathcal{M}[f]$ takes (\ref{G__Um}). For the second point in (\ref{Prop1}) the proof is similar and we omit the details.
}

\begin{lem}\label{lem discre}
Let $U_f=(\rho_f,\rho_fu_f,E_f)$ and $U_g=(\rho_g,\rho_fu_g,E_g)$ be macroscopic variables associated to $f$ and $g$, respectively. Then, 
for a positive constant $\delta$ the assumption
\[
\|U_f-U_g\|_\infty= \mathcal{O}(\delta)
\]
implies that
\[
|\mathcal{M}[f]-\mathcal{M}[g]|= \mathcal{O}(\delta).
\]
\end{lem}
\begin{proof}
The differential form of $\mathcal{M}$ satisfies
\[
d\mathcal{M}= \mathcal{M}\left[ \frac{1}{\rho}d\rho + \frac{(v-u)}{T}\cdot du + \left(\frac{|v-u|^2}{2T^2}- \frac{d_v}{2T}\right)dT\right].
\]
This further implies that
\[
\mathcal{M}[f]-\mathcal{M}[g]= \mathcal{M}[g]\left[ \frac{1}{\rho_g}(\rho_f-\rho_g) + \frac{(v-u_g)}{T_g}\cdot (u_f-u_g) + \left(\frac{|v-u_g|^2}{2T_g^2}- \frac{d_v}{2T_g}\right)(T_f-T_g)\right] + \mathcal{O}(\delta^2).
\]
This gives the desired estimate.
\end{proof}

\subsection{{Chapman-Enskog} expansion of ES-BGK model}
Here we perform the first order Chapman-Enskog expansion of ES-BGK model with respect to $\varepsilon$:
\begin{equation}\label{firstEx}
f = \mathcal{M}(f) + \varepsilon f_1,
\end{equation}
which implies that $f_1$ satisfies the so-called compatibility relations:
\begin{equation}\label{Comp}
\langle \phi  f_1  \rangle=0, \quad \phi = (1,v,|v|^2).
\end{equation}
Now we take the expansion in terms of $\varepsilon$ for the stress tensor 
\begin{equation}\label{ST}
\Theta = T Id + \varepsilon \Theta_1,
\end{equation} 
and the heat flux
\begin{equation}\label{HF}
\mathbb{Q} := \left\langle \frac{|v-u|^2}{2}(v-u)f \right\rangle = 0 + \varepsilon \mathbb{Q}_1
\end{equation}
with 
\begin{align}\label{1_A}
\Theta_1 =  \frac{1}{\rho }\int_{\R^{d_v}} f_1 (v-u) \otimes  (v-u)dv = \frac{1}{\rho } \avg{f_1(v -u)\otimes(v-u)},
\end{align}
and 
\begin{equation}\label{3_A}
\mathbb{Q}_1 = \int_{\R^{d_v}} f_1 \frac{1}{2}(v-u) |v-u|^2 dv =  \avg{f_1\frac{1}{2}(v -u)|v-u|^2}.
\end{equation}
Inserting (\ref{firstEx}) and these latter expansions for  $\Theta$ and $\mathbb{Q}$ into the conservation laws 
$$
\langle \phi f \rangle +  \langle \phi v \cdot \nabla_x f \rangle = 0, 
$$
or 
$$
\langle \phi \mathcal{M}(f) \rangle +  \langle \phi v \cdot \nabla_x \mathcal{M}(f) \rangle = - \varepsilon  \nabla_x \cdot \langle \phi v  f_1 \rangle, 
$$
it gives the equations:
\begin{equation}\label{DiscCL_bis}
\partial_t U +  \nabla_x F(U) = -\varepsilon  \nabla_x \cdot H(\bU) 
\end{equation}
with
\begin{equation}\label{UFH}
U =  \langle \phi f \rangle  = \left(
\begin{array}{c}
	\rho\\
	\rho u\\
	E
\end{array}
\right),
\quad
F(U) = \left(
\begin{array}{c}
	\rho u\\
	\rho u \otimes u + \rho T Id \\
	(E + \rho T)u
\end{array}
\right),
\quad
H(\bU) = \left(
\begin{array}{c}
	0\\
	\rho \Theta_1\\
	\mathbb{Q}_1 + \rho \Theta_1 u
\end{array}
\right).
\end{equation}
%%%%%%%%%%%%%%%%%%%%%%%%%%%
{To evaluate the quantities $\rho \Theta_1$ and $\mathbb{Q}_1$, we seek the form of $f_1$. For this aim,}
we consider the expansion of the  anisotropic Gaussian $\mathcal{G}(f)$ with respect to $\varepsilon$, i.e. 
\begin{equation}\label{G_ex}
\mathcal{G}(f) = \mathcal{M}(f) + \varepsilon g.
\end{equation}
By (\ref{TTT}) and considering the expansion (\ref{ST}), we get $\mathcal{T} = T Id + \nu \varepsilon \Theta_1$, and using the fact that $tr(\Theta_1) = 0$,  this gives $det(\mathcal{T}) = T^{d_v} + \mathcal{O}(\varepsilon^2)$ and 
$$
[\mathcal{T}]^{-1} = \frac{1}{T}\left(Id - \frac{\nu \varepsilon}{T}\Theta_1\right) + \mathcal{O}(\varepsilon^2),
$$
and therefore by (\ref{Gauss}) we obtain
\begin{equation}\label{gVect}
g := \frac{\mathcal{G}(f) - \mathcal{M}(f)}{\varepsilon} = \nu \frac{ \mathcal{M}(f)}{2 T^2}(v-u)^\top\Theta_1(v-u).
\end{equation}
where we neglect the term $\mathcal{O}(\varepsilon^2)$.
Now the next step is to insert expansions (\ref{firstEx}) and (\ref{G_ex}) into the equation 
$$
\partial_t f + v \cdot \nabla_x f = \frac{\tau}{\varepsilon}(\mathcal{G}[f]-f).
$$
Then, we get
\begin{equation}\label{1}
\tau (f_1 - g) = - \left( 
\partial_t \mathcal{M} + v\cdot \nabla_x \mathcal{M} +\varepsilon\left(\partial_t {f}_1 + v\cdot \nabla_x 
{f}_1 \right)\right).
\end{equation}

Before proceeding, we review a Hilbert space and an associated operator considered in \cite{Mieussens} and some Lemmas. 
\begin{rem}\label{B1}
Given a Maxwellian function $\mathcal{M}(f)$,  $\Pi_{\mathcal{M}}(f)$ is the orthogonal projection 
in the Hilbert space $L_\mathcal{M}^2 =\{ \phi\textrm{ such that } \,\phi\mathcal{M}^{-1/2} \in L_2(\mathbb{R}^{d_v})\}$ onto $$\mathcal{N} = Span \{\mathcal{M},\,v\mathcal{M},\,|v|^2\mathcal{M}\},$$ equipped with the following weighted inner product $$
\biggl\langle \phi\psi \biggr\rangle_\mathcal{M} = \biggl\langle \phi\psi \mathcal{M}^{-1}\biggr\rangle = \int_{\mathbb{R}^{d_v}}  \phi\psi \mathcal{M}^{-1}dv.
$$
Then, by using the orthogonal basis of $\mathcal{N}$ 
\begin{equation}\label{Bases}
	\mathcal{B} = \left\{ \frac{1}{\rho}\mathcal{M},\, \frac{(v-u)}{\sqrt{T}}\frac{1}{\rho}\mathcal{M},\, \left( \frac{|v-u|^2}{2T}- \frac{d_v}{2}\right)\frac{1}{\rho}\mathcal{M}\right\},
\end{equation}
its explicit form is written as 
\begin{align}\label{PiMf}
	\Pi_\mathcal{M}(f) = \frac{1}{\rho}\biggl[ \langle f \rangle + \frac{(v-u) \cdot \langle (v-u)f \rangle }{T} + \left( \frac{|v-u^2|}{2T}- \frac{d_v}{2}\right)\frac{2}{d_v} \biggl\langle \left( \frac{|v-u^2|}{2T}- \frac{d}{2}\right)f \biggr\rangle \biggr]\mathcal{M}.
\end{align}
Note that a direct calculation gives $\Pi_{\mathcal{M}}(\partial_t \mathcal{M})=0$ and $\Pi_{\mathcal{M}}(f_1-g)=f_1-g$. Then, from \eqref{1} we have
\begin{align}\label{2}
	\begin{split}
		\tau(f_1-g)&=(I-\Pi_{\mathcal{M}})(\tau(f_1-g))\cr
		&=-(I-\Pi_{\mathcal{M}})(\partial_t \mathcal{M}+ v \cdot \nabla_x \mathcal{M})+\mathcal{O}(\varepsilon)\cr
		&=-(I-\Pi_{\mathcal{M}})(v \cdot \nabla_x \mathcal{M})+\mathcal{O}(\varepsilon).
	\end{split}
\end{align}
%Moreover, by \cite{Mieussens, Golse, TQ} it follows that
Moreover, by \cite{Golse,JinFilbet} it follows that
\begin{align}\label{5_A_bis}
	(I-\Pi_{\mathcal{M}})(v \cdot \nabla_x \mathcal{M})=  \mathcal{M}\left(A(V):\frac{\sigma(u)}{2} + 2B(V)\cdot \nabla_x \sqrt{T}\right) +\mathcal{O}(\varepsilon),
\end{align}
where
\begin{align}\label{VAB def}
	V = \frac{v-u}{\sqrt{T}}, \quad A(V) = V \otimes V - \frac{1}{d_v}|V|^2 Id, \quad B(V) = \frac{1}{2}V(|V|^2 - (d_v+2)),
\end{align}
and 
$$
\sigma(u) =  \nabla_x u + (\nabla_x u)^\top - \frac{2}{d_v}\nabla_x \cdot u Id.
$$
%\begin{equation}\label{I-Pi}
%(I-\Pi_{\mathcal{M}})(v \cdot \nabla_x \mathcal{M}) = %(I-\Pi_{\mathcal{M}})(\tau(g - f_1)) + \mathcal{O}%(\varepsilon) =\tau(g - f_1) + \mathcal{O}(\varepsilon).
%\end{equation}
To sum up, the relations \eqref{1}, \eqref{5_A_bis} and \eqref{2} imply that 
\begin{equation}\label{M_lim}
	\partial_t \mathcal{M} + v\cdot \nabla_x \mathcal{M} = 
	\mathcal{M}\left(A(V):\frac{\sigma(u)}{2} + 2B(V)\cdot \nabla_x \sqrt{T}\right) +\mathcal{O}(\varepsilon).
\end{equation}
\end{rem}

\begin{lem}\label{Lemma1}
For $\phi = (1, v, |v|^2/2)^\top$, it follows 
\begin{equation}\label{Hu}
	\avg{v\phi f_1} = H(U), \quad \textrm{where} \quad H(U) = (0, \rho \Theta_1, \mathbb{Q}_1 + \rho \Theta_1 u)^\top,
\end{equation}
with $U = (\rho, \rho u, E)^\top$.
\end{lem}
%%%%
\begin{lem}\label{Lemma2}
For $\phi = (1, v, |v|^2/2)^\top$, it follows 
\begin{equation}\label{Gu}
	\avg{v\phi g} = {G}(U), \quad \textrm{where} \quad {G}(U)  = (0, \nu \rho\Theta_1,\nu \rho \Theta_1 u)^\top,
\end{equation}
with $U = (\rho, \rho u, E)^\top$.
\end{lem}
\begin{lem}\label{Lemma3}
For $\phi = (1, v, |v|^2/2)^\top$, it follows 
\begin{equation}\label{Fu}
	\avg{v \phi (I-\Pi_{\mathcal{M}})(v \cdot \nabla \mathcal{M})} =\avg{v\phi \mathcal{M}\left(A(V):\frac{\sigma(u)}{2} + 2B(V)\cdot \nabla_x \sqrt{T}\right)} = \mathcal{F}(U)  
\end{equation}
where 
$$
\mathcal{F}(U) = \left(0, \rho T\sigma(u), \rho T\sigma(u)u + \frac{d_v + 2}{2}\rho T\nabla T\right)^\top,
$$
with $U = (\rho, \rho u, E)^\top$.
\end{lem}
%%%%%%%%%%%%%%%%

%\SB{From \cite{Golse, Mieussens}, we further have
%\begin{equation}\label{M_lim}
%\partial_t  \mathcal{M}  + v \cdot \nabla_x    \mathcal{M} = \mathcal{M}\left(A(V):\frac{\sigma(u)}{2} + 2B(V)\cdot \nabla_x \sqrt{T}\right) + \mathcal{O}(\varepsilon),
%\end{equation}
%with the notation
%$$
%V = \frac{v-u}{\sqrt{T}}, \quad A(V) = V \otimes V - \frac{1}{d_v}|V|^2 Id, \quad B(V) = \frac{1}{2}V(|V|^2 - (d_v+2)),
%$$
%and where 
%$$
%\sigma(u) =  \nabla_x u + (\nabla_x u)^\top - \frac{2}%{d_v}\nabla_x \cdot u Id.
%$$
%}
By Remark \ref{B1}, we can rewrite (\ref{1}) as
\begin{equation}\label{f11}
f_1=  g-\frac{1}{\tau}\left(\mathcal{M}\left(A(V):\frac{\sigma(u)}{2} + 2B(V)\cdot \nabla_x \sqrt{T}\right)\right)  + \mathcal{O}(\varepsilon).
\end{equation}

{Next, we multiply both sides of (\ref{f11}) by $v \phi$} and take integration and use Lemmas \ref{Lemma1}-\ref{Lemma3} to obtain
\begin{align}\label{f11_bis}
\begin{split}
	\avg{v\phi f_1}&=  \avg{v\phi g}- \frac{1}{\tau}\avg{v\phi \mathcal{M}\left(A(V):\sigma(u) + 2B(V)\cdot \nabla_x \sqrt{T}\right)}  + \mathcal{O}(\varepsilon),\\
	&H(U) = {G}(U) - \frac{1}{\tau}\mathcal{F}(U) + \mathcal{O}(\varepsilon),
\end{split}
\end{align}
that is,
%%%%%%%%%%%%%%%%%%%
\begin{equation}\label{f11_tris}
\left(
\begin{array}{c}
	0\\
	\rho \Theta_1\\
	\mathbb{Q}_1 + \rho \Theta_1 u
\end{array}
\right)=
\left(
\begin{array}{c}
	0\\
	\nu \rho \Theta_1\\
	\nu \rho \Theta_1 u
\end{array}
\right) - \frac{1}{\tau}
\left(
\begin{array}{c}
	0\\
	\rho T\sigma(u)\\
	\rho T\sigma(u)u + \frac{d_v +2}{2}\rho T \nabla_x T
\end{array}
\right) + \mathcal{O}(\varepsilon).
\end{equation}
Then, we get
$$
\rho \Theta_1 =- \mu\sigma(u) + \mathcal{O}(\varepsilon), 
$$
with viscosity $\mu = \frac{\rho T}{(1-\nu)\tau}$
and 
$$
\mathbb{Q}_1 =- \kappa \nabla_x T + \mathcal{O}(\varepsilon),
$$
with the thermal conductivity
$$
\kappa = \frac{d_v + 2}{2}\frac{p}{\tau},
$$
where $p = \rho T$. Finally, from (\ref{DiscCL_bis}) we derive the CNS equations
\begin{equation}\label{Mom2_bis}
\partial_t U + \nabla_x \cdot F(U) = \varepsilon \nabla_x \cdot S(U),
\end{equation}
with
\begin{equation}
S(U)=\left(
\begin{array}{c}
	0\\
	\mu \sigma(u)\\
	\mu\sigma(u)u -\kappa \nabla_x T
\end{array}
\right),
\end{equation}
$$
\sigma(u) =  \nabla_x u + (\nabla_x u)^   \top - \frac{2}{d}\nabla_x \cdot u I.
$$

\bibliographystyle{plain}
\bibliography{references}

\end{document}